\documentclass[11pt,reqno]{amsart}
\usepackage[a4paper,portrait,left=3.5cm,right=3.5cm,top=3.5cm,bottom=3.5cm]{geometry}

\usepackage{bm}
\usepackage{mathrsfs}
\usepackage{amsmath, amsthm, amsfonts, amssymb, amsbsy, bigstrut, graphicx, enumerate, upref, longtable, comment, booktabs, array, caption, subcaption}
\usepackage{newtxtext,newtxmath}
\usepackage{verbatim}
\usepackage{cleveref}
\usepackage{graphicx}
\usepackage[latin1]{inputenc}
\usepackage{latexsym}
\usepackage{lscape}
\usepackage{epsfig}
\usepackage{dsfont}
\usepackage{todonotes}
\usepackage[foot]{amsaddr}
\usepackage{setspace} 
\usepackage{enumerate}
\usepackage{verbatim}
\usepackage{dsfont}
\usepackage{bm}
\usepackage{color}
\usepackage{mathrsfs}
\usepackage{mdframed}
\usepackage{csquotes}
\usepackage{stmaryrd}
\usepackage{cases}
\usepackage{tikz}

\newtheorem{theorem}{Theorem}
\newtheorem{lemma}{Lemma}[section]
\newtheorem{definition}[lemma]{Definition}

\newtheorem{proposition}[lemma]{Proposition}
\theoremstyle{definition}

\numberwithin{equation}{section}

\usepackage[ruled,vlined]{algorithm2e}
    
    \SetCommentSty{mycommfont}
    
\usepackage{etoolbox}
\makeatletter
\patchcmd{\@algocf@start}
  {-1.5em}
  {0pt}
  {}{}
\makeatother

\newcommand{\N}{\mathbb{N}}

\newcommand{\Z}{\mathbb{Z}}
\renewcommand{\P}{\mathbb{P}}
\newcommand{\E}{\mathbb{E}}
\newcommand{\1}{\mathbf{1}}
\newcommand{\A}{\mathcal{A}}

\newcommand{\eps}{\epsilon}

\renewcommand{\H}{\mathbb{H}}
\DeclareMathOperator{\rad}{rad}
\newcommand{\bdy}{\partial}
\DeclareMathOperator{\dist}{dist}
\DeclareMathOperator{\PP}{\mathbf{P}}
\newcommand{\EE}{\mathbf{E}}

\DeclareMathOperator{\diam}{diam}

\newcommand{\iso}{\mathrm{Iso}}
\newcommand{\noniso}{\mathrm{NonIso}}
\newcommand{\wtiso}{\mathrm{\wt{I}so}}
\newcommand{\wtnoniso}{\mathrm{\wt{N}onIso}}

\newcommand{\wt}{\widehat}
\renewcommand{\A}{\mathcal{A}}
\newcommand{\bdye}{\partial_{\mathrm{exp}}}
\newcommand{\bdyvis}{\partial_{\mathrm{vis}}}
\DeclareMathOperator{\QQ}{\mathbf{Q}}
\DeclareMathOperator{\PPprod}{\mathbf{Q}}
\DeclareMathOperator{\sep}{sep}
\DeclareMathOperator{\Refcon}{\mathrm{Ref}}
\DeclareMathOperator{\lex}{lex}

\newcommand{\cc}{\mathsf{c}}
\DeclareMathOperator{\GG}{\mathsf{G}}
\DeclareMathOperator{\FF}{\mathsf{F}}

\newcommand{\sd}[1]{\setminus^{\hspace{-0.2em} #1}}
\newcommand{\llb}{\llbracket}
\newcommand{\rrb}{\rrbracket}

\DeclareMathOperator{\near}{\mathrm{clust}}
\DeclareMathOperator{\capac}{cap}
\DeclareMathOperator{\Es}{esc}
\newcommand{\etalchar}[1]{$^{#1}$}

\newcommand{\CC}{\mathcal{C}}
\newcommand{\UU}{\mathcal{U}}
\newcommand{\VV}{\mathcal{V}}
\newcommand{\DD}{\mathcal{D}}
\DeclareMathOperator{\Conf}{\mathrm{Conf}}

\DeclareMathOperator{\Clust}{\mathrm{Clust}}

\DeclareMathOperator{\Fin}{\mathrm{Fin}}
\DeclareMathOperator{\Grow}{\mathrm{Grow}}

\DeclareMathOperator{\expo}{\mathrm{expo}}
\newcommand{\scr}{\mathscr}

\title[A Phase Transition in Harmonic Activation and Transport]{Existence of a Phase Transition\\ in Harmonic Activation and Transport}
\author{Jacob Calvert}
\address{Department of Statistics\\
 U.C. Berkeley \\
  367 Evans Hall \\
  Berkeley, CA, 94720-3840 \\
  U.S.A.}
  \email{calvert@berkeley.edu}


\begin{document}

\begin{abstract}
Harmonic activation and transport (HAT) is a stochastic process that rearranges finite subsets of $\Z^d$, one element at a time. Given a finite set $U \subset \Z^d$ with at least two elements, HAT removes $x$ from $U$ according to the harmonic measure of $x$ in $U$, and then adds $y$ according to the probability that simple random walk from $x$, conditioned to hit the remaining set, steps from $y$ when it first does so. In particular, HAT conserves the number of elements in $U$.

We study the classification of HAT as recurrent or transient, as the dimension $d$ and number of elements $n$ in the initial set vary. In \cite{calvert2021}, it was proved that the stationary distribution of HAT (on sets viewed up to translation) exists when $d = 2$, for every number of elements $n \geq 2$. We prove that HAT exhibits a phase transition in both $d$ and $n$, in the sense that HAT is transient when $d \geq 5$ and $n \geq 4$. 

Remarkably, transience occurs in only one ``way'': The set splits into clusters of two or three elements---but no other number---which then grow steadily, indefinitely separated. We call these clusters dimers and trimers. Underlying this characterization of transience is the fact that, from any set, HAT reaches a set consisting exclusively of dimers and trimers, in a number of steps and with at least a probability which depend on $d$ and $n$ only.
\end{abstract}

\maketitle

\setcounter{tocdepth}{1}
\tableofcontents

\section{Introduction}\label{sec intro}

Harmonic activation and transport (HAT) is a Markov chain that rearranges finite subsets of $\Z^d$ with at least two elements. With each step, an element is removed from the set ({\em activation}) and an element is added to the boundary of what remains ({\em transport}). Activation occurs according to the harmonic measure of the set which, informally, is the hitting probability of random walk ``from infinity.'' Transport occurs according to a certain hitting probability of simple random walk from the activated element.

HAT is interesting in part because of its remarkable behavior, and in part because of its connections to Laplacian growth, programmable matter, and studies of collective behavior. While HAT is not a growth model, it is related by harmonic measure to models of Laplacian growth, like diffusion-limited aggregation (DLA) \cite{witten1981}, which describe the evolution of a variety of physical interfaces \cite{MR3662912}. The value of this connection was demonstrated in \cite{calvert2021}, where HAT inspired a novel estimate of harmonic measure that generalizes a prediction about DLA from the physics literature \cite{lee1988}. By virtue of being a Markov chain that rearranges a finite subset of a graph, HAT is also related to models of programmable matter, like the amoebot model \cite{derakshandeh2014,cannon2016}. The amoebot model was used to design a self-organizing robot swarm that exhibits collective transport of objects \cite{li2021}. This functionality arises from a phase transition in the model's long-term behavior, which suggests that a phase transition in HAT could inspire new functionality for progammable matter. More broadly, models like HAT can be used to explore the possible behaviors of engineered and natural collectives \cite{calvert2023}.

We use the following notation to define harmonic measure. For an integer $i$, we denote $\Z_{\geq i} = \{i, i+1, \dots\}$ and $\N = \Z_{\geq 0}$ in particular. Fix a dimension $d \in \Z_{\geq 1}$. For $x \in \Z^d$, we denote the distribution of random walk $(S_j)_{j \in \N}$ from $S_0 = x$ by $\P_x$. Here and throughout, ``random walk'' refers to simple, symmetric random walk in $\Z^d$. We denote the first time that random walk returns to a set $A \subseteq \Z^d$ by $\tau_A = \inf\{j \geq 1: S_j \in A\}$. We denote the Euclidean norm by $\|{\cdot}\|$.

For finite $A \subset \Z^d$, we define the harmonic measure of $y$ in $A$ as a limit of conditional hitting probabilities of $A$:
\begin{equation}\label{lim him}
\H_A (y) = \lim_{\| x \| \to \infty} \P_x (S_{\tau_A} = y \mid \tau_A < \infty ).
\end{equation}
This limit exists and does not depend on the sequence in $\Z^d$ that is implicit in the notation ``$\|x\| \to \infty$,'' so $\H_A$ is well defined (see, e.g., \cite[Chapter 2]{lawler2013intersections}). 

The state space of HAT is $\Conf_d = \{V \subset \Z^d: 2 \leq |V| < \infty \}$, the collection of $d$-dimensional {\em configurations}, or finite subsets of $\Z^d$ with at least two elements. We denote the HAT configuration at time $t \in \N$ by $U_t$. To obtain $U_{t+1}$ from $U_t$, we first remove an element $X_t$ from $U_t$ according to $\H_{U_t}$. Then, we consider a random walk from $X_t$ that is conditioned to hit $U_t {\setminus} \{X_t\}$. If this random walk steps from $Y_t$ when it first does so, then we add $Y_t$ to form
\[
U_{t+1} = (U_t \setminus \{X_t\}) \cup \{Y_t\}.
\]
Note that, if $X_t = Y_t$, then $U_{t+1} = U_t$. Since $X_t$ cannot be an interior site of $U_t$, $X_t$ and $Y_t$ differ with positive probability.

In other words, given $U_t$, the probability that activation occurs at $x$ and transport occurs to $y$ is
\begin{equation}\label{eq tps0}
p_{U_t} (x,y) = \H_{U_t} (x) \, \P_x ( S_{\tau - 1} = y \mid \tau < \infty ),
\end{equation}
where $\tau$ abbreviates $\tau_{U_t {\setminus} \{x\}}$. 
We refer to the two factors in \eqref{eq tps0} as the activation and transport components of the dynamics.

\begin{definition}[Harmonic activation and transport] HAT is the discrete-time Markov chain $(U_t)_{t \geq 0}$ on the state space $\Conf_d$ with transition probabilities given by
\begin{equation}\label{eq tps}
\PP \left( U_{t+1} = (U_t \setminus \{x\}) \cup \{y\} \bigm\vert U_t \right) = 
\begin{cases}
p_{U_t} (x,y) & x \neq y,\\ 
\sum_{z \in \Z^d} p_{U_t} (z,z) & x = y,
\end{cases}
\end{equation}
for $x, y \in \Z^d$.
\end{definition}

Four key properties of HAT are apparent from its definition. To state them, we denote the diameter of $A \subseteq \Z^d$ by $\diam (A) = \sup_{x,y \in A} \|x-y\|$ and the law of HAT from $V$ (i.e., conditioned on $U_0 = V$) by $\PP_V$.
\begin{enumerate}
\item {\em Conservation of mass}. The number of elements in $U_t$ is fixed by the initial configuration $U_0$. Consequently, every irreducible component of the state space is contained in $\Conf_{d,n} = \{V \subset \Z^d: |V| = n\}$ for some $n \in \Z_{\geq 2}$.
\item {\em Variable connectivity}. If $|U_0| \geq 3$, then $U_t$ eventually reaches a configuration with two or more connected components.
\item {\em Asymmetric behavior of diameter}. The diameter of $U_t$ increases by at most one with each step: for every configuration $V$,
\begin{equation}\label{amlg0}
\PP_V (\diam (U_{t+1}) \leq \diam (U_t) + 1) = 1.
\end{equation}
In contrast, the diameter of $U_t$ can decrease by as much as $\diam (U_t) - 1$ in one step. For example, if $|V| = 2$, then $\PP_V (\diam (U_1) = 1) = 1$.
\item {\em Translation invariance}. The transition probabilities satisfy
\[
\PP_{W} (U_1 = V) = \PP_{W+x} (U_1 = V+x)
\] 
for all $V, W \in \Conf_d$ and $x \in \Z^d$. This motivates the association of each set $A \subseteq \Z^d$ to the equivalence class $\wt A$ consisting of the translates of $A$:
\[ \wt A = \big\{ B \subseteq \Z^d: \exists x \in \Z^d: B = A + x \big\}.\]
For convenience, if $V$ is a configuration, then we will also call $\wt V$ a configuration.
\end{enumerate}

This paper primarily concerns the classification of HAT as recurrent or transient, as the dimension $d$ and the number of elements $n$ in the initial configuration vary. First, note that there are $n$-element configurations that HAT cannot reach (i.e., realize as $U_t$ for some $t \geq 1$). This is because, for $p_{U_t} (x,y)$ to be positive, $y$ must have a neighbor in $U_t {\setminus} \{x\}$ and $y$ must have positive harmonic measure in $(U_t \setminus \{x\}) \cup \{y\}$. It is therefore impossible for all elements with positive harmonic measure in $U_{t+1}$ to be neighborless. This fact motivates the following definition.

\begin{definition}[Isolated, non-isolated configurations]
If $A \subset \Z^d$ is finite and if $x \in A$, then we say that $x$ is exposed in $A$ if $\H_A (x) > 0$. We say that $A$ is isolated if every element that is exposed in $A$ has no neighbors in $A$; we say that $A$ is non-isolated if it is not isolated. We denote by $\iso_{d,n}$ and $\noniso_{d,n}$ the collections of isolated and non-isolated $n$-element configurations in $\Z^d$, respectively. We denote the collections of the corresponding equivalence classes of configurations by $\wtiso_{d,n}$ and $\wtnoniso_{d,n}$.
\end{definition}

It is easy to see that HAT is positive recurrent on $\wtnoniso_{d,n}$, for any dimension $d$, when the number of elements $n$ is two or three. When $n=3$, isolated elements are removed with uniformly positive probability, which prevents a configuration's diameter from steadily growing. This argument does not apply when $n \geq 4$, because the diameter of a configuration can grow without isolating an element. For example, when $n=4$, two pairs of adjacent elements can ``walk'' apart. Nevertheless, it is possible to prove that, in two dimensions, HAT is positive recurrent on the class of non-isolated configurations for every $n \geq 2$.

\begin{theorem}[Positive recurrence in two dimensions; Theorem 1.6 of \cite{calvert2021}]\label{pos rec}
For every $n \geq 2$, from any $n$-element subset of $\Z^2$, HAT converges to a unique probability measure, supported on $\wtnoniso_{2,n}$. In particular, HAT is positive recurrent on $\wtnoniso_{2,n}$ for every $n \geq 2$.
\end{theorem}

Theorem \ref{pos rec} is a consequence of a phenomenon called {\em collapse} that HAT exhibits in two dimensions. Informally, collapse occurs when the diameter of a configuration is reduced to its logarithm over a number of steps proportional to this logarithm. When a configuration has a sufficiently large diameter in terms of $n$, collapse occurs with high probability in $n$ \cite[Theorem 1.5]{calvert2021}. In other words, the diameter experiences a negative drift in the sense of a Foster--Lyapunov theorem (e.g., Theorem~2.2.4 of \cite{fayolle1995}), which implies that HAT is positive recurrent.

In the context of Theorem \ref{pos rec}, the first of our main results establishes that HAT exhibits a phase transition, in the sense that HAT is transient in any dimension $d \geq 5$, for every $n \geq 4$.

\begin{theorem}[Transience in high dimensions]\label{trans} HAT is transient for every $d \geq 5$ and $n \geq 4$.
\end{theorem}

It is unnecessary to qualify that HAT is transient on $\wtnoniso_{d,n}$, because the states of $\wtiso_{d,n}$ are transient for every $d$ and $n$. At the end of this section, we briefly discuss a heuristic which suggests that $d=5$ is the critical dimension for the transience of HAT. Figure \ref{fig: phase_diagram} summarizes what is known about the phase diagram of HAT in the $d$--$n$ grid.

\begin{figure}[htbp]
\centering {\includegraphics[width=\linewidth]{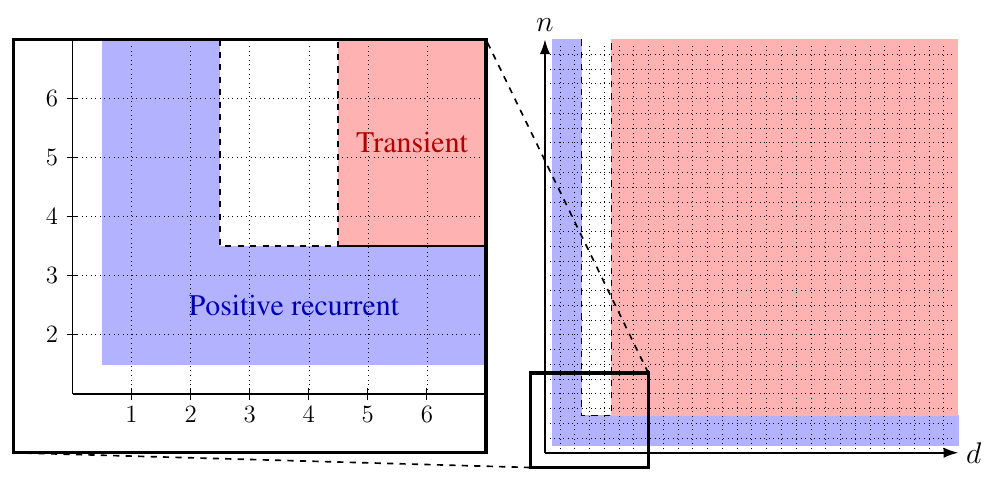}}
\caption{The phase diagram for HAT in the $d$--$n$ grid. HAT is positive recurrent on $\wtnoniso_{d,n}$ in the blue-shaded region and transient in the red-shaded region. The classification of HAT has not been established in the unshaded region.}
\label{fig: phase_diagram}
\end{figure}

Our second main result, Theorem~\ref{fate}, is a detailed description of the way that transience occurs when $d \geq 5$ and $n \geq 4$. Informally, HAT eventually reaches a configuration consisting of ``clusters'' of two or three elements, which grow apart indefinitely and never exchange elements. We state this result in terms of partitions of the HAT configuration. We refer to the parts of these partitions as clusters when the parts have small diameters relative to their separation. The next two definitions make this precise.

By a partition $\VV$ of a configuration $V$, we mean an ordered partition $(\VV^1, \dots, \VV^k)$ of $V$ into $k \geq 2$ nonempty, disjoint subsets. We use the following notation in this context. 
\begin{itemize}
\item For $i \in \{1,\dots,k\}$ and $A \subseteq \Z^d$, we denote by $\VV \cup^i A$ and $\VV \sd{i} A$ the partitions with $i$\textsuperscript{th} parts $\VV^i \cup A$ and $\VV^i \setminus A$, and all other parts equal to those of $\VV$.
\item If $x \in V$, then we denote the unique part of $V$ to which $x$ belongs by $[x]_\VV$. In other words, $x \in \VV^{[x]_\VV}$. 
\item The separation of a partition is the smallest distance between its parts:
\[
\sep (\VV) = \min_{1 \leq i \leq k} \dist (\VV^i, \VV^{\neq i}).
\]
Here, we use $\dist (A,B) = \inf_{x \in A, y \in B} \|x-y\|$ to denote the distance between $A, B \subseteq \Z^d$ and we use $\VV^{\neq i}$ to denote the union $\cup_{j \neq i} \VV^j$.
\end{itemize}

Given a partition of the HAT configuration at one time, there is a natural way to obtain a partition of the configuration at every later time: with each step, assign the transported element to the same part as the activated element.

\begin{definition}[Natural partitioning]\label{nat part}
Given a partition $\VV$ of the time-$s$ configuration $U_s$, the natural partitioning $(\UU_t)_{t \geq s}$ of $(U_t)_{t \geq s}$ with $\VV$ is inductively defined by $\UU_s = \VV$ and  
\begin{equation}\label{c with u}
\UU_{t+1} = \left( \UU_t \sd{i_t} \{X_t\} \right) \cup^{i_t} \{Y_t\}, \quad t \geq s,
\end{equation}
in terms of the time-$t$ sites of activation $X_t$ and transport $Y_t$, and the part $i_t = [X_t]_{\UU_t}$ of $\UU_t$ to which $X_t$ belongs.
\end{definition}

A clustering is a partition such that each part has at least two elements and the parts satisfy bounds on separation, in terms both absolute and relative to their diameter.

\begin{definition}[Clustering]\label{clustering def}
For $a, b > 0$, a partition $\VV = (\VV^1, \dots, \VV^k)$ of a configuration $V$ is an $(a,b)$ clustering of $V$, denoted $\VV \in \Clust_{a,b} (V)$, if
\begin{equation}\label{abs rel sep}
|\VV^i| \geq 2, \quad \dist (\VV^i, \VV^{\neq i}) \geq a, \quad \text{and} \quad \diam (\VV^i) \leq b \log \dist (\VV^i, \VV^{\neq i}), \quad 1 \leq i \leq k.
\end{equation}
We refer to the parts of a clustering as clusters. In particular, we call $\VV^i$ a dimer if $|\VV^i| = 2$ and a trimer if $|\VV^i| = 3$. We say that $\VV$ is an $(a,b)$ {\em dimer-or-timer} (DOT) clustering, denoted $\VV \in \Clust_{a,b}^\bullet (V)$, if $\VV$ satisfies
\begin{equation}\label{two or three elems}
|\VV^i| \in \{2,3\}, \quad 1 \leq i \leq k,
\end{equation}
in addition to \eqref{abs rel sep}.
\end{definition}

Dimers and trimers have a special status in $d \geq 5$ dimensions because they are the only clusters that can ``persist'' over many steps, in a sense that we elaborate at the end of the section. This is counterintuitive because smaller clusters are at greater risk of losing all of their elements to other clusters. However, in $d \geq 3$ dimensions, an activated element would likely escape to infinity in the absence of the conditioning in the transport component of the HAT dynamics \eqref{eq tps0}. Consequently, when clusters are well separated, the HAT dynamics biases the activated element to be transported to the cluster at which it was activated. As $d$ increases, this effect becomes more pronounced, enabling dimers and trimers to persist over many steps despite comprising few elements. In contrast, clusters of four or more elements cannot persist---not because they lose their elements to other, distant clusters---but because they can split into dimers or trimers, which grow apart as they persist.

By viewing dimers and trimers at consecutive return times to a given orientation, we can model their individual motions as $d$-dimensional random walks (albeit not simple ones), which suggests that their separation grows at a rate of roughly $t^{1/2}$. These return times will be exponentially tight and so, because diameter increases at most linearly in time \eqref{amlg0}, the clusters' diameters should never exceed roughly $\log t$. These observations suggest that, after $t$ steps, the clusters should constitute a $(t^{1/2}, 1)$ DOT clustering of the HAT configuration. This is the heuristic behind our second main result.

\begin{theorem}[The mechanism that produces transience]\label{fate}
Let $d \geq 5$ and $n \geq 4$, and let $V \in \Conf_{d,n}$. There exists $b = b(d,n) > 0$ such that, for any $\delta \in (0,\frac12)$, there exist $a = a(d,n,\delta) > 0$ and a $\PP_V$-a.s.\@ finite random time $\theta = \theta (d,n,\delta) \in \N$
 at which there is a clustering $\mathcal{W} \in \Clust_{a,b}^\bullet (U_\theta)$ such that the natural partitioning $(\UU_t)_{t \geq \theta}$ of $(U_t)_{t \geq \theta}$ with $\mathcal{W}$ satisfies 
\begin{equation}\label{dot conds}
\UU_t \in \Clust_{g(t-\theta),b}^\bullet (U_t), \quad t > \theta,
\end{equation}
where $g(s) = s^{\frac12 - \delta}$ for $s \geq 0$.
\end{theorem}

Theorem \ref{fate} identifies an a.s.\ finite random time $\theta$ at which there is a clustering of $U_\theta$ into dimers or trimers and forever after which {\em the same dimers or trimers} become steadily, increasingly separated---in terms both absolute and relative to their diameters. In particular, \eqref{dot conds} is stronger than 
\[
\Clust_{g (t-\theta),b}^\bullet (U_t) \neq \emptyset, \quad t > \theta,
\]
because it rules--out the exchange of elements between clusters after time $\theta$. 
According to the preceding heuristic (which the proof makes precise), a growth rate of $s^{\frac12}$ would be tight.

Theorem~\ref{fate} implies Theorem \ref{trans}, because HAT is irreducible on non-isolated configurations.

\begin{theorem}[Irreducibility]\label{irred}
HAT is irreducible on $\wtnoniso_{d,n}$, for every $d \geq 5$ and $n \geq 2$.
\end{theorem}

In fact, there is no barrier to establishing irreducibility for other values of $d$; we simply assume that $d \geq 5$ to facilitate the reuse of inputs to the other theorems.

\begin{proof}[Proof of Theorem \ref{trans}]
Let $d \geq 5$ and $n \geq 4$, and let $V$ be the $n$-element segment $\{(j,0,\dots,0): 1 \leq j \leq n\} \subset \Z^d$. 
It suffices to show that $\wt V$ is transient, since $\wt V \in \wtnoniso_{d,n}$ and since HAT is irreducible on $\wtnoniso_{d,n}$ by Theorem~\ref{irred}. Theorem~\ref{fate} implies that there is a $\PP_V$-a.s.\@ finite time $\theta$ such that $\diam (U_t) \geq n$ for $t \geq \theta$. Because $\diam (\wt V) = n - 1 < n$, this implies that there are $\PP_V$-a.s.\ finitely many returns to $\wt V$, hence $\wt V$ is transient.
\end{proof}

A key input to the proof of Theorem \ref{fate} is the fact that HAT reaches a configuration with an $(a,1)$ DOT clustering in a number of steps $N$ and with a probability of at least $p$ that depend only on $a$, $d$, and $n$. A conceptually minor but useful fact is that this clustering consists of line segments parallel to $e_1 = (1,0,\dots,0) \in \Z^d$. Specifically, we define $\Refcon$ to be the collection of reference dimers and trimers
\[
\Refcon = \left\{ \{x,x+e_1\}: x \in \Z^d\right\} \cup \left\{\{x, x+e_1, x+2e_1\}: x \in \Z^d\right\}
\]
and denote the collection of tuples of such configurations by $\Refcon^\times$.

\begin{theorem}[Formation of dimers and trimers]\label{thm form dot} Fix $d \geq 5$ and $n \geq 4$. For every $a > 0$, there exist $N = N (a,d,n) \in \N$ and $p = p(a,d,n) \in (0,1]$ such that, for any $V \in \Conf_{d,n}$,
\begin{equation}\label{thm eq form dot}
\PP_{V} \left( \Clust_{a,1}^\bullet (U_N) \cap \Refcon^\times \neq \emptyset \right) \geq p.
\end{equation}
\end{theorem}

The important feature of Theorem \ref{thm form dot} is that $N$ and $p$ do not depend on the diameter of $V$.  The proof takes the form of an analysis of three algorithms, which sequentially: (i) rearrange the configuration into well separated, connected clusters with at least two elements each; (ii) organize each cluster into a line segment; and (iii) split the segments into dimers and trimers. Collectively, the algorithms take as input an arbitrary configuration $V$ and desired separation $a$, and return a configuration $W$ that has an $(a,1)$ DOT clustering. It does not seem possible to appreciably simplify this process without introducing into $p$ a dependence on the diameter of $V$.

\subsection*{A heuristic which suggests that $d=5$ is the critical dimension for transience}\label{sub heuristic}

We address the role of the assumption that $d \geq 5$ in Theorem~\ref{trans} with a discussion of a heuristic. Consider a pair of dimers. Until they exchange elements, we can model the distance between them by the norm of a $d$-dimensional random walk. If they never exchange elements, then, because random walk is transient in $d \geq 3$ dimensions, their separation will grow steadily and without bound, as Theorem \ref{fate} predicts.

This basic picture is modified by the fact that dimers exchange elements over a number of steps which depends on their separation. Specifically, if the dimers are separated by a distance $a$, then they will typically exchange elements over $a^{d-2}$ steps. This timescale reflects the fact that a random walk from the origin $o \in \Z^d$, which is conditioned to return to $\{o,x\}$ for $o \neq x \in \Z^d$, reaches $x$ first with a probability of roughly $\|x\|^{2-d}$. If the dimers do not exchange elements during a period of $a^2$ steps, then the separation of the dimers typically doubles over the same period, after which it takes $2^{d-2}$ times longer for them to exchange elements. Hence, if $d \geq 5$, then dimer separation grows quickly enough that elements are never exchanged, which suggests transience. In contrast, if $d \in \{3,4\}$, then dimers typically consolidate before their separation doubles, which suggests recurrence. An analogous heuristic concludes the same of trimers. 

This paper develops the preceding heuristic into a proof of transience in $d \geq 5$ dimensions. As the proof shows, when $d \geq 5$, it suffices to understand how DOTs grow in separation and exchange elements. However, to prove recurrence in $d \in \{3,4\}$ dimensions, it would be necessary to extend such an understanding to clusters of all sizes, which introduces new challenges that are left to future work.

In summary, the greater the separation between DOTs, the longer it takes for them to exchange elements. This effect becomes more pronounced as $d$ increases. Until DOTs exchange elements, the pairwise distances between clusters behave like the norms of $d$-dimensional random walks, which inclines them to grow increasingly separated due to the transience of random walk in $d \geq 3$ dimensions. In $d \geq 5$ dimensions, we will be able to show that DOT separation grows rapidly enough in the absence of element exchange that it is typical for no element to be exchanged, leading to Theorem~\ref{fate}.

\subsection*{Organization} Figure~\ref{fig: logdep} shows how the proof of Theorem~\ref{fate} is organized. The main tool that we use to prove Theorem \ref{fate} is an approximation of HAT by another Markov chain, called {\em intracluster} HAT (IHAT), which treats clusters as if they inhabited separate copies of $\Z^d$. Under IHAT, we can model the motion of dimers and trimers as random walks. This is the approach that underlies the proof of Proposition~\ref{inf xi}. In Section~\ref{sec fate}, we prove Theorem~\ref{fate}, assuming Theorem~\ref{thm form dot} and Proposition~\ref{inf xi}. In Section~\ref{strat}, we briefly discuss the strategy of the proof of Proposition~\ref{inf xi}. We motivate and define IHAT in Section~\ref{sec intra}, which requires an extension of harmonic measure to tuples of sets. Section \ref{hm bd} proves estimates that we use in Section~\ref{comp act} to compare the transition probabilities of HAT and IHAT. We apply these results to bound the error of approximating HAT by IHAT in Section \ref{sec comparison}. The main approximation result, Proposition~\ref{pe to qe}, states that we can bound below HAT probabilities with IHAT probabilities, for events that entail sufficiently rapid growth of separation in the natural partitioning of HAT. Section \ref{ref times} introduces a random walk model of the separation between a pair of IHAT clusters, and Section \ref{sec clust sep} uses this random walk to obtain key estimates of separation growth under IHAT. Beginning in Section \ref{sec strat 11}, the focus shifts to the proof of Theorem \ref{thm form dot}, which concerns the formation of configurations with DOT clusterings. Section \ref{sec inputs} presents additional geometric inputs and random walk estimates, which are applied in Section \ref{sec form dot} to analyze the three algorithms around which the proof of Theorem \ref{thm form dot} is organized. The last section, Section \ref{sec irred}, proves Theorem \ref{irred}, which states that HAT is irreducible on non-isolated configurations.

\begin{figure}[htbp]
\centering \resizebox{0.85\textwidth}{!}{
\begin{tikzpicture}[->,>=stealth,auto,node distance=3cm,  thick,main node/.style={rectangle,draw}]

  \node[main node] (1) {Theorem~\ref{fate}};
  \node[main node] (2) [below left of=1] {Theorem~\ref{thm form dot}};
  \node[main node] (3) [below right of=1] {Proposition~\ref{inf xi}};
  \node[main node] (4) [below left of=3] {Proposition~\ref{pe to qe}};
  \node[main node] (5) [below of=3] {Proposition~\ref{rel sep bd}};
  \node[main node] (6) [below right of=3] {Proposition~\ref{tl ntb}};
  \node[main node] (7) [below left of=2] {Propositions~\ref{alg1 works}--\ref{break lines}};
  \node[main node] (8) [below of=4] {Proposition~\ref{p to q}};
  \node[main node] (9) [below of=6] {Proposition~\ref{qg1}};
  \node[main node] (10) [below right of=9] {Propositions~\ref{h sub g}--\ref{h4}};
  \node[main node] (11) [below right of=8] {Proposition~\ref{at pairs}};
  \node[main node] (12) [below left of=8] {Propositions~\ref{prop h comp} \& \ref{prop t comp}};

  \path[every node/.style={font=\sffamily\small}]
    (2) edge node {} (1)
    (3) edge node {} (1)
    (4) edge node {} (3)
    (5) edge node {} (3)
    (6) edge node {} (3)
    (7) edge node {} (2)
    (8) edge node {} (4)
    (9) edge node {} (6)
    (10) edge node {} (9)
    (11) edge node {} (8)
    (12) edge node {} (8);
\end{tikzpicture}
}
\caption[A diagram of the inputs to the proof of Theorem~\ref{fate}]{A diagram of the inputs to the proof of Theorem~\ref{fate}, excluding basic geometric results and random walk estimates.
}
\label{fig: logdep}
\end{figure}
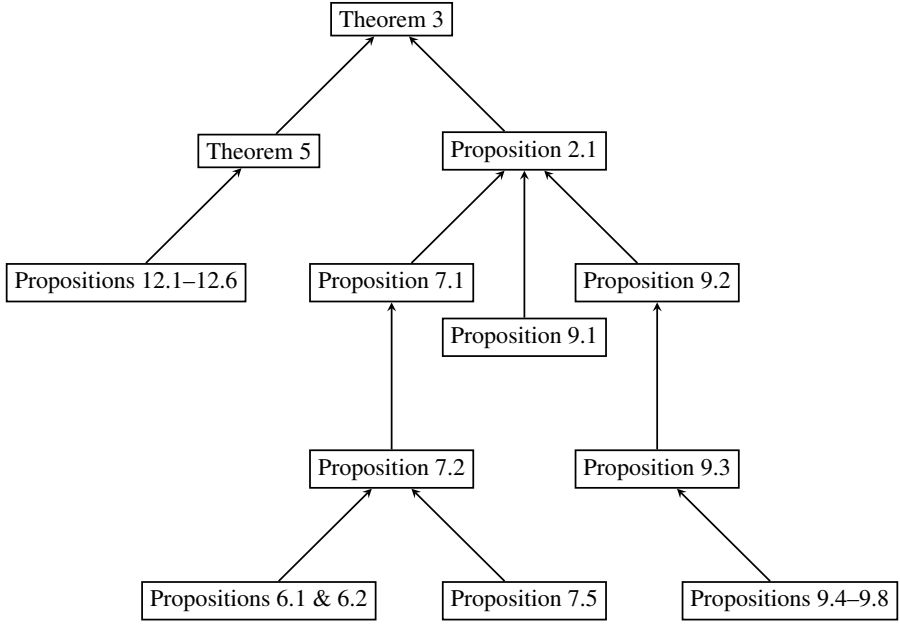

\subsection*{Conventions and Forthcoming Notation} When we refer to a ``constant'' without further qualification, we always mean a positive number. We always use $d$ and $n$ to denote positive integers that represent the ambient dimension and a number of elements. We use $o$ to denote the origin in $\Z^d$ and $e_j$ to denote the element of $\Z^d$ with $j$\textsuperscript{th} coordinate equal to $1$ and all other coordinates equal to $0$. For a real number $r$, we use $[r]$ to denote the integer part of $r$. We use $\bdy A$ to denote the exterior vertex boundary $\{z \in \Z^d: \dist(z,A) = 1\}$ of a set $A \subseteq \Z^d$, and $\rad (A) = \{\|x\| : x \in A\}$ to denote its radius. We use $\E_x$ and $\EE_V$ to denote expectation with respect to $\P_x$ and $\PP_V$, and $\1_E$ to denote the indicator of an event $E$. We use $f \lesssim g$ to denote the estimate $f \leq c g$ for a constant $c$ that may depend on $d$, and we denote such a quantity $f$ by $O(g)$. In some instances, we use $f \lesssim_n g$ to denote the same estimate, except permitting $c$ to depend on $n$ as well. We use $f \gtrsim g$ and $\Omega (g)$ analogously, for the reverse estimate. We use $f \asymp g$ when $f \lesssim g$ and $f \gtrsim g$.

\section{Proof of Theorem \ref{fate}}\label{sec fate}

Theorem \ref{fate} states that there is a random time $\theta$ which is $\PP_V$-a.s.\ finite for every configuration $V$ and after which the natural clustering of $(U_\theta, U_{\theta+1}, \dots)$ grows in separation according to \eqref{dot conds}. Informally, we will define $\theta$ as the time of the first success in a sequence of trials, each of which attempts to observe the natural clustering with a sufficiently well separated DOT clustering satisfy \eqref{dot conds}. The fact that $\theta$ is a.s.\ finite will be a simple consequence of two results. First, Theorem \ref{thm form dot} implies that, if the present trial fails, then we can conduct another after waiting an a.s.\ finite number of steps for $U_t$ to have a sufficiently well separated DOT clustering. Second, the following result states that each trial succeeds with a uniformly positive probability, hence we need only conduct an a.s.\ finite number of trials before one succeeds.

\begin{proposition}\label{inf xi}
Fix $d \geq 5$ and $n \geq 4$. There exists $b = b (n) \geq 1$ such that, for any $\delta \in (0,\frac12)$, there exists $\alpha = \alpha (d,n,\delta) > 0$ such that, if $W \in \Conf_{d,n}$ has a clustering $\mathcal{W} \in \Clust_{\alpha,b}^\bullet (W) \cap \Refcon^\times$, then 
\begin{equation}\label{eq inf xi}
\PP_{W} (\xi = \infty) \geq \frac14,
\end{equation}
where $\xi$ is the first time that the natural clustering $(\UU_t)_{t \geq 0}$ of $(U_t)_{t \geq 0}$ with $\mathcal{W}$ is not in $\Clust_{g(t),b}^\bullet (U_t)$ where $g(s) = s^{\frac12 - \delta}$ for $s \geq 0$, i.e., 
\[
\xi = \inf \{t \geq 0: \UU_t \notin \Clust_{g(t),b}^\bullet (U_t)\}.
\]
\end{proposition}

The proof of Proposition \ref{inf xi} will comprise several sections, and we dedicate the next section to a discussion of the proof strategy. For now, we assume this proposition and Theorem~\ref{thm form dot}, and use them to prove Theorem \ref{fate}.

\begin{proof}[Proof of Theorem \ref{fate}]
Fix $d \geq 5$ and $n \geq 4$. Let $\alpha > 0$ and $b \geq 1$ be the constants from Proposition~\ref{inf xi}. Fix $\delta \in (0,\frac12)$. In these terms, denote $g(s) = s^{\frac12-\delta}$ for $s \geq 0$. Additionally, let $\psi$ be a function which, given a configuration $U \in \Conf_{d,n}$ that has a clustering in $\Clust_{\alpha,b}^\bullet (U)$, determines one such clustering. (We use $\psi$ to ``pick'' one of potentially multiple clusterings; it is otherwise unimportant.) 

We define $\theta$ in terms of two sequences of random times, $(\tau_i)_{i \geq 1}$ and $(\xi_i)_{i \geq 0}$. In words, $\tau_i$ is the first time $t \geq \xi_{i-1}$ at which $U_t$ has an $(\alpha,b)$ DOT clustering in $\Refcon^\times$, and $\xi_i$ is the first time $t > \tau_i$ at which the natural clustering $(\UU_t)_{t \geq \tau_i}$ of $(U_t)_{t \geq \tau_i}$ with $\psi (U_{\tau_i})$ is not a $(g(t-\tau_i),b)$ DOT clustering of $U_t$. More precisely, we define $\xi_0 = 0$ and, for $i \geq 1$,  
\begin{align*}
\tau_i &=  \inf \left\{t \geq \xi_{i-1}: \Clust_{\alpha,b}^\bullet (U_t) \cap \Refcon^\times \neq \emptyset \right\}, \quad \text{and}\\ 
\xi_i &= \inf \left\{t > \tau_i: \UU_t \notin \Clust_{ g(t-\tau_i),b}^\bullet (U_{t}) \right\}.
\end{align*}
Lastly, we define $\theta = \tau_I$ for $I = \inf \{i \geq 1: \tau_i < \infty, \xi_i = \infty\}$. By the definition of $\theta$, \eqref{dot conds} is satisfied, so it remains to show that $\theta$ is $\PP_V$-a.s.\@ finite for every $V \in \Conf_{d,n}$.

Let $J =\inf \{j : \xi_j = \infty\}$. For every $V \in \Conf_{d,n}$, we have
\begin{equation}\label{thetatoj}
\PP_V (\theta < \infty) = \PP_V (I < \infty) = \PP_V (J < \infty).
\end{equation}
The first equality follows from the definitions of $I$ and $\theta$. The second equality follows from the fact that $\PP_V (\tau_i < \infty \mid \xi_{i-1} < \infty) = 1$, which is a simple consequence of Theorem~\ref{thm form dot}. 

To bound the tail probabilities of $J$, we write
\begin{equation}\label{jj2}
\PP_V (J > j+1 \,|\, J > j) = \frac{\EE_V \big[ \PP_{U_{\tau_j}} (\xi_1 < \infty); J > j\big]}{\PP_V (J > j)} \leq \frac{3}{4}.
\end{equation}
The equality follows from the strong Markov property applied at time $\xi_j$ and the fact that $\PP_{U_{\xi_j}} (\tau_1 < \infty) = 1$. The inequality is due to Proposition~\ref{inf xi}, which implies that $\PP_{U_{\tau_j}} (\xi_1 < \infty)$ is at most $\frac34$.

The bound \eqref{jj2} implies that $\PP_V (J > j)$ is summable, so the Borel-Cantelli lemma implies that $J$ is $\PP_V$-a.s.\@ finite, which by \eqref{thetatoj} implies the same of $\theta$.
\end{proof}

\section{Strategy for the proof of Proposition \ref{inf xi}}\label{strat}

The proof of Proposition \ref{inf xi} has two main steps. First, we prove that, when DOTs are $a$ separated in $d \geq 5$ dimensions, HAT approximates a related process, called {\em intracluster HAT}, in which transport occurs only to the cluster at which activation occurred, over $a^2$ steps, up to an error of $O(a^{-1} \log (a))$. Second, we show that over $a^2$ steps of IHAT, the separation between every pair of clusters effectively doubles, except with a probability of $O(a^{-1})$. We show this by considering the pairwise differences between representative elements of each cluster, viewed at consecutive times of return to certain, ``reference'' DOTs. Viewed in this way, the pairwise differences are $d$-dimensional (symmetric, but not simple) random walks. 
We then apply the same argument with $2a$ in the place of $a$, then $4a$ in the place of $2a$, and so on. Each time the separation doubles, the approximation and exception errors halve, which implies that if $a$ is sufficiently large, then the separation grows without bound, with positive probability.

The second step is possible because, under IHAT, each DOT inhabits a separate copy of $\Z^d$, which simplifies our analysis of their separation. We define IHAT by conditioning the transport component of the HAT dynamics on {\em intracluster transport}, i.e., transport can only occur to the boundary of the cluster at which activation occurred. Because intercluster transport over $a^2$ steps is atypical when clusters are $a$ separated, IHAT is a good approximation of HAT over the period during which clusters typically double in separation.

The activation component of the IHAT dynamics is defined in terms of an extension of harmonic measure to clusterings of configurations, which approximates the harmonic measure of the union of the clusters when they are well separated. This harmonic measure is proportional to the escape probability of each element, from the cluster to which the element belongs---not the union of the clusters. In this way, the IHAT dynamics treats each cluster in isolation. When clusters are $a$ separated, the harmonic measure of a clustering agrees with the harmonic measure of the union of the clusters, up to a factor of $1 - O(a^{2-d})$. In fact, the discrepancy between the transport components of HAT and IHAT will give rise to the dominant error factor; we discuss this in greater detail in the next section.

\section{Intracluster HAT}\label{sec intra}

This section motivates the definition of intracluster HAT by examining the transition probabilities of HAT. According to the discussion of heuristics at the end of Section~\ref{sec intro}, if a HAT configuration in $d \geq 3$ dimensions has an $(a,b)$ DOT clustering, then intercluster transport is rare to the extent that $\frac{b \log a}{a}$ is small. This observation suggests that, if a configuration consists of well separated clusters, then it is possible to approximate HAT by an analogous but simpler process in which the clusters inhabit separate copies of $\Z^d$. In other words, this analogous process, which we call intracluster HAT, is a Markov chain on {\em tuples} of configurations. To define IHAT, we adapt the activation and transport components of HAT, in a way that leads to small approximation error in terms of $a$ and $b$ when a configuration has an $(a,b)$ DOT clustering. The next two subsections elaborate the way that we adapt the components of the HAT dynamics to a Markov chain on tuples of configurations, and explain the circumstances under which these components closely approximate their analogues. The third subsection defines IHAT.

\subsection{Adapting the activation component}

The activation component of the HAT dynamics from a configuration $U \in \Conf_d$ is simply the harmonic measure of $U$ \eqref{eq tps0}. We therefore seek to define an analogue of harmonic measure for the states IHAT, i.e., tuples of configurations. To be useful, this analogue must closely approximate $\H_U$ whenever a tuple of configurations constitutes a well separated DOT clustering of $U$.

We use an expression for harmonic measure that is equivalent to \eqref{lim him} in $d \geq 3$ dimensions. Define the escape probability of a finite set $A \subset \Z^d$ by
\[
\Es_A (x) = \P_x (\tau_A = \infty) \ \1 (x \in A), \quad x \in \Z^d.
\]
Further define the capacity of $A$ by
\[
\capac_A = \sum_{x \in A} \Es_A (x).
\]
The harmonic measure of $A$ equals
\begin{equation*}
\H_A (x) = \frac{\Es_A (x)}{\capac_A}, \quad x \in \Z^d.
\end{equation*}
For a proof of this fact, see \cite[Theorem 2.1.3]{lawler2013intersections}. Note that, for $\Es_A (x)$ and $\H_A (x)$ to be positive, it is necessary (but not sufficient) for $x$ to belong to the interior vertex boundary of $A$. 

The following example motivates the way that we define the activation component of IHAT. Consider a configuration $U \in \Conf_d$ that can be partitioned into $\mathcal{P} = (\mathcal{P}^1, \mathcal{P}^2)$ such that $|\mathcal{P}^1| \in \{2,3\}$ and $\sep (\mathcal{P}) \geq a$ for some $a > 1$. We think of $\mathcal{P}^1$ as one DOT and $\mathcal{P}^2$ as the union of one or more clusters. If $x \in \mathcal{P}^1$, then we can bound below $\H_U (x)$ in terms of $\H_{\mathcal{P}^1} (x)$, by noting that
\[
\Es_U (x) \geq \Es_{\mathcal{P}^1} (x) - \P_x (\tau_{\mathcal{P}^2} < \infty).
\]
In Section~\ref{hm bd}, we will prove that, if $a$ is at least a certain constant, then
\[
\Es_{\mathcal{P}^1} (x) \gtrsim 1 \quad \text{and} \quad \P_x (\tau_{\mathcal{P}^2} < \infty) \lesssim |U| a^{2-d}.
\]
The first bound is plausible because $\mathcal{P}^1$ only has two or three elements. The second bound follows from a union bound over the elements of $\mathcal{P}^2$ and the fact that, if $y \in \Z^d$ is sufficiently far from $x$, then $\P_x (\tau_y < \infty) \asymp \|x-y\|^{2-d}$.   
According to these bounds, the escape probability of $U$ satisfies
\[
\Es_U (x) \geq \left( 1 - O \left( |U| a^{2-d} \right) \right) \Es_{\mathcal{P}^1} (x).
\]
To convert this into a lower bound of $\H_U (x)$, we divide by the capacity of $U$ and identify a factor of $\H_{\mathcal{P}^1} (x)$:
\begin{equation*}
\H_U (x) \geq \left( 1 - O \left( |U| a^{2-d} \right) \right) \frac{\capac_{\mathcal{P}^1}}{\capac_U} \H_{\mathcal{P}^1} (x).
\end{equation*}
Lastly, we note that the capacity of a union is at most the sum of the individual sets' capacities \cite[Proposition 2.2.1]{lawler2013intersections}, hence 
\begin{equation}\label{heur2}
\H_U (x) \geq \left( 1 - O \left( |U| a^{2-d} \right) \right) \frac{\capac_{\mathcal{P}^1}}{\capac_{\mathcal{P}^1} + \capac_{\mathcal{P}^2}} \H_{\mathcal{P}^1} (x).
\end{equation}

The virtue of \eqref{heur2} is that, aside from an error term that is small when the separation $a$ is large relative to $|U|$, the lower bound refers to the parts of the partition in isolation, which aligns with our goal of treating well-separated clusters as if they inhabit separate copies of $\Z^d$. It suggests that we should define the activation component of IHAT in the following way:
\begin{enumerate}
\item Given a tuple of finite sets, randomly select one of these sets with a probability that is proportional to its capacity.
\item Then, select an element of this set according to its harmonic measure.
\end{enumerate}
We formalize this as a definition of harmonic measure for tuples of finite sets.

Denote by $\Fin_d$ the collection of nonempty, finite subsets of $\Z^d$ and by $\Fin_d^\times$ the collection of tuples of such sets, i.e., 
\[
\Fin_d = \{A \subset \Z^d: 1 \leq |A| < \infty\} \quad \text{and} \quad \Fin_d^\times = \cup_{k=1}^\infty \left\{ (\A^i)_{i=1}^k \in \Fin_d^k \right\}.
\]
For $\A = (\A^i)_{i=1}^k \in \Fin_d^\times$, we refer to $\A^i$ as the $i$\textsuperscript{th} entry of $\A$ and we use $\# \A$ to denote the number of entries $k$ in $\A$.

We define the capacity of $\A \in \Fin_d^\times$ by
\[
\capac_\A = \sum_{i=1}^{\# \A} \capac_{\A^i},
\]
and the harmonic measure of $\A$ by 
\begin{equation}\label{ichm}
\H_\A (i,x) = \frac{\Es_{\A^i} (x)}{\capac_\A}, \quad 1 \leq i \leq \# \A, \quad x \in \A^i.
\end{equation}
We refer to $\H_\A (i,x)$ as the harmonic measure of $\A$ at $(i,x)$. Note that we must specify $i$ in $\H_\A (i,x)$ because the entries of $\A$ may not be disjoint. 
To interpret $\H_\A (i,x)$, note that the harmonic measure of $\A$ at $(i,x)$ is equal to the harmonic measure of $\A^i$ at $x$, weighted by the ratio of the capacities of $\A^i$ and $\A$:
\begin{equation}\label{hm interp}
\H_\A (i,x) = \frac{\capac_{\A^i}}{\capac_\A} \, \H_{\A^i} (x).
\end{equation}
In other words, to obtain $(I,X) \sim \H_\A$, randomly select an entry $I$ of $\A$ in proportion to its capacity, then randomly select an element $X$ of this entry according to its harmonic measure. 

We use this definition of harmonic measure in Section~\ref{subsec ihat} to define the activation component of IHAT. In Section~\ref{comp act}, we revisit \eqref{heur2} in preparation for the proof of our main approximation result, which compares the transition probabilities of HAT and IHAT. In the following subsection, we motivate the definition of the transport component of IHAT.

\subsection{Adapting the transport component}\label{subsec motivation}

Recall that the transport component of HAT from a configuration $U \in \Conf_d$, given a site of activation $x \in U$, is the conditional probability
\begin{equation}\label{heur3}
\P_x (S_{\tau-1} = y \mid \tau < \infty), \quad y \in \Z^d,
\end{equation}
where $\tau$ abbreviates $\tau_{U \setminus \{x\}}$. We aim to compare this conditional probability to the conditional probability that results from replacing $\tau$ with the return time to the rest of the cluster of $U$ to which $x$ belongs. The latter is the analogue of \eqref{heur3} with exclusively intracluster transport.

Let $a > e$ and $b > 0$, and assume $d \geq 3$. Suppose that $U$ has an $(a,b)$ clustering $\CC = (\CC^1, \CC^2)$, i.e., each cluster has at least two elements, the clusters are separated by a distance of at least $a$, and their diameters are at most $b \log \dist (\CC^1, \CC^2)$. Fix elements $x \in \CC^1$ and $y$ in the exterior vertex boundary of $\CC^1 {\setminus} \{x\}$, and abbreviate $A = \CC^1 {\setminus} \{x\}$. In these terms, we aim to compare the probability in \eqref{heur3} to
\begin{equation*}
\P_x (S_{\tau_A-1} = y \mid \tau_A < \infty).
\end{equation*}

For $\{S_{\tau-1} = y\}$ to occur, the random walk must visit $A$ before $\CC^2$ because $y \in \bdy A$ and $\dist (A, \CC^2) > 1$ by assumption. Hence,   
\[
\P_x ( S_{\tau - 1} = y, \tau < \infty) = \P_x (S_{\tau - 1} = y, \tau_A < \tau_{\CC^2}, \tau < \infty).
\]
In fact, because $\tau = \min \{\tau_A, \tau_{\CC^2}\}$, this probability equals
\[
\P_x (S_{\tau_A - 1} = y, \tau_A < \infty) - \P_x (S_{\tau_A - 1} = y, \tau_{\CC^2} < \tau_A < \infty),
\]
which implies that
\begin{multline}\label{ict hat}
\P_x ( S_{\tau - 1} = y \mid \tau < \infty) = \P_x (S_{\tau_A - 1} = y \mid \tau_A < \infty)\\ 
\times \underbrace{\frac{\P_x (\tau_A < \infty)}{\P_x (\tau < \infty)}}_{( \ref{ict hat} \text{a})} \Bigg( 1 - \underbrace{\frac{\P_x (S_{\tau_A - 1} = y, \tau_{\CC^2} < \tau_A < \infty)}{\P_x (S_{\tau_A - 1} = y, \tau_A < \infty)}}_{( \ref{ict hat} \text{b})} \Bigg).
\end{multline}

In Section~\ref{comp act}, we will show that (\ref{ict hat}a) and (\ref{ict hat}b) satisfy
\[
\text{(\ref{ict hat}a)} \geq 1 - O \left( |U| \left( b a^{-1} \log a\right)^{d-2} \right) \quad \text{and} \quad \text{(\ref{ict hat}b)} \leq O \left( |U| \left(b a^{-2} \log a \right)^{d-2} \right). 
\]
These bounds rely on the fact that random walk from $x$ hits a sufficiently distant set $B$ with a probability of at least $\dist(x,B)^{2-d}$ and at most $|B| \dist(x,B)^{2-d}$, up to constant factors. The separation properties that define an $(a,b)$ clustering are valuable for simplifying ratios of these quantities.

For example, the $O$ term in the lower bound of (\ref{ict hat}a) arises as the ratio of the hitting probabilities of $\CC^2$ and $A$ from $x$. Since $\CC^2$ and $A$ are distances of at least $\sep (\CC)$ and at most $\diam (\CC^1)$ from $x$, this ratio is roughly
\begin{equation*}
|U| \left( \frac{\diam (\CC^1)}{\sep (\CC)} \right)^{d-2} \leq |U| \left( \frac{b \log \sep (\CC)}{\sep (\CC)} \right)^{d-2} \leq |U| \left( \frac{b \log a}{a} \right)^{d-2}.
\end{equation*}
The first inequality holds because $\CC$ is $(\cdot, b)$ separated; the second holds because $\CC$ is $(a,\cdot)$ separated and because $\frac{\log r}{r}$ decreases as $r > e$ increases.

The upper bound of (\ref{ict hat}b) follows from similar considerations, with the exception that the event in the numerator of (\ref{ict hat}b) requires random walk from $x$ to traverse a distance of $\sep (\CC)$ twice, from $\CC^1$ to $\CC^2$ and then back to $A$. This leads to an additional factor of $a^{2-d}$ in the bound of (\ref{ict hat}b).

Substituting these bounds into \eqref{ict hat} gives
\begin{equation}\label{heur4}
\P_x (S_{\tau-1} = y \mid \tau < \infty) \geq \left(1 - O \left( |U| \left( b a^{-1} \log a \right)^{d-2} \right) \right) \, \P_x (S_{\tau_A-1} = y \mid \tau_A < \infty).
\end{equation}
The virtue of the lower bound in \eqref{heur4} is that, aside from the error term, it refers only to the cluster to which $x$ belongs, and not the entire configuration $U$. It shows that the approximation of HAT's transport component by one in which only intracluster transport is allowed is accurate to the extent that $|U| (\frac{b\log a}{a})^{d-2}$ is small. In the next subsection, we use this intracluster transport component to define IHAT.

\subsection{The definition of IHAT}\label{subsec ihat}

The state space of IHAT is the collection $\Conf_d^\times \subset \Fin_d^\times$ of tuples of configurations: 
\[
\Conf_d^\times = \cup_{k=1}^\infty \left\{ (\A^i)_{i=1}^k \in \Conf_d^k \right\}.
\]
We denote the state of IHAT at time $t \geq 0$ by $\VV_t$. Given $\VV_t$, we define the transition probabilities in terms of 
\begin{equation}\label{eq tps q}
q_{\VV_t} (i,x,y) = \H_{\VV_t} (i,x) \, \P_x ( S_{\tau-1} = y \mid \tau < \infty), \quad 1 \leq i \leq \# \VV_t, \quad x, y \in \Z^d,
\end{equation}
where $\tau$ abbreviates $\tau_{\VV_t^i \setminus \{x\}}$. In analogy with $p_{U_t} (x,y)$ \eqref{eq tps0}, the quantity $q_{\VV_t} (i,x,y)$ is the probability that activation occurs in $\VV_t^i$ at $x$ and transport occurs to $y$. Note that, for \eqref{eq tps q} to be positive, it is necessary for $y$ to belong to the exterior vertex boundary of $\VV_t^i \setminus \{x\}$.

\begin{definition}[Intracluster HAT] 
IHAT is the discrete-time Markov chain $(\VV_t)_{t \geq 0}$ on the state space $\Conf_d^\times$ with the following transition probabilities given by 
\begin{equation}\label{eq tps pprod}
\QQ \left( \VV_{t+1} = (\VV_t \sd{i} \{x\}) \cup^i \{y\} \bigm\vert \VV_t \right) = 
\begin{cases}
q_{\VV_t} (i,x,y) & x \neq y,\\ 
\sum_{j=1}^{\# \VV_t} \sum_{z \in \Z^d} q_{\VV_t} (j,z,z) & x = y,
\end{cases}
\end{equation}
for $1 \leq i \leq \# \VV_t$ and $x,y \in \Z^d$. We denote the law of IHAT from $\VV$ by $\QQ_\VV$.
\end{definition}

\section{Inputs to the comparison of HAT and IHAT}\label{hm bd}

This section proves estimates of hitting probabilities and harmonic measure that we will use to compare the transition probabilities of HAT and IHAT. 
Our basic tool is Green's function $G(x)$, which is defined in $d \geq 3$ dimensions as the expected number of visits to $x$ by random walk from the origin: 
\begin{equation}\label{eq green}
G(x) = \E_o \left[ \sum_{j=0}^\infty \1 (S_j = x) \right], \quad x \in \Z^d.
\end{equation}
There is a constant $\kappa > e$ such that
\begin{equation}\label{eq green asymp}
G(x) \asymp \|x\|^{2-d},
\end{equation}
for every $x \in \Z^d$ with norm $\|x\| \geq \kappa$ \cite[Theorem 4.3.1]{lawler2010random}. (We specify that $\kappa$ exceeds $e$ so that $\frac{\log r}{r}$ decreases in $r \geq \kappa$, a fact that we will use later.) This estimate of Green's function implies useful bounds on hitting probabilities.

\begin{lemma}\label{simple set hit}
Let $x \in \Z^d$ and let $A \subset \Z^d$ be nonempty and finite. If $\dist (x,A) \geq 1$, then
\begin{equation}\label{eq simple set hit lb}
\P_x (\tau_A < \infty) \gtrsim \dist (x, A)^{2-d}.
\end{equation}
If $\dist (x, A) \geq \kappa$, then
\begin{equation}\label{eq simple set hit ub}
\P_x (\tau_A < \infty) \lesssim |A| \dist (x,A)^{2-d}.
\end{equation}
\end{lemma} 

\begin{proof}
Suppose that $y_1$ is the element of $A$ closest to $x$ and that $z = y_2$ maximizes $\P_x (\tau_z < \infty)$ among $z \in A$. The inclusion $\{\tau_{y_1} < \infty\} \subseteq \{\tau_A < \infty\}$ and a union bound over the elements of $A$ imply that
\[
\P_x (\tau_{y_1} < \infty) \leq \P_x (\tau_A < \infty) \leq |A| \, \P_x (\tau_{y_2} < \infty).
\]
If $\dist (x,A) \in [1,\kappa)$, then \eqref{eq simple set hit lb} holds because $\P_x (\tau_{y_1} < \infty) \gtrsim 1$. If $\dist (x,A) \geq \kappa$, then \eqref{eq simple set hit lb} holds because
\[
\P_x (\tau_{y_1} < \infty) = \frac{G(x-y_1)}{G(o)} \stackrel{\eqref{eq green asymp}}{\asymp} \|x-{y_1}\|^{2-d} = \dist(x,A)^{2-d},
\]
while \eqref{eq simple set hit ub} holds because
\[
\P_x (\tau_{y_2} < \infty) = \frac{G(x-y_2)}{G(o)} \stackrel{\eqref{eq green asymp}}{\asymp} \|x-y_2\|^{2-d} \leq \dist(x,A)^{2-d}.
\]
The first equality in the two preceding displays is a standard identity; see, e.g., \cite[Equation 3.4]{popov2021}.
\end{proof}

We also need two monotonicity properties of Green's function. The first property is that Green's function is nonincreasing with respect to a partial order $\preceq$ on $\Z^d$, defined for $x, y \in \Z^d$ by 
\[
x \preceq y \quad \iff \quad \forall i \in \llb d \rrb, \quad |x_i| \leq |y_i|,
\]
where $x_i$ and $y_i$ denote the $i$\textsuperscript{th} components of $x$ and $y$.

\begin{lemma}[Lemma 8 of \cite{chiarini2016note}]\label{gpart}
If $x, y \in \Z^d$ and $x \preceq y$, then $G(x) \geq G(y)$. 
\end{lemma}

The second property is that the value of Green's function at the origin $o \in \Z^d$ is nonincreasing with dimension. To emphasize the dependence of this value on $d$, we denote it by $G_d (o)$.
\begin{lemma}\label{g dec}
If $d_1 \leq d_2$, then $G_{d_1} (o) \geq G_{d_2} (o)$.
\end{lemma}

\begin{proof}
Montroll \cite[Eq. (2.10)]{montroll1956} expresses $G_d (o)$ in terms of $I_0$, the modified Bessel function of the first kind, as
\[ G_d(o) = \int_0^\infty e^{-t} I_0 \Big( \frac{t}{d} \Big)^d d t.\]
According to the integral representation of $I_0$ \cite[Eq. 9.6.19]{abramowitz1964}, $I_0 (\frac{t}{d})^d$ is the $L^{\frac{1}{d}}$ norm of $e^{t \cos \theta}$ with respect to the probability measure $\pi^{-1} d \theta$ on $[0,\pi]$:
\[ I_0 \Big(\frac{t}{d}\Big)^d = \left(\frac{1}{\pi} \int_0^\pi e^{\tfrac{t \cos \theta}{d}} d \theta \right)^d.\]
If $d_1 \leq d_2$, then the $L^{\frac{1}{d_2}}$ norm is no larger than the $L^{\frac{1}{d_1}}$ norm, hence
\[ G_{d_1} (o) \geq \int_0^\infty e^{-t} I_0 \Big(\frac{t}{d_2}\Big)^{d_2} d t = G_{d_2} (o).\]
\end{proof}

Note that $G_d (o) = (1-p_d)^{-1}$, where $p_d$ is the probability that $d$-dimensional random walk from the origin returns to the origin \cite[Equation 3.15]{popov2021}. Since $p_4 \leq 0.2$ \cite{finch2003}, Lemma~\ref{g dec} implies a bound of $G_d (o) \leq \frac54$ when $d \geq 4$. We state this fact as a lemma, because we will use it several times.

\begin{lemma}\label{g5}
If $d \geq 4$, then $G_d (o) \leq \frac54$.
\end{lemma}

We use these monotonicity properties to prove simple lower bounds on escape probabilities and harmonic measure for sets with at most four elements.
\begin{lemma}\label{min es}
Let $d \geq 4$. If $U \subset \Z^d$ has $|U| \leq 4$, then
\begin{equation}\label{eq min es2}
\Es_U (x) \geq \frac{4 - 3 G(o)}{G(o)} \geq \frac15 \quad \text{and} \quad \H_U (x) \geq \frac{1}{16}, \quad x \in U.
\end{equation}
\end{lemma}

\begin{proof}
Let $x \in U$. We can express $\Es_U (x)$ in terms of $N = \sum_{j=1}^\infty \1 (S_j \in U)$, the number of returns made to $U$ by random walk, as 
\begin{equation}\label{esv bd1}
\Es_U (x) = \P_x (N=0) = 1 - \frac{\E_x [N]}{\E_x [N \mid N > 0]}.
\end{equation}

To bound above $\E_x [N]$, we note that
\begin{equation}\label{esv bd1a}
\E_x [N] = \sum_{y \in U} G(x-y) - 1.
\end{equation}
Since $x \in U$ and $|U| \leq 4$, there are $u, v, w\in \Z^d$ such that
\[
\{x - y: y \in U\} \subseteq \{o, u, v, w\} \quad \text{and} \quad \exists k,l, m \in \{1,\dots,d\}: u \succeq e_k, \,\, v \succeq e_l, \,\, w \succeq e_m.
\]
Green's function is nonnegative and nonincreasing in $\succeq$ (Lemma \ref{gpart}), hence
\begin{equation*}
\sum_{y \in U} G(x-y) \leq G(o) + G(u) + G(v) + G(w) \leq G(o) + G(e_k) + G(e_l) + G(e_m).
\end{equation*}
This bound is at most $4 G(o) - 3$ because $G(e_i) = G(o) - 1$ for every $1 \leq i \leq d$ (see, e.g., \cite[Exercise 3.2]{popov2021}). By \eqref{esv bd1a}, $\E_x [N]$ is at most $4 G(o) - 4$.

To bound below $\E_x [N \mid N>0]$, we use the fact that $\E_z [N] \geq G(o) - 1$ for every $z \in U$, since $N$ is at least the number of times that random walk returns to $z$. Consequently, 
\begin{align*}
\E_x [N \mid N > 0] & = 1 + \E_x \big[ \E_{S_{\tau_{U}}} [N] \bigm\vert \tau_{U} < \infty \big] \geq G(o).
\end{align*}

We substitute the preceding bounds into \eqref{esv bd1} to find that
\[
\Es_U (x) \geq 1 - \frac{4 G(o) - 4}{G(o)} = \frac{4 - 3G(o)}{G(o)}.
\]

To bound below the harmonic measure of $x$ in $U$, we note that $\capac_U$ is at most $4G(o)^{-1}$ because $|U| \leq 4$ and because $\Es_U (y)$ is at most $G(o)^{-1}$ for $y \in U$. We combine this fact with the preceding display to find that
\[
\H_U (x) = \frac{\Es_U (x)}{\capac_U} \geq \frac{4 - 3G(o)}{4}.
\]

Since the preceding lower bounds decrease with $G(o)$ and since $G(o) \leq \frac54$ for $d \geq 4$ (Lemma~\ref{g5}), we conclude that $\Es_U (x) \geq \frac15$ and $\H_U (x) \geq \frac{1}{16}$.
\end{proof}

We use Lemma~\ref{simple set hit} to extend the escape probability lower bound to configurations in which four or fewer elements are sufficiently far from the rest.

\begin{lemma}\label{min hm}
Let $d \geq 4$. There is a constant $c = c(d)$ such that, if $U \in \Conf_d$ can be partitioned into $\mathcal{P} = (\mathcal{P}^1, \mathcal{P}^2)$ such that $|\mathcal{P}^1| \leq 4$ and $\sep (\mathcal{P}) \geq \kappa$, then 
\begin{equation*}
\Es_U (x) \geq \left( 1 - c |U| \sep (\mathcal{P})^{2-d} \right) \Es_{\mathcal{P}^1} (x), \quad x \in \mathcal{P}^1.
\end{equation*}
\end{lemma}

\begin{proof}
Let $U$ and $\mathcal{P}$ satisfy the hypotheses and let $x \in \mathcal{P}^1$. We bound below $\Es_U (x)$ as
\begin{align*}
\Es_U (x) &= \Es_{\mathcal{P}^1} (x) - \P_x (\tau_{\mathcal{P}^1} = \infty, \tau_{\mathcal{P}^2} < \infty)\\ 
&\geq \left(1 - \frac{\P_x (\tau_{\mathcal{P}^2} < \infty)}{\Es_{\mathcal{P}^1} (x)} \right) \Es_{\mathcal{P}^1} (x) \geq \left(1 - O \left( |U| \sep (\mathcal{P}^{2-d}) \right) \right) \Es_{\mathcal{P}^1} (x).
\end{align*}
The first inequality follows from dropping the sub-event $\{\tau_{\mathcal{P}^1} = \infty\}$; the second uses the bound $\Es_{\mathcal{P}^1} (x) \geq \frac15$ from Lemma~\ref{min es}, which applies because $|\mathcal{P}^1| \leq 4$, and equation \eqref{eq simple set hit ub} of Lemma~\ref{simple set hit}, which applies because $\sep (\mathcal{P}) \geq \kappa$.
\end{proof}

We also note a simple lower bound of the harmonic measure of partitions with parts of four or fewer elements, which follows from Lemma~\ref{min es}.

\begin{lemma}\label{hm small parts}
Let $d \geq 4$. If $U \in \Conf_d$ can be partitioned into $\mathcal{P} = (\mathcal{P}^i)_{i=1}^k$ such that $|\mathcal{P}^i| \leq 4$ for every $1 \leq i \leq k$, then
\[
\H_{\mathcal{P}} ([x]_{\mathcal{P}},x) \geq \frac{1}{16 |U|}, \quad x \in U.
\] 
\end{lemma}

\begin{proof}
Let $x \in U$ and denote $i = [x]_{\mathcal{P}}$. By \eqref{hm interp},
\[
\H_{\mathcal{P}} (i,x) = \frac{\capac_{\mathcal{P}^i}}{\capac_{\mathcal{P}}} \H_{\mathcal{P}^i} (x).
\]
The claimed bound follows from the observation that $\capac_{\mathcal{P}} \leq |U| \capac_{\mathcal{P}^i}$  and the harmonic measure lower bound $\H_{\mathcal{P}^i} (x) \geq \frac{1}{16}$ of Lemma~\ref{min es}, which applies because $|\mathcal{P}^i| \leq 4$.
\end{proof}

The next result applies Lemma~\ref{simple set hit} and Lemma~\ref{min es} to obtain two estimates that we need for the transport comparison.

\begin{lemma}\label{lemma help est}
Let $d \geq 4$. There is a constant $c = c(d)$ such that, if $U \in \Conf_d$ can be partitioned into $\mathcal{P} = (\mathcal{P}^1, \mathcal{P}^2)$ such that $|\mathcal{P}^1| \leq 3$ and $\sep (\mathcal{P}) \geq \kappa$, then, for every $x \in \mathcal{P}^1$ and every $y$ in the exterior boundary of $A = \mathcal{P}^1 {\setminus} \{x\}$, 
\begin{align}
\P_x (S_{\tau_A - 1} = y, \tau_{\mathcal{P}^2} < \tau_A < \infty) &\leq c |U| \sep (\mathcal{P})^{4-2d}, \quad \quad \text{and} \label{help est1}\\ 
\P_x (S_{\tau_A - 1} = y, \tau_A < \infty) &\geq c \diam (A \cup \{y\})^{2-d}. \label{help est2}
\end{align}
\end{lemma}  

For the event in \eqref{help est1} to occur, random walk must go from $\mathcal{P}^1$ to $\mathcal{P}^2$ and then from $\mathcal{P}^2$ to $\mathcal{P}^1$. By Lemma~\ref{simple set hit}, it does so with probabilities of roughly $O\left(|\mathcal{P}^1| \sep (\mathcal{P})^{2-d} \right)$ and then $O\left( |\mathcal{P}^2| \sep (\mathcal{P})^{2-d} \right)$. To prove \eqref{help est2}, we consider the event that random walk from $x$ first escapes $B = A \cup \{y\}$ to a distance of roughly $\diam (B)$ away, before returning to $B$ at $y$. The virtue of this event is that the conditional hitting distribution of $B$ from a distance of roughly $\diam (B)$ away is comparable to the harmonic measure of $B$, which is bounded below by a constant due to Lemma~\ref{min es}. The lower bound \eqref{help est2} is due to the fact that random walk from a distance of $\diam (B)$ from $B$ returns to $B$ with a probability of $\Omega (\diam (B)^{2-d})$ by Lemma~\ref{simple set hit}.

The assumptions on $\mathcal{P}$ ensure that Lemma~\ref{simple set hit} and Lemma~\ref{min es} are applicable. In particular, we assume that $|\mathcal{P}^1| \leq 3$ instead of $|\mathcal{P}^1| \leq 4$ so that $B \cup \{x\} = \mathcal{P}^1 \cup \{y\}$ has at most four elements. This enables our use of Lemma~\ref{min es} to bound below $\Es_{B \cup \{x\}} (x)$ by a constant, which further bounds below the probability that random walk escapes to a distance of roughly $\diam (B)$ away from $B$.

\begin{proof}[Proof of Lemma~\ref{lemma help est}]
Let $U$ and $\mathcal{P}$ satisfy the hypotheses, and let $x \in \mathcal{P}^1$ and $y \in \bdy A$. We bound the probability in \eqref{help est1} as
\begin{align*}
\P_{x} (S_{\tau_A - 1} = y, \tau_{\mathcal{P}^2} < \tau_A < \infty)
&\leq \P_x (\tau_{\mathcal{P}^2} < \tau_A < \infty)\\ 
&\leq \P_x (\tau_{\mathcal{P}^2} < \infty) \max_{z \in \mathcal{P}^2} \P_z ( \tau_A < \infty) \lesssim |U| \sep (\mathcal{P})^{4-2d}.
\end{align*}
The first inequality follows from dropping the sub-event $\{S_{\tau_A-1} = y\}$; the second holds because $\max_{z \in \mathcal{P}^2} \P_z (\tau_A < \infty)$ bounds above the conditional hitting probability of $A$ from an arbitrarily distributed $Z$ in $\mathcal{P}^2$; the third follows from two applications of Lemma~\ref{simple set hit}, which applies because $\sep (\mathcal{P}) \geq \kappa$, and the fact that $|\mathcal{P}^1| |\mathcal{P}^2| \leq 3 |U|$.

We bound below the probability in \eqref{help est2} using the fact (see, e.g., \cite[Exercise 3.25]{popov2021}) that there is a constant $\lambda > 1$ such that, if the distance between $u \in \Z^d$ and a finite set $B \subset \Z^d$ is at least $\lambda \diam (B)$, then 
\begin{equation}\label{ex325}
\P_u (S_{\tau_B} = v \mid \tau_B < \infty) \geq \frac12 \H_B (v), \quad v \in B.
\end{equation}
To this end, let $B = A \cup \{y\}$ and $C = \{u \in \Z^d: \dist (u, B) \leq \lambda \diam (B)\}$, and note that
\begin{equation}\label{hm like bd}
\P_x (S_{\tau_A - 1} = y, \tau_A < \infty) \geq \frac{1}{2d} \P_x (\tau_{\bdy C} < \tau_B < \infty, S_{\tau_B} = y).
\end{equation}
In words, the probability that the random walk steps from $y$ when it first returns to $A$ is at least the probability that it does so after hitting $\bdy C$. Note that the factor of $\frac{1}{2d}$ addresses the step that the random walk takes from $y$ into $A$ to ensure that $\tau_B = \tau_A - 1$.  
The lower bound \eqref{hm like bd} factors into
\[
\frac{1}{2d} \P_x (\tau_{\bdy C} < \tau_B) \P_x (\tau_B < \infty \mid \tau_{\bdy C} < \tau_B) \P_x (S_{\tau_B} = y \mid \tau_{\bdy C} < \tau_B < \infty).
\]
We address these factors in turn.

First, the probability that random walk hits $\bdy C$ before returning to $B$ is at least the probability of escape from $B \cup \{x\}$, which, by Lemma~\ref{min es}, is at least $\frac{1}{5}$ because $|B \cup \{x\}| \leq |\mathcal{P}^1| + 1 \leq 4$:
\[
\P_x (\tau_{\bdy C} < \tau_B ) \geq \P_x (\tau_B = \infty) \geq \Es_{B \cup \{x\}} (x) \geq \frac15.
\]
Note that we use $\Es_{B \cup \{x\}} (x)$ instead of $\Es_B (x)$ because if $x \neq y$, then $x \notin B$, hence $\Es_B (x) = 0$ by definition and the bound of Lemma~\ref{min es} does not apply. 

Second, given that the random walk hits $\bdy C$ before $B$, the conditional probability that it subsequently hits $B$ satisfies
\[
\P_x (\tau_B < \infty \mid \tau_{\bdy C} < \tau_B) \geq \min_{z \in \bdy C} \P_z (\tau_B < \infty) \gtrsim \diam (B)^{2-d}.
\]
The second inequality holds by Lemma~\ref{simple set hit}, which implies that the hitting probability of $B$ from $z \in \bdy C$ is $\Omega (\dist(z, B)^{2-d})$, and by the definition of $C$, which implies that $\dist (z, B) \lesssim \diam (B)$ for every such $z$.

Third, given that the random walk hits $\bdy C$ and then $B$, the conditional hitting probability of $y$ is at least 
\[
\P_x (S_{\tau_B} = y \mid \tau_{\bdy C} < \tau_B < \infty) \geq \min_{z \in \bdy C} \P_z (S_{\tau_B} = y \mid \tau_B < \infty) \geq \frac12 \H_B (y) \geq \frac{1}{32}.
\]
The second inequality holds by \eqref{ex325}; the third by Lemma~\ref{min es}, which applies because $|B| \leq 3$.

We multiply the three preceding bounds to conclude that
\[
\frac{1}{2d} \P_x (\tau_{\bdy C} < \tau_B < \infty, S_{\tau_B} = y) \gtrsim \diam (B)^{2-d}.
\]
The bound \eqref{help est2} follows from substituting this into \eqref{hm like bd}. 
\end{proof}

\section{Comparison of activation and transport components}\label{comp act}

First, we compare the activation components of HAT and IHAT. Recall $\kappa$, the constant associated with an estimate of Green's function \eqref{eq green asymp}, and our convention that a partition must have at least two parts.

\begin{proposition}[Activation comparison]\label{prop h comp}
Let $d \geq 4$. There is a constant $c = c (d)$ such that, if $U \in \Conf_d$ can be partitioned into $\mathcal{P} = (\mathcal{P}^i)_{i=1}^k$ such that $|\mathcal{P}^i| \leq 3$ for every $1 \leq i \leq k$ and $\sep (\mathcal{P}) \geq \kappa$, then 
\begin{equation}\label{hm for laycon}
\H_U (x) \geq \big(1 - c |U| \sep(\mathcal{P}) ^{2-d} \big) \H_{\mathcal{P}} ([x]_{\mathcal{P}},x), \quad x \in U.
\end{equation}
\end{proposition}

The proof applies the escape probability lower bound of Lemma~\ref{min hm} to the partition of $U$ consisting of the part $[x]_{\mathcal{P}}$ to which $x$ belongs and the union of the other parts.

\begin{proof}[Proof of Proposition \ref{prop h comp}]
Let $U$ and $\mathcal{P}$ satisfy the hypotheses, and let $x \in U$ and $i = [x]_\mathcal{P}$. Define the coarser partition $\CC = (\CC^1, \CC^2)$ by $\CC^1 = \mathcal{P}^i$ and $\CC^2 = \mathcal{P}^{\neq i}$, which satisfies $|\CC^1| \leq 3$ and $\sep (\CC) \geq \kappa$. By Lemma~\ref{min hm} and the fact that $\sep (\CC) \geq \sep (\mathcal{P})$, there is a constant $c = c(d)$ such that
\[
\Es_U (x) \geq \big(1 - c |U| \sep (\CC)^{2-d} \big) \Es_{\mathcal{P}^i} (x) \geq \big(1 - c |U| \sep (\mathcal{P})^{2-d} \big) \Es_{\mathcal{P}^i} (x).
\]
We divide by $\capac_U$ to bound below $\H_U (x)$:
\[
\H_U (x) = \frac{\Es_U (x)}{\capac_U} \geq \big(1 - c |U| \sep (\mathcal{P})^{2-d} \big) \frac{\Es_{\mathcal{P}^i} (x)}{\capac_U} \geq \big(1 - c |U| \sep (\mathcal{P})^{2-d} \big) \H_{\mathcal{P}} (i,x).
\]
The second inequality uses the bound $\capac_U \leq \capac_\mathcal{P}$, which follows from the fact that $\Es_U (y) \leq \Es_{\mathcal{P}^j} (y)$ for every $1 \leq j \leq k$ and $y \in \mathcal{P}^j$, and the definition of $\H_{\mathcal{P}} (i,x)$.
\end{proof}

Next, we compare the transport components of HAT and IHAT. Recall that a partition $\mathcal{P}$ of a configuration $U$ is an $(a,b)$ clustering if $\sep (\mathcal{P}) \geq a$ and if the diameter of each $\mathcal{P}^i$ is at most $b \log \dist (\mathcal{P}^i, \mathcal{P}^{\neq i})$ (Definition~\ref{clustering def}). If $\mathcal{P}$ also satisfies $|\mathcal{P}^i| \in \{2,3\}$ for each $i$, then it is an $(a,b)$ DOT clustering.

\begin{proposition}[Transport comparison]\label{prop t comp}
Let $a \geq \kappa$, $b > 0$, and $d \geq 4$. There is a constant $c = c(b,d)$ such that, if $U \in \Conf_d$ has an $(a,b)$ DOT clustering $\mathcal{P}$, 
then, for every $x \in U$ and every $y$ in the exterior boundary of $A = \mathcal{P}^{[x]_{\mathcal{P}}} {\setminus} \{x\}$, 
\begin{equation*}
\P_x (S_{\tau_{U {\setminus} \{x\}}-1} = y \,\vert\, \tau_{U {\setminus} \{x\}} < \infty) \geq \left(1 - c |U| \left( \frac{\log \sep (\mathcal{P})}{\sep (\mathcal{P})} \right)^{d-2} \right) \P_x (S_{\tau_A-1} = y \,\vert\, \tau_A < \infty).
\end{equation*}
\end{proposition}

Proposition~\ref{prop t comp} only considers instances of intracluster transport---instances in which $y$ belongs to the exterior boundary of the part $[x]_{\mathcal{P}}$ to which $x$ belonged before it was removed---because the IHAT transport component would be zero otherwise. The proof revisits the equation \eqref{ict hat} and applies Lemma~\ref{simple set hit} and Lemma~\ref{help est2} to bound below its third and fourth factors.

\begin{proof}[Proof of Proposition \ref{prop t comp}]
Let $U$ and $\mathcal{P}$ satisfy the hypotheses, let $x \in U$ and $y \in \bdy A$, and let $i = [x]_{\mathcal{P}}$. Define the coarser partition $\CC = (\CC^1, \CC^2)$ by $\CC^1 = \mathcal{P}^i$ and $\CC^2 = \mathcal{P}^{\neq i}$, which satisfies $A \subset \CC^1$, $|\CC^1| \leq 3$, and $\sep (\CC) \geq \kappa$. For brevity, denote $\tau = \tau_{U \setminus\{x\}}$.

Following \eqref{ict hat}, the probability in question can be expressed as 
\begin{multline}\label{inc hat1}
\P_x (S_{\tau-1} = y \,|\, \tau < \infty) = \P_x (S_{\tau_A-1} = y \,|\, \tau_A < \infty)\\ 
\times \underbrace{\frac{\P_{x} (\tau_A < \infty)}{\P_{x} (\tau < \infty)}}_{( \ref{inc hat1} \text{a})} \Bigg( 1
- \underbrace{\frac{\P_{x} (S_{\tau_A-1} = y, \tau_{\CC^2} < \tau_A < \infty)}{\P_{x} (S_{\tau_A-1} =y, \tau_A < \infty)}}_{( \ref{inc hat1} \text{b})} \Bigg).
\end{multline}
To bound below (\ref{inc hat1}a), note that $\tau = \min \{\tau_A, \tau_{\CC^2}\}$, so lower and upper bounds of $\P_x (\tau_A < \infty) \geq p_1$ and $\P_x (\tau_{\CC^2} < \infty) \leq p_2$ for some $p_1, p_2 \in (0,1]$ would imply that
\begin{equation}\label{618a bd}
\frac{\P_{x} (\tau_A < \infty)}{\P_{x} (\tau < \infty)} \geq \frac{\P_{x} (\tau_A < \infty)}{\P_{x} (\tau_A < \infty) + \P_x (\tau_{\CC^2} < \infty)} \geq 1 - \frac{p_2}{p_1}.
\end{equation} 
By Lemma~\ref{simple set hit}, the fact that $\dist (x,A) \leq \diam (\CC^1)$, and the assumption that $\mathcal{P}$ is a $(\cdot,b)$ clustering, 
\[
\P_x (\tau_A < \infty) \gtrsim \dist(x, A)^{2-d} \gtrsim \diam (\CC^1)^{2-d} \gtrsim \left( b \log \sep (\CC) \right)^{2-d}.
\]
By \eqref{eq simple set hit ub} of Lemma~\ref{simple set hit}, which applies because $\dist(x,\CC^2) \geq \sep (\CC) \geq \kappa$, and the fact that $|\CC^2| \leq |U|$, 
\[
\P_{x} (\tau_{\CC^2} < \infty) \lesssim |U| \sep (\CC)^{2-d}.
\]
We take $p_1$ and $p_2$ to be the preceding lower and upper bounds (including the implicit constants). Substituting these choices into \eqref{618a bd} yields

\begin{equation*}
\frac{\P_{x} (\tau_A < \infty)}{\P_{x} (\tau < \infty)} \geq 1 - c_1 b^{d-2} |U| \left( \frac{\log \sep (\CC)}{\sep (\CC)} \right)^{d-2},
\end{equation*}
for a constant $c_1 = c_1 (d)$. 
To bound above (\ref{inc hat1}b), we use Lemma~\ref{lemma help est} with $\CC$ in the place of $\mathcal{P}$. The lemma is applicable because $|\CC^1| \leq 3$ and $\sep (\CC) \geq \kappa$. It implies that
\[
\frac{\P_{x} (S_{\tau_A-1} = y, \tau_{\CC^2} < \tau_A < \infty)}{\P_{x} (S_{\tau_A-1} =y, \tau_A < \infty)} \lesssim |U| \left( \frac{\diam(A \cup \{y\})}{\sep (\CC)^2} \right)^{d-2} \leq c_2 b^{d-2} |U| \left( \frac{\log \sep (\CC)}{\sep (\CC)^2} \right)^{d-2}, 
\]
for a constant $c_2 = c_2 (d)$. 
The second inequality holds because the diameter of $A \cup \{y\}$ is at most $\diam (\CC^1) + 1$, since $y \in \bdy A$ and $A \subset \CC^1$, and because the diameter of $\CC^1 = \mathcal{P}^i$ is at most $b \log \sep (\CC)$, since $\mathcal{P}$ is a $(\cdot,b)$ clustering.

We substitute the bounds on (\ref{inc hat1}a) and (\ref{inc hat1}b) into \eqref{inc hat1} to conclude that
\[
\P_x (S_{\tau-1} = y \,\vert\, \tau < \infty) \geq \left(1 - c b^{d-2} |U| \left( \frac{\log \sep (\CC)}{\sep (\CC)} \right)^{d-2} \right) \P_x (S_{\tau_A-1} = y \,\vert\, \tau_A < \infty),
\] 
where $c = 2 \max\{c_1,c_2\} b^{d-2}$. This implies the claimed bound because $\frac{\log r}{r}$ decreases with $r > e$ and $\sep (\CC)$ is at least $\sep (\mathcal{P}) \geq \kappa > e$.

\end{proof}

\section{Approximation of HAT by IHAT}\label{sec comparison}

Informally, the estimates of Section~\ref{comp act} show that the ``error'' of approximating HAT by IHAT over one step, for a transition that involves intracluster transport, is $O( |U| (\frac{\log \rho}{\rho})^{d-2})$, in terms of the separation $\rho$ of a suitable partition of the current configuration $U$. This suggests that we can use IHAT to estimate the probabilities of events that entail sufficiently rapid separation growth for the natural partitioning of HAT. The main result of this section formalizes this idea.

Fix $d \geq 5$ and $n \geq 4$. Recall that, for $a, b> 0$, and a configuration $U \in \Conf_{d,n}$, we use $\Clust_{a,b}^\bullet (U)$ to denote the (possibly empty) collection of $(a,b)$ DOT clusterings of $U$. We denote the collection of $(a,b)$ DOT clusterings of $n$-element configurations in $\Z^d$ by $\Clust_{a,b}^\bullet = \cup_{U \in \Conf_{d,n}} \Clust_{a,b}^\bullet (U)$. For $\delta \in (0,\frac12)$, we define the following set of infinite sequences of $(a,b)$ DOT clusterings:
\begin{equation}\label{event grow}
\Grow_{a,b,\delta} = \left\{ (\CC_t)_{t \geq 0} \in (\Clust_{a,b}^\bullet)^\N: \sep (\CC_t) \geq t^{\frac12-\delta} \,\,\text{for every $t \geq 0$} \right\}.
\end{equation}

Recall the definition of the natural partitioning of HAT (Definition~\ref{nat part}). The main result of this section states that, if $a$ is sufficiently large, then the probability that the natural partitioning $(\UU_t)_{t \geq 0}$ of $(U_t)_{t \geq 0}$ with an $(a,b)$ DOT clustering is a sequence of clusterings in $\Grow_{a,b,\delta}$ is at least the analogous probability for IHAT, up to a factor of $(1- o_a (1))$. Given a clustering $\CC \in \Clust_{a,b}^\bullet$, we will use $\PP_\CC$ to denote the distribution of HAT from the configuration $\cup_i \CC^i$. Under $\PP_\CC$, we will use $(\UU_t)_{t \geq 0}$ to denote the natural partitioning of $(U_t)_{t \geq 0}$ with $\CC$. Recall that $\QQ_\CC$ denotes the law of IHAT from $\CC$.

\begin{proposition}[Main approximation result]\label{pe to qe}
Fix $d \geq 5$ and $n \geq 4$. Let $b > 0$ and let $\delta \in (0, \frac12 - \frac{1}{d-2})$. For every $\varepsilon \in (0,1)$, there is a constant $\alpha = \alpha (b,d,n,\delta,\eps) > 1$ such that, if $a \geq \alpha$, then, for every $\CC \in \Clust_{a,b}^\bullet$, 
\[
\PP_{\CC} \left( (\UU_t)_{t \geq 0} \in \Grow_{a,b,\delta} \right) \geq (1 - \varepsilon) \QQ_{\CC} \left( (\VV_t)_{t \geq 0} \in \Grow_{a,b,\delta} \right).
\]
\end{proposition}

Proposition~\ref{pe to qe} is one of three inputs to the proof of Proposition~\ref{inf xi}. The other two inputs state that (i) the event which Proposition~\ref{inf xi} concerns---that the natural clustering grows in separation as Theorem~\ref{fate} predicts---is a subset of $\{(\VV_t)_{t \geq 0} \in \Grow_{a,b,\delta}\}$ for appropriate choices of $a$, $b$, and $\delta$, and that (ii) its probability under IHAT is positive. We prove these other inputs in the next two sections. In the remainder of this section, we prove Proposition~\ref{pe to qe}.

The key to the proof of Proposition~\ref{pe to qe} is a one-step approximation of HAT by IHAT.

\begin{proposition}[One-step approximation]\label{p to q}
Fix $d \geq 5$ and $n \geq 4$, and let $a,b > 0$. There are constants $\alpha_0 = \alpha_0 (b,d,n)$ and $c = c(b,d,n)$ such that, if $a \geq \alpha_0$, then, for every $\CC_0, \CC_1 \in \Clust_{a,b}^\bullet$,
\begin{equation}\label{eq p to q}
\PP_{\CC_0} (\UU_1 = \CC_1) \geq \left( 1 - c \left(a^{-1} \log a\right)^{d-2} \right) \QQ_{\CC_0} (\VV_1 = \CC_1).
\end{equation}
\end{proposition}

Proposition~\ref{pe to qe} is a consequence of a short calculation with Proposition~\ref{p to q}.

\begin{proof}[Proof of Proposition~\ref{pe to qe}]
Let $\alpha_0$ and $c$ be the constants from Proposition~\ref{p to q} and assume that $a \geq \alpha_0$. For $t \geq 0$, denote 
\[
g(t) = \max\left\{a, t^{\frac12 - \delta}\right\} \quad \text{and} \quad h (t) = \prod_{s=0}^{t-1} \left(1 - c \left( g(s)^{-1} \log g(s) \right)^{d-2} \right).
\]
If $(\DD_s)_{s \geq 0} \in \Grow_{a,b,\delta}$, then $g(t)$ is a lower bound on the separation of $\DD_t$, so the Markov property and the one-step approximation (Proposition~\ref{p to q}) imply that $h(t)$ satisfies
\begin{equation}\label{multistep h}
\PP_{\CC_0} \left( \UU_s = \DD_s, s \leq t \right) \geq h(t) \QQ_{\CC_0} \left( \VV_s = \DD_s, s \leq t \right).
\end{equation}
If $\delta < \frac12 - \frac{1}{d-2}$, then $h(t) \geq 1 - o_a (1)$. Indeed, some simple but tedious algebra shows that, due to the condition on $\delta$, there are constants $\eta = \eta (d,\delta) > 0$ and $\alpha_1 = \alpha_1 (d, \eta)$ such that, if $a \geq \alpha_1$, then
\[
c\sum_{s = 0}^\infty \left( \frac{\log g(s)}{g(s)} \right)^{d-2} \leq O \left(a^{-\eta}\right).
\]
Here, the implicit constant depends on all parameters: $b$, $d$, $n$, and $\delta$. This bound implies that $h(t)$ decreases to a limit of $h_\infty = 1 - O (a^{-\eta})$ as $t \to \infty$.  

We apply \eqref{multistep h} to conclude that
\begin{align*}
\PP_{\CC_0} \left( (\UU_s)_{s \geq 0} \in \Grow_{a,b,\delta} \right)
&= \sum \PP_{\CC_0} \left( \UU_s = \DD_s, s \geq 0 \right)\\ 
&= \sum \lim_{t \to \infty} \PP_{\CC_0} \left( \UU_s = \DD_s, 0 \leq s \leq t\right)\\ 
&\geq h_\infty \sum \lim_{t \to \infty} \QQ_{\CC_0} \left( \VV_s = \DD_s, 0 \leq s \leq t\right)\\ 
&= h_\infty \sum \QQ_{\CC_0} \left( \VV_s = \DD_s, s \geq 0\right)\\ 
&= h_\infty \QQ_{\CC_0} \left( (\VV_s)_{s \geq 0} \in \Grow_{a,b,\delta} \right), 
\end{align*}
where the sums range over $(\DD_s)_{s \geq 0} \in \Grow_{a,b,\delta}$. 
\end{proof}

The proof of Proposition~\ref{p to q} is primarily a straightforward application of Propositions~\ref{prop h comp} and \ref{prop t comp}, which compare the activation and transport components of HAT and IHAT. However, to apply these results, we must also show that the possible transitions of IHAT are equivalent to possible transitions for the natural partitioning of HAT, when the transition is between two, well separated clusterings. This is the focus of the next subsection.

\subsection{An equivalence between the possible transitions of HAT and IHAT}

We need a geometric lemma that is due to Kesten (Lemma 2.23 of \cite{kesten1986}; alternatively, Theorem~4 of \cite{timar2013}). To state it, denote by $\Z^{d\ast}$ the graph with vertex set $\Z^d$ and with an edge between distinct $x, y \in \Z^d$ when $x$ and $y$ differ by at most one in each coordinate. For $A \subseteq \Z^d$, we define the $\ast$-visible boundary $\bdyvis A$ as
\begin{multline}\label{bdyvis} \bdyvis A = \big\{x \in \Z^d: \text{$x$ is adjacent in $\Z^{d\ast}$ to some $y \in A$}\\ \text{and there is a path from $\infty$ to $x$ disjoint from $A$} \big\}.
\end{multline}

\begin{lemma}[Lemma 2.23 of \cite{kesten1986}]\label{kesten result}
If $A$ is a finite, $\ast$-connected subset of $\Z^d$, then $\bdyvis A$ is connected in $\Z^d$. 
\end{lemma}

For a finite set $A \subset \Z^d$ and $x \in A$, we say that $x$ is exposed in $A$ if $\Es_A (x) > 0$. The following proposition states that, if a clustering $\CC$ is sufficiently separated and if an element is exposed in $\CC^i$, then that element is also exposed in $\cup_j \CC^j$.

\begin{proposition}\label{help for at pairs}
Fix $d \geq 5$ and $n \geq 4$. Let $a, b > 0$, and let $\CC \in \Clust_{a,b}^\bullet$. If $a \geq \max\{e^{2b},8\sqrt{d}\}$, then
\begin{align*}
\Es_{\CC^i} (x) > 0 \implies \Es_{\cup_j \CC^j} (x) > 0, \quad x \in \Z^d.
\end{align*}
\end{proposition}

This conclusion is not surprising as, when the separation between clusters is large relative to their diameters, the clusters cannot surround another cluster so as to disconnect it from $\infty$. The proof uses Lemma~\ref{kesten result} to identify a path from an element that is exposed in one cluster to the boundary of a set that contains the union of clusters.

\begin{proof}[Proof of Proposition \ref{help for at pairs}]
Let $x$ be exposed in $\CC^i$ and denote the union of clusters by $U = \cup_j \CC^j$. To prove that $x$ is exposed in $U$, it suffices to show that there is a path from $x$ to the boundary of 
\[
F = \{ z \in \Z^d: \dist (z, U ) \leq 2 \diam ( U ) \},
\]
which otherwise lies outside of $U$. 

Because $x$ is exposed in $\CC^i$, there is a path $\Gamma$ from $x$ to $\bdy F$, which otherwise lies outside of $\CC^i$. We modify $\Gamma$ to obtain a path that otherwise lies outside of the rest of $U$ as well. To this end, let 
\[ 
F_j = \big\{ z \in \Z^d: \dist (z,\CC^j) \leq \diam (\CC^j) \big\}, \quad 1 \leq j \leq \# \CC.
\]
We make use of two facts about the $F_j$.

{\bf Fact 1}.\ Each $F_j$ is finite and $\ast$-connected, so each $\bdyvis F_j$ is connected by Lemma \ref{kesten result}.

{\bf Fact 2}.\ If $\sep (\CC)$ is at least $\max\{e^{2b},8\sqrt{d}\}$, then each $\bdyvis F_j$ is disjoint from $\cup_k F_k$. To see why, note that any element of $\bdyvis F_j$ is within $\sqrt{d}$ of an element of $F_j$, while the distance between distinct $F_j$ and $F_k$ exceeds $\sqrt{d}$:
\begin{align*}
\dist (F_j,F_k) & \geq \dist(\CC^j,\CC^k) - \diam (\CC^j) - \diam (\CC^k)  - 2\sqrt{d}\\ 
& \geq \dist (\CC^j,\CC^k) \Big( 1 - \frac{2b \log \dist (\CC^j,\CC^k) + 2\sqrt{d}}{\dist (\CC^j,\CC^k)} \Big) > \sqrt{d}. 
\end{align*}
The first inequality follows from the triangle inequality; the second from the fact that $\CC$ is a $(\cdot,b)$ clustering; the third from the fact that the ratio is decreasing in $\dist (\CC^j, \CC^k)$, which is at least $\sep (\CC)$, and some simple algebra using the assumption that $\sep (\CC) \geq \max\{e^{2b},8\sqrt{d}\}$.

Assume that $\sep (\CC) \geq \max\{e^{2b},8\sqrt{d}\}$. We will keep the part of $\Gamma$ from $x$ until it first encounters $\bdyvis F_i$, which otherwise avoids $\cup_k F_k$ by assumption. We denote by $J$ the set of labels of the $F_j$ subsequently hit by $\Gamma$. If $J$ is empty, then we are done. Otherwise, let $\ell$ be the label of the first of the $F_j$ that $\Gamma$ hits, and let $\Gamma_u$ and $\Gamma_v$ be the first and last elements of $\Gamma$ which intersect $\bdyvis F_\ell$. By Fact 1, $\bdyvis F_\ell$ is connected, so there is a shortest path $\Lambda$ in $\bdyvis F_\ell$ from $\Gamma_u$ to $\Gamma_v$. We then edit $\Gamma$ to form $\Gamma'$ as
\[ \Gamma' = \big( \Gamma_1, \dots, \Gamma_{u-1}, \Lambda_1, \dots, \Lambda_{|\Lambda|}, \Gamma_{v+1}, \dots, \Gamma_{|\Gamma|} \big).\] 
Because $\Gamma_{v+1}$ was the last element of $\Gamma$ which intersected $\bdyvis F_\ell$, $\Gamma'$ avoids $F_\ell$. By Fact 2, $\Lambda$ avoids $\cup_k F_k$, so if $J'$ is the set of labels of $F_j$ encountered by $\Gamma'$, then $|J'| \leq |J|-1$.

If $J'$ is empty, then we are done. Otherwise, we can relabel $\Gamma$ to $\Gamma'$ and $J$ to $J'$ in the preceding argument to continue inductively, obtaining $\Gamma''$ and $|J''| \leq |J|-2$, and so on. Because $|J| \leq |U|$, we need to modify the path at most $|U|$ times before the resulting path to $\bdy F$ does not return to $\cup_k F_k$ after reaching $\bdyvis F_i$. In summary, we edited $\Gamma$ to obtain a path from $x$ to $\bdyvis F_i$ and then from $\bdyvis F_i$ to $\bdy F$, which otherwise avoids $U$. We conclude that $x$ is exposed in $U$.
\end{proof}

The next result is an equivalence between the possible transitions of HAT and IHAT. It states that, if $a$ is large enough in terms of $b$ and $d$, then a transition between two $(a,b)$ DOT clusterings is possible for IHAT if and only if it is possible for the natural partitioning of HAT. We state the result in terms of $p_{U_t} (x,y)$ and $q_{\VV_t} (i,x,y)$, defined in \eqref{eq tps0} and \eqref{eq tps q}.

\begin{proposition}[Equivalence between possible transitions]\label{at pairs}
Fix $d \geq 5$ and $n \geq 4$. Let $a,b > 0$, let $\CC_0, \CC_1 \in \Clust_{a,b}^\bullet$, and denote $U_0 = \cup_j \CC_0^j$. If $a$ is at least $\max\{e^{2b},8\sqrt{d}\}$,  
then, for any $x \in U_0$ and $y \in \Z^d$ such that 
\begin{equation}\label{c0 to c1 equiv}
\CC_1 = (\CC_0 \sd{[x]_{\CC_0}} \{x\}) \cup^{[x]_{\CC_0}} \{y\},
\end{equation}
the following equivalence holds
\begin{equation}\label{corresp}
p_{U_0} (x,y) > 0 \iff q_{\CC_0} ([x]_{\CC_0},x,y) > 0.
\end{equation}
\end{proposition}

\begin{proof}[Proof of Proposition \ref{at pairs}]
Let $a,b$ and $\CC_0, \CC_1$ satisfy the hypotheses, let $x \in U_0$ and $y \in \Z^d$ satisfy \eqref{c0 to c1 equiv}, and denote $U_1 = \cup_j \CC_1^j$ and $i = [x]_{\CC_0}$. Observe that 
\begin{subnumcases}{p_{U_0} (x,y) > 0 \iff}
\Es_{U_0} (x) > 0, \label{p.1}\\ 
\Es_{U_1} (y) > 0, \,\,\text{and} \label{p.2}\\ 
y \in \bdy \big( U_0 {\setminus} \{x\} \big). \label{p.3}
\end{subnumcases}

Analogously, 
\begin{subnumcases}{q_{\CC_0} (i,x,y) > 0 \iff}
\Es_{\CC_0^i} (x) > 0, \label{q.1}\\ 
\Es_{\CC_1^i} (y) > 0, \,\,\text{and} \label{q.2}\\ 
y \in \bdy \big( \CC_0^i {\setminus} \{x\} \big). \label{q.3}
\end{subnumcases}
We claim that 
\[
\eqref{p.1} \iff \eqref{q.1}, \,\, \eqref{p.2} \iff \eqref{q.2}, \,\,\text{and}\,\, \eqref{p.3} \iff \eqref{q.3},
\]
which together imply \eqref{corresp}. We address the forward implications first.

\textbf{Forward implications}. Because $x \in \CC_0^i \subseteq U_0$ and $y \in \CC_1^i \subseteq U_1$, 
\[
\Es_{U_0} (x) \leq \Es_{\CC_0^i} (x) \,\,\text{and}\,\, \Es_{U_1} (y) \leq \Es_{\CC_1^i} (y),
\] hence $\eqref{p.1} \implies \eqref{q.1}$ and $\eqref{p.2} \implies \eqref{q.2}$. Next, because $\sep (\CC_0) \geq a > 1$, we have
\[
\bdy \big( U_0 {\setminus} \{x\} \big) = \bdy \Big( \bigcup_i \CC_0^i {\setminus} \{x\} \Big) = \bigcup_i \bdy \big( \CC_0^i {\setminus}\{x\} \big).
\]
Consequently, \eqref{p.3} implies that $y \in \bdy \big( \CC_0^j {\setminus} \{x\} \big)$ for at least one choice of $j$. In fact, $j = i$ is the only choice that works as, otherwise, we would have $\sep (\CC_1) = 1$. We conclude that $\eqref{p.3} \implies \eqref{q.3}$.

\textbf{Reverse implications}. Because $\CC_0$ and $\CC_1$ are $(a,b)$ separated for an $a$ which is at least $\max\{e^{2b},8\sqrt{d}\}$, Proposition \ref{help for at pairs} applies with $\CC_0$ or $\CC_1$ in the place of $\CC$. Using it, we conclude that $\eqref{p.1} \impliedby \eqref{q.1}$ and $\eqref{p.2} \impliedby \eqref{q.2}$. The argument we used for the last forward implication applies in reverse to show that $\eqref{p.3} \impliedby \eqref{q.3}$. 
\end{proof}

\subsection{Proof of the one-step approximation of HAT by IHAT}

The proof is an application of the activation and transport comparisons (Propositions~\ref{prop h comp} and \ref{prop t comp}). We also need Proposition~\ref{at pairs}, to address a minor technical point.

\begin{proof}[Proof of Proposition \ref{p to q}]
Recall the constant $\kappa$ that appears in the statements of Propositions~\ref{prop h comp} and \ref{prop t comp}. Let $b > 0$ and $a \geq \max \{e^{2b},8\sqrt{d}, \kappa\}$, let $\CC_0, \CC_1 \in \Clust_{a,b}^\bullet$, and denote $U_0 = \cup_j \CC_0^j$.

By the definition of the natural partitioning of HAT with $\CC_0$ (Definition~\ref{nat part}),
\begin{equation}\label{sum puo}
\PP_{\CC_0} (\UU_1 = \CC_1) = \sum p_{U_0} (x,y),
\end{equation}
where the sum ranges over $x \in U_0$ and $y \in \Z^d$ that satisfy $\CC_1 = (\CC_0 \sd{[x]_{\CC_0}} \{x\}) \cup^{[x]_{\CC_0}} \{y\}$. According to Proposition~\ref{at pairs}, since $\sep (\CC_0)$ is at least $\max \{e^{2b}, 8 \sqrt{d}\}$, the summand in \eqref{sum puo} is positive if and only if $q_{\CC_0} ([x]_{\CC_0},x,y)$ is. Moreover, since $\CC_0$ is a partition of $U_0$, $q_{\CC_0} (j,x,y) = 0$ for every $j \neq [x]_{\CC_0}$. Hence,
\begin{equation}\label{sum quo}
\QQ_{\CC_0} (\VV_1 = \CC_1) = \sum q_{\CC_0} ([x]_{\CC_0},x,y),
\end{equation}
where the sum ranges over the same $x$ and $y$ as in \eqref{sum puo}.

Since $\CC_0$ is a DOT clustering with a separation of at least $a \geq \kappa$, Propositions~\ref{prop h comp} and \ref{prop t comp} together imply that there is a constant $c = c (b,d,n)$ such that 
\begin{equation}\label{pq comp}
p_{U_0} (x,y) \geq \left(1 - c \left( a^{-1} \log a \right)^{d-2} \right) q_{\CC_0} ([x]_{\CC_0},x,y).
\end{equation}
 We apply this bound to \eqref{sum puo} and then use \eqref{sum quo} to conclude that
\[
\PP_{\CC_0} (\UU_1 = \CC_1) \geq \left(1 - c \left( a^{-1} \log a \right)^{d-2} \right) \QQ_{\CC_0} (\VV_1 = \CC_1).
\]
\end{proof}

\section{A random walk related to cluster separation}\label{ref times}

Proposition \ref{pe to qe} allows us to bound the probability in \eqref{eq inf xi} of Proposition~\ref{inf xi} with the corresponding probability under IHAT. The purpose of this section and Section~\ref{sec clust sep} is to show that this probability is strictly positive under IHAT. Our strategy is to view each pair of DOTs at the consecutive times at which both clusters form line segments parallel to $e_1$. When viewed at these times, the difference between the clusters' elements with the smallest $e_1$ components is a random walk in $\Z^d$, albeit not a simple one. In this section, we define these random walks and prove some preliminary results about them.


\subsection{Definitions}\label{sub rw def}

Fix $d \geq 5$ and $n \geq 4$. Recall that $\Refcon$ denotes the collection of reference dimers and trimers
\[
\Refcon = \left\{ \{x,x+e_1\}: x \in \Z^d\right\} \cup \left\{\{x, x+e_1, x+2e_1\}: x \in \Z^d\right\}
\]
and $\Refcon^\times$ denotes the collection of tuples of such configurations. 
 We inductively define the {\em references times} of IHAT $(\VV_t)_{t \geq 0}$ by $\xi_0 = 0$ and
\[
\xi_m = \inf\left\{t > \xi_{m-1}: \VV_t \in \Refcon^\times \right\}, \quad m \geq 1.
\]

As a representative element of each cluster $\VV_t^i$, we arbitrarily choose the element of $\VV_t^i$ that is least in the lexicographic order on $\Z^d$ and denote it by $M_t^i$. For example, if $\VV_t^i \in \Refcon$, then $M_t^i$ is the element with the least $e_1$ component. We use pairs of these representative elements to define random walks.

For every distinct pair of clusters $i, j \in \{1, \dots, \# \VV_0\}$ with $i < j$, we define a random walk $(S_k^{i,j})_{k \geq 0}$ by 
\begin{equation}\label{rw assoc w}
S_0^{i,j} = M_{\xi_0}^i - M_{\xi_0}^j \quad \text{and} \quad S_k^{i,j} = S_{k-1}^{i,j} + \left( M_{\xi_k}^i - M_{\xi_k}^j - \left( M_{\xi_{k-1}}^i + M_{\xi_{k-1}}^j\right) \right), \quad k \geq 1.
\end{equation}
We refer to the collection $\{(S_k^{i,j})_{k \geq 0}, 1 \leq i < j \leq \# \VV_0\}$ as the random walks associated with IHAT.

\subsection{Cluster separation and the distance between their representative elements}\label{sec sep to dist}

We will need to convert the distance between distinct clusters $\VV_t^i$ and $\VV_t^j$ of IHAT to the distance between their representative elements $M_t^i$ and $M_t^j$, and vice versa. The triangle inequality implies that
\begin{equation}\label{tri app both}
\left| \dist (\VV_t^i, \VV_t^j) - \| M_t^i - M_t^j\| \right| \leq \diam (\VV_t^i) + \diam (\VV_t^j), \quad t \geq 0.
\end{equation}

We can combine \eqref{tri app both} with two basic facts about IHAT. First, at most one cluster diameter and representative element $M_t^i$ change with each step and, when they do, they change by at most $1$ and $\diam (\VV_t^i)$, respectively. In other words, for $t \geq 0$,
\begin{align}
\|M_{t+1}^i - M_t^i\| + \|M_{t+1}^j - M_t^j\| &\leq \max \{ \diam (\VV_t^i), \diam (\VV_t^j) \}, \label{sep by one}\\
\diam (\VV_{t+1}^i) + \diam (\VV_{t+1}^j) &\leq \diam (\VV_t^i) + \diam (\VV_t^j) + 1. \label{diam by one}
\end{align}
We use these facts to prove a further result in the spirit of \eqref{tri app both}.

\begin{lemma}\label{amlg2}
If $t \geq 0$ satsfies $\xi_m \leq t < \xi_{m+1}$, then
\begin{equation}\label{eq amlg2}
\dist (\VV_t^i, \VV_t^j) \geq \big\| S_m^{i,j} \big\| - 6 \big( \xi_{m+1} - \xi_m \big)^2.
\end{equation}
\end{lemma}

\begin{proof}
Assume that $\xi_m \leq t < \xi_{m+1}$. By \eqref{tri app both},
\begin{equation}\label{tri app}
\dist (\VV_t^i, \VV_t^j) \geq \|M_t^i - M_t^j\| - \diam (\VV_t^i) - \diam (\VV_t^j).
\end{equation}
To replace $M_t^i - M_t^j$ by $S_m^{i,j} = M_{\xi_m}^i - M_{\xi_m}^j$ in \eqref{tri app}, we bound the norm of their difference:
\begin{equation}\label{telesc}
M_t^i - M_t^j - (M_{\xi_m}^i - M_{\xi_m}^j) = \sum_{s = \xi_m}^{t-1} \left( M_{s+1}^i - M_{s+1}^j - (M_s^i - M_s^j) \right).
\end{equation}
Taking the norm of both sides and applying \eqref{sep by one} gives
\begin{align*}
\| M_t^i - M_t^j - (M_{\xi_m}^i - M_{\xi_m}^j) \| &\leq \sum_{s = \xi_m}^{t-1} \max \{\diam (\VV_s^i), \diam (\VV_s^j)\}\\ 
&\leq \sum_{s = \xi_m}^{t-1} (2+s-\xi_m) \leq \frac12 (t - \xi_m)^2 \leq \frac12 (\xi_{m+1} - \xi_m)^2.
\end{align*}
The second inequality holds because cluster diameter grows by at most one with each step and because, at time $\xi_m$, the clusters have diameters of at most $2$, since they belong to $\Refcon$. 
The fourth inequality holds because $t < \xi_{m+1}$.

The preceding bound addresses the first term of \eqref{tri app}. To address the other terms, we use \eqref{diam by one}:
\[
\diam (\VV_t^i) + \diam (\VV_t^j) \leq (\xi_{m+1} - \xi_m) + \diam (\VV_{\xi_m}^i) + \diam (\VV_{\xi_m}^j) \leq 5 (\xi_{m+1} - \xi_m).
\] 
The second inequality holds because $\xi_{m+1} - \xi_m \geq 1$ and because, at time $\xi_m$, the clusters have diameters of at most $2$.

Substituting the preceding bounds into \eqref{tri app} gives
\[
\dist (\VV_t^i, \VV_t^j) \geq \big\| S_m^{i,j} \big\| - \frac12 (\xi_{m+1} - \xi_m)^2 - 5 (\xi_{m+1} - \xi_m).
\]
The claimed bound \eqref{eq amlg2} follows from $\xi_{m+1} - \xi_m \geq 1$.
\end{proof}

\subsection{The reference times of IHAT have exponential tails}\label{exp tails}

The next result uses Lemma~\ref{hm small parts} to show that the references times $(\xi_m)_{m \geq 1}$ have exponentially small tails under $\PPprod$.

\begin{lemma}\label{xi tail}
Fix $d \geq 5$ and $n \geq 4$. There is a constant $\gamma_1 = \gamma_1 (d,n) > 0$ such that, if $\CC_0 \in \Conf_{d,n}^\times$ satisfies $|\CC_0^i| \in \{2,3\}$ for every $1 \leq i \leq \#\CC_0$, then 
\begin{equation}\label{eq xi tail}
\QQ_{\CC_0} \big( \xi_1 > t \big) \leq 2e^{-\gamma_1 t}.
\end{equation}
\end{lemma}

This result applies to any tuple of configurations $\CC_0$ in which each configuration has two or three elements. In particular, $\CC_0$ does not need to satisfy a separation lower bound and it does not need to belong to $\Refcon^\times$.

\begin{proof}[Proof of Lemma~\ref{xi tail}]
Lemma~\ref{hm small parts} states that we can activate any element of any cluster of $\VV_t$ with a probability of at least $(16n)^{-1}$, since each cluster has two or three elements under $\QQ_{\CC_0}$. To arrange a cluster of $\VV_t$ into a reference cluster, we first ensure that it is connected, possibly by activating one or two isolated elements, and then we dictate up to two additional IHAT steps to organize the connected cluster into a line segment parallel to $e_1$. After addressing the first cluster of $\VV_t$, we apply the same procedure to the rest of the clusters.

Consider an arbitrary cluster $\VV_t^i$. The following procedure shows that $\VV_{t+4}^i \in \Refcon$ with a probability of at least $p^4$, where $p = (16n)^{-1} (2d)^{-4}$.
\begin{enumerate}
\item If there is an isolated element of $\VV_t^i$, activate it (w.p.\ $\geq (16n)^{-1}$). Otherwise, ``keep'' the current cluster by transporting to wherever activation occurs (w.p.\ $\geq (16n\cdot 2d)^{-1}$). Repeat this step to ensure that the resulting cluster is connected.
\item Once the cluster $\VV_{t+2}^i$ is connected, if it does not belong to $\Refcon$, activate any element with the least $e_1$ component and transport it to $x+e_1$, where $x$ is any element of $\VV_{t+2}^i$ with the greatest $e_1$ component (w.p.\ $\geq (16n)^{-1} (2d)^{-4}$). Otherwise, keep the current cluster (w.p.\ $\geq (16n \cdot 2d)^{-1})$. Repeat this step to ensure that $\VV_{t+4}^i$ belongs to $\Refcon$ (i.e., equals $\{x,x+e_1\}$ or $\{x,x+e_1,x+2e_1\}$ for some $x \in \Z^d$).
\end{enumerate}
The factors of $(16n)^{-1}$ arise from the use of Lemma~\ref{hm small parts}; factors of $(2d)^{-1}$ arise from dictating random walk steps during the transport component of the dynamics.

We can apply this procedure to the rest of the clusters to ensure that $\VV_{t+4\#\CC_0} \in \Refcon^\times$. Note that $\# \CC_0$ is at most $\frac{n}{2}$. This implies that
\[ \PPprod_{\CC_0} \big( \xi_1 > t + \tfrac{n}{2} \bigm\vert \xi_1 > t \big) \leq q\]
where $q = 1 - p^{n/2}$. Continuing inductively, we find that
\[ \PPprod_{\CC_0} \big( \xi_1 > t \big) \leq q^{[\frac{2t}{n}]} \leq q^{-1} e^{- \gamma_1 t}  \]
for $\gamma_1 = \tfrac{2}{n} \log \tfrac{1}{q}$. Since $q^{-1} \leq 2$, this implies \eqref{eq xi tail}.
\end{proof}

\subsection{Two standard estimates for random walks associated with IHAT}\label{standard ests}

Our analysis of the random walks $(S_k^{i,j})_{k \geq 0}$ associated with IHAT relies on two standard estimates. We state these estimates and then explain why they apply.

Define the first hitting time of a set $A \subseteq \Z^d$ by the random walk $(S_k^{i,j})_{k \geq 0}$ associated with IHAT by 
\[
T^{i,j}_{A} = \inf\{k \geq 0: S_k^{i,j} \in A\}.
\]
Additionally, denote by $B(r) = \{z \in \Z^d: \|z \| < r\}$ the discrete Euclidean ball of radius $r > 0$ centered at the origin.

By Proposition 2.4.5 of \cite{lawler2010random}, there are constants $\gamma_2, \gamma_3 > 0$ such that for every $\CC_0 \in \Refcon^\times$, every distinct pair $i < j$ of clusters in $\CC_0$, and every $r, \lambda > 0$, 
\begin{equation}\label{esc time}
\QQ_{\CC_0} (T^{i,j}_{B(r)^\cc} > \lambda r^2 ) \leq \gamma_2 e^{-\gamma_3 \lambda }.
\end{equation} 
Furthermore, by Proposition 6.4.2 of \cite{lawler2010random}, there are constants $\gamma_4, \gamma_5 \geq 1$ such that, if $r \geq \gamma_4$ and if $x = S_0^{i,j} \notin B(r)$ under $\QQ_{\CC_0}$, then
\begin{equation}\label{avoid}
\QQ_{\CC_0} \big( T^{i,j}_{B(r)} < \infty \big) \leq \gamma_5 \left(\frac{r}{\|x\|}\right)^{d-2}.
\end{equation}

These estimates apply because $(S_k^{i,j})_{k \geq 0}$ is a symmetric, aperiodic, and irreducible random walk on $\Z^d$, with increment norms $\|S_{k+1}^{i,j} - S_k^{i,j}\|$ that are exponentially tight. First, it is symmetric because the transition probabilities of IHAT are translation invariant. Second, it is aperiodic because, if $\VV_t \in \Refcon^\times$, then $\VV_{t+1} = \VV_t$ with positive probability. This is true because, if $\VV_t \in \Refcon^\times$, then every element $x \in \VV_t^1$ can be activated with positive probability and has a neighbor, hence it can be transported to its site of activation with positive probability. Third, it is not hard to see that, for any $l \in \{1,\dots,d\}$, $S_1^{i,j} = S_0^{i,j} + e_l$ with positive probability, which implies that this random walk is irreducible on $\Z^d$. Fourth, the norm of the increment $\|S_{k+1}^{i,j} - S_k^{i,j} \|$ is exponentially tight because it is at most $\xi_{k+1}^{i,j} - \xi_k^{i,j}$, which is exponentially tight by Lemma~\ref{xi tail}.

In fact, the results cited from \cite{lawler2010random} assume that the increments of the walk are bounded, and the constants $\gamma_2$ through $\gamma_5$ may in general depend on the increment distribution. However, the proofs of these results apply as written to the case when the norms of the increments are exponentially tight, with the exception that the reference to Proposition 4.3.1 in the proof of Proposition 6.4.2 must be replaced by a reference to Proposition 4.3.5 (all of these are results in \cite{lawler2010random}). Additionally, because there are only three possible increment distributions (corresponding to dimer-dimer, dimer-trimer, and trimer-trimer), we can assume that $\gamma_2$ through $\gamma_5$ are the same for all clusters.

\section{Growth of cluster separation under HAT and IHAT}\label{sec clust sep}

The purpose of this section is to prove Proposition \ref{inf xi}, which states that there is a positive probability that the natural clustering of HAT satisfies the separation condition \eqref{dot conds} of Theorem~\ref{fate}, so long as the initial clustering is an $(a,b)$ separated DOT clustering for sufficiently large numbers $a$ and $b$. We do so by establishing the same result for IHAT and then invoking Proposition~\ref{pe to qe}. We establish the result for IHAT by applying standard random walk estimates \eqref{esc time} and \eqref{avoid} to the random walks associated with IHAT, and then translating these results into analogous conclusions about the separation of the clusters, using Lemma~\ref{amlg2}.

\subsection{Definitions of key quantities and events}

Throughout this section, we fix $d \geq 5$ and $n \geq 4$. To state the main results, we need to define several events, which formalize the following picture. Starting from a DOT clustering $\CC_0 \in \Refcon^\times$ with a separation of $a > 0$, we model the distance between two clusters $i < j$ of $\CC_0$ using the random walk $(S_k^{i,j})_{k \geq 0}$ associated with IHAT $(\VV_t)_{t \geq 0}$. We aim to observe the distance between these clusters double to $2a$ over roughly $(2a)^2$ steps of $S_k^{i,j}$, without dropping below, say, $2 \eps a$ for some $\eps \in (0,1)$. We then aim to observe the separation double again, over $(4a)^2$ steps of $S_k^{i,j}$, without dropping below $4 \eps a$, and so on. In fact, we will budget slightly more time to observe the doubling, and $\eps$ will become smaller as we observe more doublings.

We will use $\ell \geq 1$ to count the number of doublings. In terms of the initial separation $a$ of $\CC_0$, we will wait $t_\ell (a)$ (roughly $(2^\ell a)^2$) steps for the $\ell$\textsuperscript{th} doubling, during which time the separation can decrease by at most a factor $\eps_\ell (a)$, which is roughly $(\ell \log a)^{-1}$. We will require that the number of steps between reference times is at most a quantity $\kappa_\ell (a)$, which is roughly $\log t_\ell (a)$. We define $t_0 (a) = 0$, 
\[
t_\ell (a) = [(n\ell \log a)^4 (2^\ell a)^2 ], \quad \eps_\ell = (n \ell \log (a))^{-1}, \quad \text{and} \quad \kappa_\ell (a) = \beta \log (n \ell t_\ell (a)), \quad \ell \geq 1,
\]
in terms of a constant $\beta > 0$ that we will later choose in terms of $\gamma_1$ from Lemma~\ref{xi tail}. Here, $[r]$ denotes the integer part of a real number $r$. 

We define four events for each $\ell \geq 1$. We aim to observe that:
\begin{enumerate}
\item The clusters become separated by $2^\ell \eps_\ell^{-1} a$ by time $t_\ell$.
\item The separation remains above $2^{\ell-1} \eps_\ell a$ during $\{t_{\ell-1}, \dots, t_\ell\}$.
\item The separation remains above $2^\ell a$ during $\{\sigma_\ell^{i,j}, \dots, t_\ell\}$, where
\[
\sigma_\ell^{i,j} = \inf \{ s \geq t_{\ell-1}: \dist (\VV_s^i, \VV_s^j) \geq 2^\ell \eps_\ell^{-1} a\}. 
\]
\item Consecutive reference times between $\xi_{N(t_{\ell-1})}$ and $\xi_{N(t_\ell)+1}$ never differ by more than $\kappa_\ell$, where $N(t) = |\{m \geq 1: \xi_m \leq t\}|$ denotes the number of returns to $\Refcon^\times$ by time $t$.
\end{enumerate}

\begin{figure}[htbp]
\centering {\includegraphics[width=0.75\linewidth]{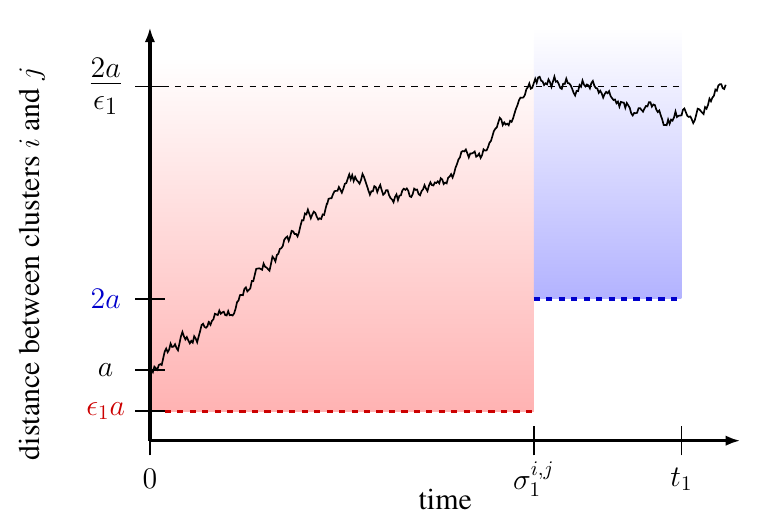}}
\caption[An occurrence of $\cap_{k=1}^3 \GG^{i,j}_k (\ell)$]{An occurrence of $\cap_{k=1}^3 \GG^{i,j}_k (1)$. 
}
\label{fig: eventfig}
\end{figure}

These descriptions correspond to the following events (Figure \ref{fig: eventfig}): 
\begin{align*}
\GG_1^{i,j} (\ell) &= \big\{ \sigma^{i,j}_\ell \leq t_\ell \big\},\\ 
\GG_2^{i,j} (\ell) &= \big\{ \dist(\VV_s^i,\VV_s^j) \geq 2^{\ell-1} \eps_\ell a  \,\,\,\text{for}\,\,\, t_{\ell-1} \leq s \leq t_\ell \big\},\\
\GG_3^{i,j} (\ell) &= \big\{ \dist(\VV_s^i,\VV_s^j) \geq 2^\ell a \,\,\, \text{for} \,\,\, \sigma^{i,j}_\ell \leq s \leq t_\ell  \big\},\\ 
\GG_4 (\ell) &= \big\{ \xi_m - \xi_{m-1} \leq \kappa_\ell \,\,\, \text{for} \,\,\, N (t_{\ell-1}) < m \leq N(t_\ell) + 1\big\}.
\end{align*}
We denote their intersections as 
\[
\GG_k (\ell) = \cap_{i < j} \GG_k^{i,j} (\ell), \,\, 1 \leq k \leq 3; \quad \GG_k = \cap_{\ell \geq 1} \GG_k (\ell), \,\, 1 \leq k \leq 4; \quad \text{and} \quad \GG = \cap_{k=1}^4 \GG_k.
\]

\subsection{Proof of Proposition~\ref{inf xi}}

Recall $\Grow_{a,b,\delta}$, which consists of sequences of $(a,b)$ DOT clusterings $(\CC_t)_{t \geq 0}$ that satisfy $\sep (\CC_t) \geq t^{\frac12 - \delta}$ for every $t \geq 0$ \eqref{event grow}. Essentially, if we can show that IHAT belongs to $\Grow_{a,b,\delta}$ with positive probability, for some choice of parameters, then the main approximation result (Proposition~\ref{pe to qe}) will imply the same of the natural partitioning of HAT. This separation growth condition implies the one in Proposition~\ref{inf xi}.

The event $\GG$ is significant because, when it occurs and when $\VV_0 \in \Refcon^\times$, the sequence of IHAT states $(\VV_t)_{t \geq 0}$ satisfies the separation condition in $\Grow_{a,b,\delta}$ for some choice of parameters. This is the content of the next result. By the preceding discussion, this reduces the proof of Proposition~\ref{inf xi} to establishing that $\GG$ occurs with positive probability.

\begin{proposition}[Separation grows when $\GG$ occurs]\label{rel sep bd}
There exists $b = b(d,n) > 0$ such that, for any $\delta \in (0,\frac12)$, there exists $\alpha = \alpha (b,d,n,\delta)$, such that, if $a \geq \alpha$ then, 
\begin{equation}\label{g v ref}
\GG \cap \{\VV_0 \in \Refcon^\times\} \subseteq \left\{ (\VV_t)_{t \geq 0} \in \Grow_{\tilde a,b,\delta} \right\},
\end{equation}
where $\tilde a = a^{0.99}$.
\end{proposition}

The second main result of this section is a bound on $\QQ_{\CC_0} (\GG)$, which applies to every $\CC_0 \in \Refcon^\times$ that satisfies $\sep (\CC_0) \geq a$ for sufficiently large $a$. Note that such clusterings are necessarily $(a,b)$ DOT clusterings for every $b \geq 2 (\log a)^{-1}$ \eqref{two or three elems}, because their clusters are reference dimers and trimers, which have diameters of at most $2$.

\begin{proposition}[$\GG$ is typical for IHAT]\label{tl ntb}
There is a constant $\alpha = \alpha (b,d,n)$ such that every $\CC_0 \in \Refcon^\times$ with $\sep (\CC_0) \geq \alpha$ satisfies $\QQ_{\CC_0} (\GG) \geq \frac12$.
\end{proposition}

Together, the two preceding propositions and Proposition~\ref{pe to qe} imply Proposition~\ref{inf xi}.

\begin{proof}[Proof of Proposition \ref{inf xi}]
Fix $d \geq 5$ and $n \geq 4$, and let $b > 0$ be the constant from Proposition~\ref{rel sep bd}. It suffices to show that, for any $\delta \in (0,\frac12 - \frac{1}{d-2})$, there is $a > 0$ such that, if $W \in \Conf_{d,n}$ has a clustering $\mathcal{W} \in \Clust_{a,b}^\bullet (W) \cap \Refcon^\times$, then the natural clustering $(\UU_t)_{t \geq 0}$ of $(U_t)_{t \geq 0}$ with $\mathcal{W}$ satisfies
\[
\PP_{\mathcal{W}} \left( (\UU_t)_{t \geq 0} \in \Grow_{\tilde a, b, \delta} \right) \geq \frac14.
\]

To this end, fix $\varepsilon = \frac12$ in Proposition~\ref{pe to qe}, and let $\alpha > 1$ be the largest of the constants it names in Propositions~\ref{pe to qe}, \ref{rel sep bd}, and \ref{tl ntb}. Then, take $a = \alpha^{1.1}$, so that $\tilde a = a^{0.99} \geq \alpha$. By Propositions~\ref{pe to qe}, \ref{rel sep bd} and \ref{tl ntb}, we have
\[
\PP_{\mathcal{W}} \left( (\UU_t)_{t \geq 0} \in \Grow_{\tilde a, b, \delta} \right) \geq \frac12 \QQ_{\mathcal{W}} \left( (\VV_t)_{t \geq 0} \in \Grow_{\tilde a, b, \delta} \right) \geq \frac12 \QQ_{\mathcal{W}} ( \GG ) \geq \frac14.
\]
Respectively, the hypotheses of these three propositions are satisfied because $\mathcal{W} \in \Clust_{\tilde a,b}^\bullet$ for $\tilde a \geq \alpha$, because $b$ is as in Proposition~\ref{rel sep bd} and $a \geq \alpha$, and because $\mathcal{W} \in \Clust_{a,b}^\bullet \cap \Refcon^\times$ for $a \geq \alpha$.
\end{proof} 

We turn our attention to the proofs of Propositions~\ref{rel sep bd} and \ref{tl ntb}. 

\subsection{Proof of Proposition~\ref{rel sep bd}}

The proof shows that the inclusion \eqref{g v ref} holds when $\GG_2$ and $\GG_4$ occur.

\begin{proof}[Proof of Proposition~\ref{rel sep bd}]
We need to show that there is $b > 0$ such that, for any $\delta \in (0,\frac12)$, we can take $a$ sufficiently large to ensure that, when $\GG$ occurs, 
\[
\sep (\VV_s) \geq \tilde a, \quad \diam (\VV_s^i) \leq b \log \dist(\VV_s^i, \VV_s^j), \quad \text{and} \quad \sep (\VV_s) \geq s^{\frac12 - \delta},
\]
for every distinct pair of clusters $1 \leq i < j \leq \# \VV_0$ and for every $s \geq 0$. We address these three conditions in turn.

\begin{enumerate}
\item When $\GG_2$ occurs, $\sep (\VV_s)$ is at least $\eps_1 a$ for every $s \geq 0$. If $a$ is sufficiently large in terms of $n$, then $\sep (\VV) \geq \tilde a$.
\item Let $\ell \geq 1$ and $s \in [t_{\ell-1}, t_\ell]$. Recall that $N (t) = |\{m \geq 1: \xi_m \leq t\}|$ denotes the number of returns to $\Refcon^\times$ by time $t$. Since diameter grows at most linearly in time and since $\VV_{\xi_{N(s)}}^i \in \Refcon$ has a diameter of at most $2$,
\begin{equation}\label{xi diam bd}
\diam (\VV_s^i) \leq \xi_{N(s)+1} - \xi_{N(s)} + 2.
\end{equation}
We can obtain a further upper bound on $\diam (\VV_s^i)$ by noting that, when $\GG_4$ occurs, $\xi_{N(s)+1} - \xi_{N(s)}$ is at most $\kappa_\ell$, which is essentially $\log t_\ell$. Recall that $t_\ell \lesssim_n (2^\ell a)^2 \cdot (\ell \log a)^4$. Hence, when $a$ is sufficiently large,
\[
\diam (\VV_s^i) \lesssim_n \ell + \log a.
\]
On the other hand, when $\GG_2$ occurs, $\dist (\VV_s^i, \VV_s^j) \geq 2^{\ell-1} \eps_\ell a$, which implies that
\[
\log \dist(\VV_s^i,\VV_s^j) \gtrsim_n \ell + \log a.
\]
By comparing these bounds, we see that $\diam (\VV_s^i) \leq b \log \dist (\VV_s^i, \VV_s^j)$ holds for all sufficiently large $a$, so long as $b$ is large enough in terms of $n$.
\item Let $s \in [t_{\ell-1},t_\ell]$, in which case $s^{1/2 - \delta}$ is at most $t_\ell^{1/2 - \delta} \lesssim (2^\ell a)^{1-2\delta} \cdot (\ell \log a)^{2-4\delta}$. When $\GG_2$ occurs, $\sep (\VV_s) \gtrsim_n \frac{2^\ell a}{\ell \log a}$, hence $\sep (\VV_s)$ is larger by a factor of roughly $(2^\ell a)^{2\delta}$:
\[
\frac{\sep (\VV_s)}{s^{1/2-\delta}} \geq \frac{\sep (\VV_s)}{t_\ell^{1/2-\delta}} \gtrsim_n \frac{(2^\ell a)^{2\delta}}{(\ell \log a)^3}.
\]
Consequently, we can take $a$ sufficiently large in terms of $\delta$ and $n$ to ensure that $\sep (\VV_s) \geq s^{\frac12 - \delta}$ (Figure~\ref{fig: sqrtfig}).
\end{enumerate} 
\end{proof}

\begin{figure}[htbp]
\centering {\includegraphics[width=0.75\linewidth]{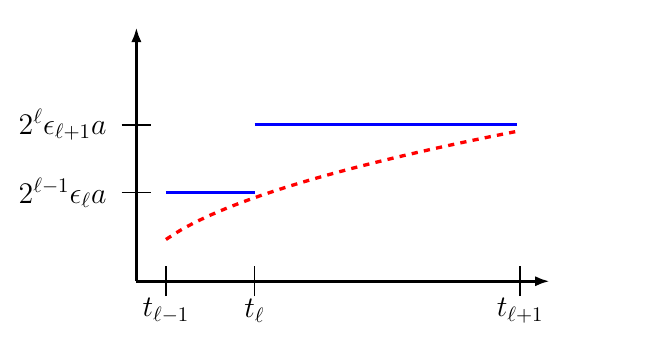}}
\caption[Growth of cluster separation]{When $\GG_2$ occurs, the separation of IHAT lies above the blue lines. The red dashed line is a lower bound of $s^{1/2-\delta}$ on cluster separation $\sep (\VV_s)$.}
\label{fig: sqrtfig}
\end{figure}

\subsection{Proof of Proposition~\ref{tl ntb}}

This subsection is devoted to a proof of the following result. To state it, we denote by $\mathscr{F}_t$ the $\sigma$-field generated by $(\VV_s)_{0 \leq s \leq t}$. Additionally, throughout this subsection, we use $\EE_\CC$ to denote expectation with respect to $\QQ_\CC$ instead of $\PP_\CC$.

\begin{proposition}\label{qg1}
There exists $\alpha = \alpha (d,n) > 0$ such that, if $\CC_0 \in \Refcon^\times$ satisfies $\sep (\CC_0) \geq \alpha$, then
\begin{equation}\label{eq qg1}
\PPprod_{\CC_0} \left( \GG (\ell)^\cc \mid \mathscr{F}_{\eta_{\ell-1}} \right) \1_{\GG(\ell-1)} \leq 8 \ell^{-2}, \quad \ell \geq 1,
\end{equation}
where $\eta_{\ell-1}$ denotes $\xi_{N(t_{\ell-1})+1}$.
\end{proposition}

Proposition~\ref{tl ntb} follows easily from Proposition~\ref{qg1}.

\begin{proof}[Proof of Proposition~\ref{tl ntb}]
Let $\alpha > 0$ be the constant named in Proposition~\ref{qg1}, let $\CC_0 \in \Refcon^\times$ satisfy $\sep (\CC_0) \geq \alpha$, and let $\ell \geq 1$. By conditioning on $\mathscr{F}_{\eta_{\ell-1}}$ and then applying Proposition~\ref{qg1}, we find that
\begin{equation*}
\PPprod_{\CC_0} \Big( \GG(\ell)^\cc \cap_{m=1}^{\ell-1} \GG (m) \Big) \leq \EE_{\CC_0} \, \Big[ \PPprod_{\CC_0} \big( \GG (\ell)^\cc \mid \mathscr{F}_{\eta_{\ell-1}} \big) \1_{\GG (\ell-1)}  \Big] \leq 8\ell^{-2}. 
\end{equation*} 
Consequently,
\[
\PPprod_{\CC_0} \Big( \big( \cap_{\ell=1}^\infty \GG (\ell) \big)^\cc \Big) = \sum_{\ell = 1}^\infty \PPprod_{\CC_0} \Big( \GG(\ell)^\cc \cap_{m=1}^{\ell-1} \GG (m) \Big) \leq \sum_{\ell=1}^\infty 8 \ell^{-2} < \frac12.
\]
\end{proof}

We will prove Proposition~\ref{qg1} in terms of events that are analogous to the $\GG_k^{i,j} (\ell)$, but which reference the random walks associated with IHAT instead of the distance between IHAT clusters, so that we can apply the random walk hitting estimates of Section~\ref{standard ests} to them. To convert between the two, we recall Lemma~\ref{amlg2}, which states that 
\[
\dist (\VV_s^i, \VV_s^j) \geq \big\| S_{m-1}^{i,j} \big\| - 6 \big( \xi_m - \xi_{m-1} \big)^2, \quad \xi_{m-1} \leq s < \xi_m.
\]
For this bound to be useful, we need the upper bound on $\xi_m- \xi_{m-1}$ that the occurrence of $\GG_4 (\ell)$ supplies,
\[
\xi_m - \xi_{m-1} \leq \log (n \ell t_\ell), \quad N (t_{\ell-1}) < m \leq N (t_\ell) + 1.
\]
By combining these two bounds, we see that
\[
\dist (\VV_s^i, \VV_s^j) \geq \|S_{N(s)}^{i,j}\| - 6 \kappa_\ell^2, \quad t_{\ell-1} \leq s \leq t_\ell.
\]
This motivates the definition of the following events, in terms of an analogue of $\sigma_\ell^{i,j}$:
\[
\rho_\ell^{i,j} = \inf\big\{m > N (t_{\ell-1}): S_m^{i,j} \notin B \big( 2^{\ell} \eps_\ell^{-1} a + 6\kappa_\ell^2 \big)\big\}, \quad \ell \geq 1.
\]
We also use $\kappa_\ell$ to define analogues of $\GG_k^{i,j} (\ell)$: 
\begin{align*}
\FF_1^{i,j} (\ell) &= \big\{ \rho_\ell^{i,j} \leq N (t_\ell) \big\},\\ 
\FF_2^{i,j} (\ell) &= \big\{ S_m^{i,j} \notin B \big( 2^{\ell-1} \eps_\ell a + 6\kappa_\ell^2 \big),\,\,N (t_{\ell-1}) \leq m \leq N (t_\ell) + 1 \big\},\\  
\FF_3^{i,j} (\ell) &= \big\{ S_m^{i,j} \notin B \big( 2^{\ell}a + 6\kappa_\ell^2 \big),\,\, \rho_\ell^{i,j} \leq m \leq N (t_\ell) + 1 \big\},
\end{align*}
and $\FF_4 (\ell) = \GG_4 (\ell)$. 
We denote their intersections as
\[
\FF_k (\ell) = \cap_{i<j} \FF_k^{i,j} (\ell), \,\, 1 \leq k \leq 3; \quad \FF_k = \cap_{\ell \geq 1} \FF_k (\ell), \,\, 1 \leq k \leq 4; \quad \text{and} \quad \FF = \cap_{k=1}^4 \FF_k.
\]

To prove Proposition~\ref{qg1}, it suffices to bound above the probability that $\FF (\ell)^\cc$ occurs because $\FF (\ell)$ is a sub-event of $\GG (\ell)$.

\begin{proposition}\label{h sub g}
For each $\ell \geq 1$, $\FF (\ell) \subseteq \GG (\ell)$.
\end{proposition}

\begin{proof}
Fix $\ell \geq 1$ and let $s \in [t_{\ell-1}, t_\ell]$. When $\FF_4 (\ell) = \GG_4 (\ell)$ occurs, $\xi_{N(s)+1} - \xi_{N(s)}$ is at most $\kappa_\ell$. In this case, Lemma~\ref{amlg2} implies that $\dist( \VV_s^i, \VV_s^j ) \geq \|S_{N(s)}^{i,j}\| - 6 \kappa_\ell^2$ for every pair of distinct clusters $i < j$. This bound implies the following inclusions. First, $\FF_1^{i,j} (\ell) \subseteq \GG_1^{i,j} (\ell)$ holds because
\begin{align*}
\FF_1^{i,j} &= \left\{ \exists s \in (t_{\ell-1},t_\ell]: \| S_{N(s)}^{i,j} \| \geq 2^\ell \eps_\ell^{-1} a + 6 \kappa_\ell^2 \right\}\\ 
&\subseteq \left\{ \exists s \in (t_{\ell-1},t_\ell]: \dist(\VV_s^i,\VV_s^j) \geq 2^\ell \eps_\ell^{-1} a \right\} = \GG_1^{i,j}.
\end{align*}
Second, the inclusion $\FF_2^{i,j} (\ell) \subseteq \GG_2^{i,j}$ is justified by
\begin{align*}
\FF_2^{i,j} (\ell) &\subseteq \left\{\|S_{N(s)}^{i,j}\| \geq 2^{\ell-1} \eps_\ell a + 6 \kappa_\ell^2, \,\, s \in [t_{\ell-1},t_\ell]\right\}\\ 
&\subseteq \left\{ \dist (\VV_s^i, \VV_s^j) \geq 2^{\ell-1}\eps_\ell a, \,\, s \in [t_{\ell-1},t_\ell] \right\} = \GG_2^{i,j} (\ell).
\end{align*}
Third, $\FF_3^{i,j} (\ell) \subseteq \GG_3^{i,j} (\ell)$ holds because the reasoning that led to the first inclusion also shows that $\rho_\ell^{i,j} \leq N(\sigma_\ell^{i,j})$, hence
\begin{align*}
\FF_3^{i,j} (\ell) &=\left\{ \|S_{m}^{i,j}\| \geq 2^\ell a + 6 \kappa_\ell^2, \,\, m \in [\rho_\ell^{i,j}, N(t_\ell)+1] \right\}\\ 
&\subseteq \left\{ \|S_{N(s)}^{i,j}\| \geq 2^\ell a + 6 \kappa_\ell^2, \,\, s \in [\sigma_\ell^{i,j}, t_\ell] \right\}\\ 
& \subseteq \left\{ \dist (\VV_s^i, \VV_s^j) \geq 2^\ell a, \,\, s \in [\sigma_\ell^{i,j}, t_\ell] \right\} = \GG_3^{i,j} (\ell).
\end{align*}
\end{proof}

Recall that Lemma~\ref{xi tail} bounds above the tail probabilities of $\xi_m - \xi_{m-1}$. We combine it with a union bound to bound above the probability that $\FF_4 (\ell)^\cc$ occurs.

\begin{proposition}\label{notrunc}
There exists $\alpha > 0$ such that, if $\CC_0 \in \Refcon^\times$ satisfies $\sep (\CC_0) \geq \alpha$, then
\[
\PPprod_{\CC_0} ( \FF_4 (\ell)^\cc ) \leq (8n \ell)^{-2}, \quad \ell \geq 1.
\]
\end{proposition}

\begin{proof}
Assume that $\sep (\CC_0) = a$ for some $a > 0$. 
Since $N(t_\ell)$ is at most $t_\ell$, $\FF_4 (\ell)^\cc$ is a union of at most $t_\ell$ events of the form $\{\xi_m - \xi_{m-1} > \kappa_\ell\}$. Lemma~\ref{xi tail} states that there is a constant $\gamma_1 > 0$ such that 
\[
\QQ_{\mathcal{D}_0} (\xi_m - \xi_{m-1} > s) \leq 2e^{-\gamma_1 s}, \quad s \geq 0,
\] 
holds uniformly over $\mathcal{D}_0 \in \Conf_{d,n}^\times$ that satisfy the same hypothesis as $\CC_0$. We combine this tail bound with a union bound and the fact that $\kappa_\ell = \beta \log ( n \ell t_\ell)$:
\[
\QQ_{\CC_0} ( \FF_4 (\ell)^\cc ) \leq 2t_\ell e^{-\gamma_1 \kappa_\ell} = \frac{2 t_\ell}{(n\ell t_\ell)^{\gamma_1 \beta}}.
\]
If $a$ is at least a sufficiently large constant $\alpha > 0$, simply to ensure that $t_\ell (a) \geq 1$, then a suitable choice of $\beta$ in terms of $\gamma_1$ implies the claimed bound.
\end{proof}

When $\FF_4 (\ell)$ occurs, the time between consecutive reference times is at most $\kappa_\ell$ until time $t_\ell$, which implies that there are roughly $t_\ell / \kappa_\ell$ random walk steps between $t_{\ell-1}$ and $t_\ell$. Hence, it is rare for the norm of this random walk to be significantly smaller than $\sqrt{t_\ell /\kappa_\ell}$ until time $t_\ell$, i.e., for $\FF_1^{i,j} (\ell)^\cc$ to occur.

\begin{proposition}\label{t before n}
There exists $\alpha = \alpha (d,n)$ such that, if $\CC_0 \in \Refcon^\times$ satisfies $\sep (\CC_0) \geq \alpha$, then
\begin{equation}\label{t bef n eq}
\PPprod_{\CC_0} \big( \FF_1^{i,j} (\ell)^\cc \cap \FF_4 (\ell) \bigm\vert \mathscr{F}_{\eta_{\ell-1}} \big) 
\leq (8 n \ell)^{-2}, \quad \ell \geq 1,
\end{equation}
where $\eta_{\ell-1}$ denotes $\xi_{N(t_{\ell-1})+1}$.
\end{proposition}

\begin{proof}
Assume that $\sep (\CC_0) = a$ for some $a > 0$. Denote $r = 2^\ell \eps_\ell^{-1} a + 6 \kappa_\ell^2$. By definition,
\[
\FF_1^{i,j} (\ell)^\cc = \left\{ S_{m}^{i,j} \in B(r), \,\, N(t_{\ell-1}) < m \leq N(t_{\ell}) \right\}.
\]
Assume for the moment that, when $\FF_4 (\ell)$ occurs, the number of steps in $\FF_1^{i,j} (\ell)^\cc$ satisfies
\begin{equation}\label{f4 claim}
N(t_\ell) - N(t_{\ell-1}) - 1 \geq \lambda r^2
\end{equation}
with $\lambda = \gamma_3^{-1} \log \big(\gamma_2 (8 n \ell)^2 \big)$. We use \eqref{f4 claim} to prove \eqref{t bef n eq} as
\begin{align*}
\PPprod_{\CC_0} \left( \FF_1^{i,j} (\ell)^\cc \cap \FF_4 (\ell) \bigm\vert \mathscr{F}_{\eta_{\ell-1}} \right) &\leq \QQ_{\CC_0} \left( S_{N(t_{\ell-1}) + 1 + m}^{i,j} \in B(r), \,\, 0 \leq m \leq \lambda r^2 \bigm\vert \scr{F}_{\eta_{\ell-1}} \right)\\ 
&= \QQ_{\VV_{\eta_{\ell-1}}} \left( T^{i,j}_{B(r)^\cc} > \lambda r^2 \right) \leq \gamma_2 e^{-\gamma_3 \lambda}  = (8n\ell)^{-2}.
\end{align*}
The first inequality follows from the definition of $\FF_1^{i,j} (\ell)^\cc$ and the claimed bound on $N(t_\ell)-N(t_{\ell-1})-1$; the first equality is due to the strong Markov property at time $\eta_{\ell-1}$; the second inequality holds by \eqref{esc time}, which applies because $\VV_{\eta_{\ell-1}} \in \Refcon^\times$; and the second equality holds by the choice of $\lambda$.

The claimed bound \eqref{f4 claim} holds when $a$ is sufficiently large. Indeed, when $\FF_4 (\ell)$ occurs, 
\[
N(t_\ell) - N(t_{\ell-1}) -1 \gtrsim \kappa_\ell^{-1} (t_\ell - t_{\ell-1}).
\]
When $a$ is sufficiently large, this lower bound is at least $(2\kappa_\ell)^{-1} t_\ell$. It therefore suffices to show that
\[
(2\kappa_\ell)^{-1} t_\ell \geq \lambda r^2.
\]
Some algebra shows that $(2\kappa_\ell)^{-1} t_\ell \gtrsim_n (\log a) \lambda r^2$. Consequently, there exists $\alpha = \alpha(\gamma_2,\gamma_3,n) = \alpha (d,n) > 0$ such that, if $a \geq \alpha$, then the preceding bound is satisfied.
\end{proof}

For $\FF_2^{i,j} (\ell)^\cc$ to occur, the norm of the random walk must eventually drop below its value at time $t_{\ell-1}$ by a factor of roughly $\eps_\ell^{-1}$. The estimate \eqref{avoid} implies that this occurs with a probability of at most roughly $\eps_\ell^{d-2}$.

\begin{proposition}\label{h3}
There exists $\alpha = \alpha (d,n)$ such that, if $\CC_0 \in \Refcon^\times$ satisfies $\sep (\CC_0) \geq \alpha$, then
\begin{equation*}
\PPprod_{\CC_0} \big( \FF_2^{i,j} (\ell)^\cc \bigm\vert \mathscr{F}_{\eta_{\ell-1}} \big) \1_{\GG (\ell-1)} \leq (8n\ell)^{-2}, \quad \ell \geq 1,
\end{equation*}
where $\eta_{\ell-1}$ denotes $\xi_{N(t_{\ell-1})+1}$.
\end{proposition}

\begin{proof}
Assume that $\sep (\CC_0) = a $ for some $a > 0$. Denote $r = 2^\ell \eps_\ell a + 6 \kappa_\ell^2$. By definition,
\[
\FF_2^{i,j} (\ell) = \left\{S_m^{i,j} \in B(r)^\cc, \,\, N(t_{\ell-1}) \leq m \leq N (t_\ell) \right\}.
\]
For $\FF_2^{i,j} (\ell)$ to occur, it suffices for the random walk from $S^{i,j}_{N(t_{\ell-1})}$ to never hit $B(r)$. In fact, when $\GG (\ell-1)$ occurs, $S^{i,j}_{N(t_{\ell-1})} \in B(r)^\cc$, so we can consider the random walk from $X = S^{i,j}_{N(t_{\ell-1})+1}$ instead. This is a minor point, so we discuss it at the end of the proof. For now, we simply use the inclusion  
\begin{equation}\label{f2 incl}
F_2^{i,j} (\ell) \cap \GG(\ell-1) \supseteq \left\{ S_m^{i,j} \in B(r)^\cc, \,m > N(t_{\ell-1})\right\} \cap \GG(\ell-1).
\end{equation}
By \eqref{f2 incl}, the strong Markov property at $\eta_{\ell-1}$, and \eqref{avoid}, when $\GG (\ell-1)$ occurs, 
\begin{align}
\QQ_{\CC_0} \left( \FF_2^{i,j} (\ell) \bigm\vert \scr{F}_{\eta_{\ell-1}} \right) 
&\geq \QQ_{\CC_0} \left( S_m^{i,j} \in B(r)^\cc, \,m > N(t_{\ell-1}) \bigm\vert \scr{F}_{\eta_{\ell-1}} \right) \nonumber\\ 
&= \QQ_{\VV_{\eta_{\ell-1}}} \left( T_{B(r)^\cc}^{i,j} = \infty \right) \geq 1 - \gamma_5 \Big( \frac{r}{\|X\|} \Big)^{d-2}. \label{h3 bd}
\end{align}
Our use of \eqref{avoid} relies on the fact that $\VV_{\eta_{\ell-1}} \in \Refcon^\times$ and requies that we take $a$ sufficiently large to ensure that $r \geq \gamma_4$.

To bound below $\|X\|$, we note that the separation of $\VV_{\eta_{\ell-1}}$ is at least $2^{\ell-1} a$, while the diameters of the clusters are at most $2$, when $\GG (\ell-1)$ occurs. By \eqref{tri app both},
\begin{equation}\label{f2 x bd}
\| X \| \geq \dist (\VV_{\eta_{\ell-1}}^i,\VV_{\eta_{\ell-1}}^j) - \diam (\VV_{\eta_{\ell-1}}^i) - \diam (\VV_{\eta_{\ell-1}}^j) \geq 2^{\ell-1} a - 4.
\end{equation}
Call this lower bound $R$, in which case \eqref{h3 bd} and \eqref{f2 x bd} imply that
\[
\QQ_{\CC_0} \left( \FF_2^{i,j} (\ell)^\cc \bigm\vert \scr{F}_{\eta_{\ell-1}} \right) \1_{\GG(\ell-1)}
\leq \gamma_5 \Big( \frac{r}{R} \Big)^{d-2}.
\]
Some algebra shows that $\frac{r}{R} \gtrsim (n \ell \log a)^{-1}$ when $a$ is sufficiently large in terms of $n$. Hence, there exists $\alpha = \alpha (\gamma_4,\gamma_5,n) = \alpha (d,n) > 0$ such that, if $a \geq \alpha$, then the preceding bound is at most $(8n\ell)^{-2}$.

To conclude, we explain why $S^{i,j}_{N(t_{\ell-1})} \notin B(r)$, which justifies the inclusion \eqref{f2 incl}. By \eqref{tri app both} and \eqref{xi diam bd}, and when $\GG (\ell-1)$ occurs,
\[
\| S_{N(t_{\ell-1})}^{i,j} \| \geq \dist (\VV_{t_{\ell-1}}^i, \VV_{t_{\ell-1}}^j) - 2 \left(\xi_{N(t_{\ell-1})+1} - \xi_{N(t_{\ell-1})} + 2\right) \geq 2^{\ell-1}a - 2 \kappa_{\ell-1}.
\]
This quantity is roughly larger than $r$ by a factor of $\ell \log a$, so $2^{\ell-1}a - 2\kappa_{\ell-1} - r > 0$ when $a$ is sufficiently large in terms of $n$. We can assume that this is true of every $a \geq \alpha$ by increasing $\alpha$ as necessary.
\end{proof}

For $\FF_3^{i,j} (\ell)^\cc$ to occur, the norm of the random walk must eventually drop by a factor of $\eps_\ell^{-1}$, relative to its value at step $\rho_\ell^{i,j}$. As in the proof of Proposition~\ref{h3}, the estimate \eqref{avoid} implies that this occurs with a probability of at most roughly $\eps_\ell^{d-2}$.

\begin{proposition}\label{h4}
There exists $\alpha = \alpha (d,n)$ such that, if $\CC_0 \in \Refcon^\times$ satisfies $\sep (\CC_0) \geq \alpha$, then
\begin{equation*}
\PPprod_{\CC_0} \big( \FF_1^{i,j} (\ell) \cap \FF_3^{i,j} (\ell)^\cc \bigm\vert \mathscr{F}_{\eta_{\ell-1}} \big) \leq (8n\ell)^{-2}.
\end{equation*}
\end{proposition}

\begin{proof}
Assume that $\sep (\CC_0) = a$ for some $a > 0$. When $\FF_1^{i,j} (\ell)$ occurs, there is a first step $m \in (N (t_{\ell-1}), N (t_\ell)]$ at which $Y = S_m^{i,j}$ belongs to $B(R)^\cc$, where $R = 2^\ell \eps_\ell^{-1} a + 6\kappa_\ell^2$. For $\FF_3^{i,j} (\ell)^\cc$ to occur, $S_m^{i,j}$ must hit $B(r)$, where $r = 2^\ell a + 6\kappa_\ell^2$, starting from $Y$. Note that $R$ is larger than $r$ by a factor of roughly $n \ell \log a$. By the strong Markov property at time $\xi_m$ and \eqref{avoid},
\begin{align*}
\PPprod_{\CC_0} \big( \FF_1^{i,j} (\ell) \cap \FF_3^{i,j} (\ell)^\cc \bigm\vert \mathscr{F}_{\eta_{\ell-1}} \big)
&= \EE_{\VV_{\eta_{\ell-1}}} \Big[ \PPprod_{\VV_{\xi_m}} \big( \FF_3^{i,j} (\ell)^\cc \big); \FF_1^{i,j} (\ell) \Big] \\  
&\leq \EE_{\VV_{\eta_{\ell-1}}} \Big[ \gamma_5 \Big( \frac{r}{\| Y \|} \Big)^{d-2} \Big] \leq \gamma_5 \Big( \frac{r}{R} \Big)^{d-2}.
\end{align*}
Note that our use of \eqref{avoid} makes use of the fact that $\VV_{\eta_{\ell-1}} \in \Refcon^\times$ and, to use it, we must take $a$ sufficiently large to ensure that $r \geq \gamma_4$. 
Some algebra shows that $\frac{r}{R} \gtrsim (n\ell \log a)^{-1}$ when $a$ is sufficiently large. Hence, there exists $\alpha = \alpha (\gamma_4,\gamma_5,n) = \alpha (d,n) > 0$ such that, if $a \geq \alpha$, then the preceding bound is at most $(8n\ell)^{-2}$.
\end{proof}

We combine the preceding five propositions to prove Proposition~\ref{qg1}.

\begin{proof}[Proof of Proposition~\ref{qg1}]
Let $\alpha$ be the largest of the constants it names in Propositions~\ref{notrunc} through \ref{h4}, and assume that $\CC_0 \in \Refcon^\times$ satisfies $\sep (\CC_0) \geq \alpha$. We aim to show that 
\begin{equation}\label{loc qg1}
\PPprod_{\CC_0} \left( \GG (\ell)^\cc \mid \mathscr{F}_{\eta_{\ell-1}} \right) \1_{\GG (\ell-1)} \leq (8 \ell)^{-2},
\end{equation}
in terms of $\eta_{\ell-1} = \xi_{N(t_{\ell-1})+1}$. 
By Proposition~\ref{h sub g} and a union bound over distinct pairs of clusters, we have
\begin{equation}\label{g to h}
\PPprod_{\CC_0} \left( \GG (\ell)^\cc \mid \mathscr{F}_{\eta_{\ell-1}} \right) \leq \PPprod_{\CC_0} \left( \FF (\ell)^\cc \mid \mathscr{F}_{\eta_{\ell-1}} \right) \leq \sum_{i < j} \PPprod_{\CC_0} \left( \FF^{i,j} (\ell)^\cc \mid \mathscr{F}_{\eta_{\ell-1}} \right).
\end{equation}

Using the fact that, for events $E_1$ and $E_2$, $E_2^\cc$ is contained in the disjoint union $( E_1 \cap E_2^\cc ) \cup E_1^\cc$, we find that
\begin{multline*}
\PPprod_{\CC_0} \left( \FF^{i,j} (\ell)^\cc \bigm\vert \mathscr{F}_{\eta_{\ell-1}} \right) \leq 3 \PPprod_{\CC_0} \left( \FF_4 (\ell)^\cc \bigm\vert \mathscr{F}_{\eta_{\ell-1}} \right) + 2 \PPprod_{\CC_0} \left( \FF_1^{i,j} (\ell)^\cc \cap \FF_4 (\ell) \bigm\vert \mathscr{F}_{\eta_{\ell-1}} \right)\\  
+ \PPprod_{\CC_0} \left( \FF_2^{i,j} (\ell)^\cc \bigm\vert \mathscr{F}_{\eta_{\ell-1}} \right) + \PPprod_{\CC_0} \left( \FF_1^{i,j} (\ell) \cap \FF_3^{i,j} (\ell)^\cc \bigm\vert \mathscr{F}_{\eta_{\ell-1}} \right).
\end{multline*}
When $\GG(\ell-1)$ occurs, we can apply Propositions~\ref{notrunc} through \ref{h4} to bound the terms on the right-hand side as 
\[
\PPprod_{\CC_0} \big( \FF^{i,j} (\ell)^\cc \mid \mathscr{F}_{t_{\ell-1}} \big) \1_{\GG(\ell-1)} \leq 8(n\ell)^{-2}.
\]
The bound \eqref{loc qg1} then follows from \eqref{g to h} and the fact that there are at most $n^2$ distinct pairs of clusters. 
\end{proof}

\section{Strategy for the proof of Theorem \ref{thm form dot}}\label{sec strat 11}

Let us briefly summarize what the preceding sections have accomplished. In Section~\ref{sec fate}, we proved our main result, Theorem~\ref{fate}, assuming Proposition~\ref{inf xi} and Theorem~\ref{thm form dot}. In Sections~\ref{sec intra} through \ref{sec clust sep}, we proved Proposition~\ref{inf xi}, using an approximation of HAT by IHAT and a random walk model of cluster separation under IHAT. In this section, our focus shifts to proving Theorem~\ref{thm form dot}.

We continue to assume that $d \geq 5$ and $n \geq 4$. Recall that Theorem~\ref{thm form dot} identifies, for every $a > 0$, a number of steps $N = N(a,d,n)$ and a positive probability $p = p(a,d,n)$ such that the $\PP_{V}$ probability that $U_N$ has an $(a,1)$ DOT clustering in $\Refcon^\times$ is at least $p$, for any $V \in \Conf_{d,n}$.

If $p$ is allowed to depend on $V$, then it is easy to identify a sequence of $N'$ configurations that can be realized as $(U_1, \dots, U_{N'})$ under $\PP_V$ and is such that $U_{N'}$ has an $(a,1)$ DOT clustering in $\Refcon^\times$. Indeed, it would take only two ``stages'':
\begin{enumerate}
\item[(1')] First, we rearrange $V$ into a line segment emanating in the $-e_1$ direction from, say, the element of $V$ which is least in the lexicographic ordering of $\Z^d$.
\item[(2')] Second, we ``treadmill'' (Figures~\ref{treadmill} and \ref{treadmill2}) a pair of elements from the ``tip'' of the segment, in the $-e_1$ direction, until the pair is sufficiently far from the other elements. We then repeat this process, one pair at a time, until only two or three elements of the initial segment remain.
\end{enumerate}

\begin{figure}[htbp]
\centering {\includegraphics[width=0.85\linewidth]{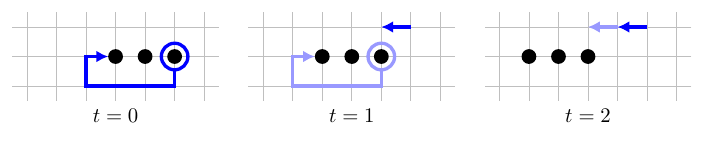}}
\caption[An example of treadmilling]{An example of treadmilling three elements in the $-e_1$ direction.}
\label{treadmill}
\end{figure}

\begin{figure}[htbp]
\centering {\includegraphics[width=0.55\linewidth]{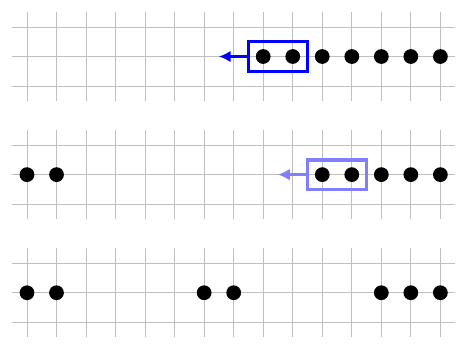}}
\caption[Treadmilling a line segment into a DOT clustering]{By treadmilling pairs of elements in the $-e_1$ direction, a line segment of seven elements can be rearranged into a configuration with a $(5,\frac{2}{\log 5})$ DOT clustering in $\Refcon^\times$.}
\label{treadmill2}
\end{figure}

While stage (2') could be realized by HAT with at least a probability depending on $d$ and $n$ only, stage (1') could introduce a dependence on $V$ into $p$. Indeed, it might require that we specify the transport of an activated element over a distance of roughly the diameter of $V$. We will avoid this by adding one preliminary stage; in the resulting, three-stage procedure, stages (1') and (2') are essentially stages (2) and (3).

To specify the stages, we need two definitions.

\begin{definition}[Lined-up]\label{def lined up}
We say that $\CC \in \Conf_{d,n}^\times$ can be lined-up with separation $r > 0$ if $\dist(\CC^i, \CC^{\neq i}) \geq 2r$ and if $\CC^i$ is a connected subset of $\Z^d$, for every $1 \leq i \leq \#\CC$. We also say that the union $\cup_i \CC^i \in \Conf_{d,n}$ can be lined-up with separation $r$.
\end{definition}

Note that tuples in $\Conf_{d,n}^\times$ may consist of only one entry, in which case it suffices for this entry to be a connected set for it to be lined-up with separation $r > 0$. 

\begin{definition}[Lex]
We say that an element $x$ of a finite set $A \subset \Z^d$ is lex in $A$, denoted $\lex (A) = x$, if it is least among the elements of $A$ in the lexicographic order of $\Z^d$.
\end{definition}

Here is the three-stage procedure, which takes as input $V \in \Conf_{d,n}$ and an integer $a \geq 1$, which we will later require to be sufficiently large in terms of $d$ and $n$. Note that the output of each algorithm is an element of $\Conf_{d,n}^\times$, not an element of $\Conf_{d,n}$.
\begin{enumerate}
\item First, we will use Algorithm $\mathcal{A}_{\ref{alg1}}$ to construct a clustering $\CC = \mathcal{A}_{\ref{alg1}} (V,a)$ of a configuration which can be lined-up with separation $dn^2 a$. This algorithm is the most complicated of the three. In brief, the algorithm repeatedly attempts to create a non-isolated lex element of a cluster, so that it can be ``treadmilled'' in the $-e_1$ direction---along with a neighboring element---to form a new dimer cluster. 
\item In the second stage, we will apply Algorithm $\mathcal{A}_{\ref{alg2}}$ to ``line-up'' the elements of each cluster $\CC^i$. Specifically, in terms of the line segment
\begin{equation}\label{line segment}
L_k = \big\{ -j e_1: j \in \{0,1,\dots,k-1\} \big\}, \quad k \geq 1,
\end{equation}
we will rearrange the elements of $\CC^i$ into the set $\lex (\CC^i) + L_{|\CC^i|}$. Specifically, the resulting clustering will be
\[ \mathcal{A}_{\ref{alg2}} (\CC) = \left( \lex (\CC^i) + L_{|\CC^i|} \right)_{i=1}^{\# \CC}.\]
\item In the third stage, Algorithm $\mathcal{A}_{\ref{alg3}}$ will iteratively treadmill pairs of elements from each segment in the $-e_1$ direction for multiples of $a$ steps until only a dimer or a trimer of the original segment remains (Figure~\ref{treadmill2}). The resulting clustering $\mathcal{A} (V,a) = \mathcal{A}_{\ref{alg3}} (\mathcal{A}_{\ref{alg2}} (\CC), a)$ will be an $(a,1)$ DOT clustering.
\end{enumerate} 

In the next section, we prove some results which will aid our analysis of the algorithms. In particular, we prove a harmonic measure lower bound for lex elements. After preparing these inputs, in Section~\ref{sec form dot}, we will state and analyze the three algorithms to prove Theorem~\ref{thm form dot}.

\section{Inputs to the proof of Theorem~\ref{thm form dot}}\label{sec inputs}

\subsection{A geometric lemma}\label{sec geo}

To facilitate the application of a harmonic measure estimate in the next subsection, we need a geometric lemma and a consequence thereof. 
We state the following lemma with greater generality than is needed for this immediate need, so that we can reuse it in a later section. The statement requires the notion of the $\ast$-visible boundary of a set, which we first defined in \eqref{bdyvis}.

\begin{lemma}\label{kesten cor}
Let $d \geq 1$. Let $A$ be a finite subset of $\Z^d$ that contains the origin, and let $x$ and $y$ be distinct elements of $\bdyvis A$. There is a path $\Gamma$ from $x$ to $y$ in $A^\cc$ of length at most $\sqrt{d} \, \diam (A) + 3^{d+1} |A|$. Moreover, $\Gamma$ is contained in $\{z \in \Z^d: \|z\| \leq \diam (A) + \sqrt{d}\}$. 
\end{lemma}

\begin{proof}
Fix a finite set $A \subset \Z^d$ containing the origin, and fix two elements, $x$ and $y$, in $\bdyvis A$. Let $\{B_i\}_i$ be the collection of $\ast$-connected components of $A$. Because $A$ is finite, each $B_i$ is finite and, as each $B_i$ is also $\ast$-connected, each $\bdyvis B_i$ is connected in $\Z^d$ by Lemma~\ref{kesten result}.

If $\Gamma$ is the path of least length from $x$ to $y$, then this length, denoted $|\Gamma|$, satisfies 
\[
|\Gamma| \leq \sqrt{d} \diam (A) + 2d.
\]
Indeed, by definition, if $u \in \bdyvis A$, then there is $v \in A$ such that $\|u-v\| \leq \sqrt{d}$. By the triangle inequality, $\diam (\bdyvis A)$ is at most $\diam (A) + 2\sqrt{d}$, which implies the bound on $|\Gamma|$ because $|\Gamma| \leq \sqrt{d}\|x-y\|$. We will edit $\Gamma$ to obtain a potentially longer path which does not intersect $A$.

If $\Gamma$ does not intersect $A$, then we are done. Otherwise, let $i_1$ denote the label of the first $\ast$-connected component of $A$ intersected by $\Gamma$. Additionally, denote by $u$ and $v$ the first and last indices of $\Gamma$ which intersect $\bdyvis B_{i_1}$. Because $\bdyvis B_{i_1}$ is connected in $\Z^d$, there is a path $\Lambda$ in $\bdyvis B_{i_1}$ from $\Gamma_u$ to $\Gamma_v$. We may therefore edit $\Gamma$ to form $\Gamma'$:
\[ \Gamma' = \big( \Gamma_1, \dots, \Gamma_{u-1}, \Lambda_1, \dots, \Lambda_{|\Lambda|}, \Gamma_{v+1}, \dots, \Gamma_{|\Gamma|} \big). \]

If $\Gamma'$ does not intersect $A$, then we are done, as $\Gamma'$ is contained in the union of $\Gamma$ and $\cup_i \bdyvis B_i$, and because $\cup_i \bdyvis B_i$ has at most $3^d |A|$ elements. Accordingly,
\begin{equation}\label{gam length}
|\Gamma'| \leq \sqrt{d} \diam (A) + 2d + 3^d |A| \leq \sqrt{d}\, \diam (A) + 3^{d+1} |A|.
\end{equation}
Otherwise, if $\Gamma'$ intersects another $\ast$-connected component $B_{i_2}$ of $A$, we can argue in an analogous fashion to obtain a path $\Gamma''$ which intersects neither $B_{i_1}$ nor $B_{i_2}$. Like $\Gamma'$, $\Gamma''$ is contained in the union of $\Gamma$ and $\cup_i \bdyvis B_i$ and so its length satisfies the same upper bound. By continuing inductively, we obtain a path from $x$ to $y$ with a length of at most the right-hand side of \eqref{gam length}.

The path is contained in the union of $\Gamma$ and $\cup_i \bdyvis B_i$, which is contained in $\{z \in \Z^d: \|z\| \leq \diam (A) + \sqrt{d}\}$ because $A$ contains the origin by assumption.
\end{proof}

A consequence of this result is a simple comparison of harmonic measure at two points.

\begin{lemma}\label{lem hm comp}
Let $d \geq 5$ and $n \geq 2$. There is a constant $c = c(d,n)$ such that, if $A \cup B \in \Conf_{d,n}$ such that $A$ is connected and $\dist (A,B) \geq 4^d n$, then, for any distinct $x, y \in A$ that are exposed in $A \cup B$,
\begin{equation}\label{hm comp}
\H_{A \cup B} (x) \geq c \H_{A \cup B} (y).
\end{equation}
\end{lemma}

\begin{proof}
Let $x, y$ be elements of $A$ which are exposed in $A \cup B$. If $u$ is any element of $\bdyvis A \cap \bdy \{x\}$ and $v$ is any element of $\bdyvis A \cap \bdy \{y\}$, then, by Lemma \ref{kesten cor}, there is a path $\Gamma$ from $u$ to $v$ in $A^\cc$ of length at most
\[ \sqrt{d} \, \diam (A) + 3^{d+1} |A| \leq 1.1 \cdot 3^{d+1} n \leq 4^d n.\]
The first inequality is due to the assumption that $A$ is connected, which implies that $\diam (A)$ is at most $n$, and the fact that $\sqrt{d} \leq 0.1 \cdot 3^{d+1}$ when $d \geq 5$. The second inequality holds because $d \geq 5$.

Because $B$ is a distance of at least $4^d n$ from $A$, $\Gamma$ must also lie outside of $B$. This implies that there is a constant $c = c(d,n)$ such that
\[ \Es_{A \cup B} (x) \geq c \Es_{A \cup B} (y).\]
Dividing by the capacity of $A \cup B$ gives \eqref{hm comp}.
\end{proof}

\subsection{An estimate of harmonic measure for lex elements}\label{sec lex}

We now prove a harmonic measure lower bound for lex elements. 

\begin{lemma}\label{lex hm}
Let $d \geq 5$ and $n \geq 2$. There are constants $r = r (d,n)$ and $c = c(d)$ such that, if $A \cup B \in \Conf_{d,n}$ satisfies $\dist (A,B) \geq r$ and if $x$ is lex in $A$, then
\begin{equation}\label{lex hm bd1}
\Es_{A \cup B} (x) \geq  c n^{-\frac{1}{d-2}-o_d(1)} 
\end{equation}
and, consequently,
\begin{equation}\label{lex hm bd2}
\H_{A \cup B} (x) \geq c n^{-\frac{d-1}{d-2}-o_d(1)}.
\end{equation}
For concreteness, the $o_d (1)$ quantities are never larger than $0.8$ when $d \geq 5$, and the two lower bounds can be replaced with $cn^{-1.2}$ and $c n^{-2.2}$. 
\end{lemma}

\begin{proof}
Suppose that $x$ is lex in $A$ and that there are positive integers $k$ and $\ell$ for which $\dist (A, B) \geq 2^k \ell \sqrt{d}$. We will bound below the probability that a random walk from $x$ escapes $A \cup B$ by bounding below the probability that it (1) takes $\ell$ steps in the $-e_1$ direction, then (2) exits a sequence of doubling cubes through their $-e_1$ directed faces until it is at least a distance of $2^{k-1} \ell$ from $A \cup B$, and (3) subsequently never returns to $A \cup B$.

First, a random walk from $x$ reaches $x-\ell e_1$ before returning to $A \cup B$ with a probability of at least $(2d)^{-\ell}$:
\begin{equation}\label{lex esc bd1}
\P_x (\tau_{x-\ell e_1} < \tau_{A \cup B}) \geq (2d)^{-\ell}.
\end{equation}
The random walk can do so, for example, by following $(x-e_1,x-2e_1,\dots, x - \ell e_1)$, which lies outside of $A \cup B$ because $x$ is lex in $A$ and because $B$ is a distance of at least $2^k \ell \sqrt{d} > \ell$ from $x$.

Second, for $u \in \Z^d$ and $\ell \in \Z_{\geq 1}$, define $Q (u,\ell)$ to be the open cube centered at $u$ and with side length $2\ell$, and denote the $-e_1$ directed face of its boundary by $F(u,\ell)$:
 \[
Q(u,\ell) = \left\{z \in \Z^d: \max_{1 \leq i \leq d} |z_i - u_i| < \ell \right\} \quad \text{and} \quad F(u,\ell) = \left\{z \in \bdy Q (u,\ell): z_1 = \min_{v \in \bdy Q (u,\ell)} v_1 \right\}.
 \]
We inductively define cube centers $X_j$, cubes $Q_j$, and faces $F_j$ according to $X_1 = x-\ell e_1$, 
\[
Q_j = Q(X_j,2^{j-1}\ell), \quad F_j = F(X_j,2^{j-1} \ell), \quad \text{and} \quad X_{j+1} = S_{\tau_{\bdy Q_j}}, \quad j \geq 1.
\]
When $\cap_{i=2}^j \{X_i \in F_{i-1}\}$ occurs, $X_j$ satisfies
\[
\dist (A, X_j) \geq 2^{j-1} \ell \quad \text{and} \quad \dist(B, X_j) \geq \dist (A,B) - 2^{j-1} \ell \sqrt{d}, \quad j \geq 1.
\]
Since the distance between $A$ and $B$ is at least $2^k \ell \sqrt{d}$, this bound implies that
\[
\dist (A \cup B, X_k) \geq 2^{k-1} \ell.
\]
Hence, if $C = \{z \in \Z^d: \dist(z, A \cup B) < 2^{k-1}\ell\}$, then
\begin{equation}\label{lex esc bd2}
\P_{X_1} (\tau_{\bdy C} < \tau_{A \cup B}) \geq \P_{X_1} (\cap_{j=2}^k \{X_j \in F_{j-1}\}) = (2d)^{-(k-1)}.
\end{equation}
The equality follows the strong Markov property applied to $\tau_{\bdy F_{j}}$ for each $j \geq 1$ and the fact that $\P_{X_{j}} (X_{j+1} \in F_{j}) = (2d)^{-1}$, by symmetry.

Third, when $\{\tau_{\bdy C} < \tau_{A \cup B}\}$ occurs, the probability that the random walk never returns to $A \cup B$ is at least the minimum of $\P_z (\tau_{A \cup B} = \infty)$ over $z \in \bdy C$. The distance from $\bdy C$ to $A \cup B$ is at least $2^{k-1} \ell$. By Lemma~\ref{simple set hit}, there is a constant $\ell_0 = \ell_0 (d) \in \Z_{\geq 1}$ such that, if $\ell \geq \ell_0$, then  
\begin{equation}\label{lex esc bd3}
\P_x (\tau_{A \cup B} = \infty \mid \tau_{X_1} < \tau_{\bdy C} < \tau_{A \cup B}) \geq \min_{z \in \bdy C} \P_z \left( \tau_{A \cup B} = \infty \right) \geq 1 - n 2^{-k(d-2)}.
\end{equation}

Combining \eqref{lex esc bd1}, \eqref{lex esc bd2}, and \eqref{lex esc bd3}, and taking $\ell = \ell_0$, we find that
\begin{equation}\label{lex esc bd}
\Es_{A\cup B} (x) \gtrsim (2d)^{-k+1} (1 - n 2^{-k (d-2)}).
\end{equation}
If $n$ is at most $2^{d-3}$, then choosing $k=1$ in \eqref{lex esc bd} results in a constant lower bound, depending on $d$ only. Otherwise, if $n$ exceeds $2^{d-3}$, then we can take $k$ to be the integer part of $\log_2 ((2n)^{\frac{1}{d-2}})$, in which case \eqref{lex esc bd} gives
\begin{equation*}
\Es_{A \cup B} (x) \gtrsim n^{-\frac{1+ \log_2 (d)}{d-2}}.
\end{equation*}
Because $\capac (A \cup B)$ is at most $n G(o)^{-1}$, the preceding bound implies that
\begin{equation*}\label{cube bd6}
\H_{A \cup B} (x) \gtrsim n^{-\frac{d-1 + \log_2 (d)}{d-2}}.
\end{equation*}
We conclude the proof by setting $r = 2^{k_0}\ell_0 \sqrt{d}$, where $k_0$ is the larger of $1$ and the integer part of $\log_2 ((2n)^{\frac{1}{d-2}})$.
\end{proof}

We apply the preceding lemma to prove the following conditional hitting estimate.

\begin{lemma}\label{lex hm2}
Let $d \geq 5$. There are constants $r = r (n,d)$ and $c = c(d)$ such that if $x$ is lex in $A$, if $A \cup B \in \Conf_{d,n}$ such that $\dist (A,B) \geq r$, and if $B$ can be written as a disjoint union $B^1 \cup B^2$ where $|B^1| \leq 3$ and $\dist (B^1, B^2) \geq r$, then
\begin{equation}\label{eq lex hm2}
\P_x \big( S_{\tau_{A\cup B}} \in B^1 \bigm\vert \tau_{A \cup B} < \infty \big) \geq c n^{-\frac{d}{d-2}-o_d(1)} \diam (A \cup B)^{2-d}.
\end{equation}
For concreteness, the $o_d (1)$ quantity is smaller than $1.6$ when $d \geq 5$.
\end{lemma}

\begin{proof}
Let $A$, $B$, $B^1$, and $B^2$ satisfy the hypotheses for the $r$ in the statement of Lemma \ref{lex hm}. Additionally, denote by $F$ the set of points within a distance $r \diam (A \cup B)$ of $A \cup B$. Applying the strong Markov property to $\tau_{F^\cc}$, we write
\begin{multline}\label{lex cond hit}
\P_x \big( S_{\tau_{A\cup B}} \in B^1 \bigm\vert \tau_{A \cup B} < \infty \big)\\
 \geq \E_x \Big[ \P_{S_{\tau_{F^\cc}}} \big( S_{\tau_{A \cup B}} \in B^1 \bigm\vert \tau_{A \cup B} < \infty \big) \P_{S_{\tau_{F^\cc}}} \big( \tau_{A \cup B} < \infty \big); \tau_{F^\cc} < \tau_{A \cup B}\Big].
\end{multline}

A standard result (e.g., \cite[Theorem 2.1.3]{lawler2013intersections}) implies that for all sufficiently large $r$, if $y$ belongs to $F^\cc$, then
\begin{equation}\label{lex cond hit0}
\P_y \big( S_{\tau_{A \cup B}} \in B^1 \bigm\vert \tau_{A \cup B} < \infty \big) \geq \frac12 \H_{A \cup B} (B^1).
\end{equation}
Because $B^1 \cup (A \cup B^2)$ satisfies the hypotheses of Lemma \ref{lex hm},
\begin{equation}\label{lex cond hit1}
\H_{A \cup B} (B^1) \gtrsim c_1 n^{-\frac{d-1}{d-2}-o_d(1)}.
\end{equation} 
By Lemma~\ref{simple set hit}, for any $y \in F^\cc$, we have
\begin{equation}\label{lex cond hit2}
\P_y \big( \tau_{A \cup B} < \infty \big) \gtrsim \dist (y,A \cup B)^{2-d}.
\end{equation}
Lastly, by Lemma \ref{lex hm}, the probability of hitting $F^\cc$ before returning to $A \cup B$ satifies
\begin{equation}\label{lex cond hit3}
\P_x (\tau_{F^\cc} < \tau_{A \cup B} ) \geq \Es_{A\cup B} (x) \gtrsim n^{-\frac{1}{d-2} - o_d(1)}.
\end{equation}

Applying \eqref{lex cond hit0} through \eqref{lex cond hit3} to \eqref{lex cond hit}, we conclude that
\[ 
\P_x \big( S_{\tau_{A\cup B}} \in B^1 \bigm\vert \tau_{A \cup B} < \infty \big)
\gtrsim n^{-\frac{d}{d-2}-o_d(1)} \diam (A \cup B)^{2-d}.
\] Here, $o_d(1)$ can be taken to be $1.6$ when $d \geq 5$.
\end{proof}

\section{Proof of Theorem \ref{thm form dot}}\label{sec form dot}

In this section, we will analyze three algorithms which, when applied sequentially, dictate a sequence of HAT steps to form a configuration that has an $(a,1)$ DOT clustering in $\Refcon^\times$, from an arbitrary configuration and for any sufficiently large $a$. Each subsection will contain the statement of an algorithm and two results:
\begin{enumerate}
\item Informally, the first result will conclude that the algorithm does what it is intended to do. 
\item The second will provide bounds on the number of steps and probability with which HAT realizes the steps dictated by the algorithm.
\end{enumerate}
The final subsection will combine the bounds. We fix $d \geq 5$ and $n \geq 2$ throughout.

\subsection{Algorithm 1}\label{stage1}

Algorithm $\mathcal{A}_{\ref{alg1}}$ takes a configuration $U \in \Conf_{d,n}$ and an integer $a\geq 1$ as input and returns a tuple of (one or more) configurations $\CC$ that can be lined-up with separation $dn^2 a$. This tuple represents a partition of a configuration into parts that are connected, have at least two elements, and are separated by at least $2dn^2 a$. These parts can be thought of as ``clumps'' of elements---in particular, they may not be line segments and they may have more than three elements---hence, $\CC$ may not belong to $\Refcon^\times$.

The algorithm attempts to treadmill (in the sense of Figure~\ref{treadmill}) pairs of elements in the $-e_1$ direction. To start, if the lex element $\ell$ of $U$ is non-isolated, then the algorithm treadmills $\ell$ and one of its exposed neighbors along the ray $(\ell-e_1, \ell-2e_1, \dots)$, which is empty of $U$ because $\ell$ is lex in $U$. Once these elements are sufficiently far from the rest of $U$, they become $\CC^1$, and the algorithm attempts to repeat this process with $U' = U \setminus \CC^1$ in the place of $U$. However, the next pair of elements is treadmilled less far, so that it is distant from both $\CC^1$ and $U'$.

Suppose, for example, that the lex element of $U'$ is isolated, in which case the algorithm activates this element, in an attempt to form a non-isolated lex element. However, it cannot dictate where the activated element is transported. For example, it may join $\CC^1$ or it may return to $U'$, where it is no longer isolated. The algorithm then repeats this process with the resulting $U'$, to see if its lex element is non-isolated. Note that it takes at most $n$ repetitions for the algorithm to identify a non-isolated lex element of $U'$ or for all of the elements to belong to $\CC^1$. In the former case, the algorithm proceeds to treadmill a pair of elements, calls the result $\CC^2$, and continues with $U'' = U \setminus (\CC_1 \cup \CC_2)$. In the latter case, it returns $\CC^1$.

The key property of the sequence of HAT steps dictated by this algorithm is that they occur with a probability that is bounded below in a diameter-agnostic way. The next two algorithms form the ``clumps'' of elements in $\CC$ into line segments of between $2$ and $n$ elements, and then treadmill them apart, two or three elements at a time (Figure~\ref{treadmill2}).

Before stating Algorithm $\mathcal{A}_{\ref{alg1}}$, we give names to special elements that we reference in the algorithm. Suppose $U \in \Conf_d$ has a partition $\CC$ and let $x \in \Z^d$. If $x \in U$, then $\mu (U,x)$ denotes an arbitrary maximizer $y$ of $\P_x (S_{\tau_{U\setminus\{x\}}-1} = y \mid \tau_{U\setminus\{x\}} < \infty)$ over $\bdy (U{\setminus}\{x\})$. If $U \cap \bdy \{x\}$ is nonempty, then $\expo (U,x)$ denotes an arbitrary element $y \in \bdy \{x\}$ that is exposed in $U$. If $\{1 \leq i \leq \# \CC: \dist (\CC^i,x) \leq 1\}$ is nonempty, then $\near (\CC,x)$ denotes an arbitrary element thereof. 
Additionally, for a tuple of sets $\mathcal{D}$, we will use $\pi (\mathcal{D})$ to denote their union. 

\begin{algorithm}
\small
\NoCaptionOfAlgo
\DontPrintSemicolon
\SetAlgoNoLine
\LinesNumbered
\SetKwInOut{Input}{Input}
\SetKwInOut{Output}{Output}
    \Input{$U \in \Conf_{d,n}$ and $a \in \Z_{\geq 1}$}
    \Output{$\CC \in \Conf_{d,n}^\times$ that can be lined-up with separation $dn^2 a$}
 $\CC \leftarrow \emptyset, \quad i \leftarrow 1$ \tcp*{Initialize variables.}
    
 \While{$U$ is nonempty}{
 $\ell \leftarrow \lex (U)$\;
 $R \leftarrow \dist \big( \ell, U {\setminus} \{\ell \} \big), \quad  r \leftarrow 3d(n-i+1)^3 a$\;
  \tcc{Form a non-isolated lex element if need be.}
  \While{$R>1$ and $n>1$}{
    $x \leftarrow \mu (U \cup \pi (\CC), \ell)$ \tcp*{$\ell$ will be replaced by $x$.}
    \tcc{$x$ either neighbors $U$ \dots}
    \eIf{$x \in \bdy U$}{
    $U \leftarrow (U \cup \{x\}) {\setminus} \{\ell\}$
    }
    {\tcc{\dots or one of the existing clusters.}
    $U \leftarrow U {\setminus} \{\ell\}, \quad j \leftarrow \near (\CC,x)$\;
    $\CC \leftarrow \CC \cup^j \{x\}$\;
    }
    $\ell \leftarrow \lex (U)$ \tcp*{The lex element of $U$ may have changed.}
    $R \leftarrow \dist (\ell, U {\setminus} \{\ell\})$\;
  }
  \tcc{If the lex element is non-isolated, treadmill it.}
  \eIf{$R=1$}{
    $y \leftarrow \expo (U \cup \pi (\CC), \ell)$ \tcp*{$y$ is an exposed neighbor of $\ell$.}
    $U \leftarrow U {\setminus} \{\ell,y\}$ \tcp*{Remove the pair from $U$.}
    $\CC \leftarrow \CC \cup^i \big\{\ell - r e_1, \ell - (r-1) e_1\big\}$ \tcp*{Treadmill the pair $r$ steps.}
    $i \leftarrow i + 1$\tcp*{Prepare to form the next cluster.}
  }
  {\tcc{Otherwise, $U = \{\ell\}$; add it to an existing cluster.}
    $x \leftarrow \mu (U \cup \pi (\CC), \ell)$ \tcp*{$\ell$ will be replaced by $x$.}
    $U \leftarrow U {\setminus} \{\ell\},\quad j \leftarrow \near (\CC,x)$\;
    $\CC \leftarrow \CC \cup^j \{x\}$\;
  }
  
  }
 \Return $\CC$
 \;
  \caption{Algorithm $\mathcal{A}_1$}
\label{alg1}
\end{algorithm}

To realize stage (1) of the strategy of Section \ref{sec strat 11}, we must show that $\mathcal{A}_{\ref{alg1}}$ produces a configuration which can be lined-up (Definition \ref{def lined up}) and then show that HAT forms this configuration in a number of steps and with at least a probability which do not depend on the initial configuration. The following result addresses the former.

\begin{proposition}\label{alg1 works}
If $U \in \Conf_{d,n}$ and $a \in \Z_{\geq 1}$, then $\mathcal{A}_{\ref{alg1}} (U,a)$ can be lined-up with separation $dn^2a$.
\end{proposition}

\begin{proof}
Consider the tuple of configurations $\CC$ in line {\bf 27} of algorithm $\mathcal{A}_{\ref{alg1}}$. As every element of $U$ is eventually treadmilled with another element, or is transported next to another element, every entry of $\CC$ is connected and has two or more elements. To prove that $\CC$ can be lined-up with separation $r = dn^2 a$, we must additionally show that $\CC$ is $2r$ separated. 
Ignoring those elements that were assigned to clusters in lines {\bf 10} and {\bf 23}, due to line {\bf 19}, clusters $i < j$ are separated by at least
\[ 3d(n-i+1)^3 a - 3d(n-j+1)^3 a.\]
Because there are at most $[n/2]$ clusters, the preceding expression is at least
\[ 3d(n/2+1)^3 a - 3d(n/2)^3 a \geq 2d n^2 a + n.\]
At most $n$ elements are added to clusters by executing lines {\bf 10} and {\bf 23}. Because the clusters are connected, the preceding bound implies that the pairwise separation of clusters must be at least $2r = 2dn^2a$.
\end{proof}

We now verify that HAT realizes $\pi \big( \mathcal{A}_{\ref{alg1}} (U,a) \big)$ in a number of steps and with at least a probability which do not depend on $U$. We assume that $a$ is an integer for convenience.

\begin{proposition}\label{follow alg1}
Let $U \in \Conf_{d,n}$. There exists $\alpha_1 = \alpha_1 (d,n) \in \Z_{\geq 1}$ such that, if $a \geq \Z_{\geq \alpha_1}$, then there are $N_1 = N_1 (a,d,n) \in \N$ and $p_1 = p_1 (a,d,n) \in (0,1]$ such that 
\begin{equation}\label{eq follow alg1}
\PP_U \Big( U_{N_1} = \pi \big( \mathcal{A}_{\ref{alg1}} (U,a) \big) \Big) \geq p_1.
\end{equation} 
\end{proposition}

\begin{proof}
The proof takes the form of an analysis of Algorithm $\mathcal{A}_{\ref{alg1}}$. Denote by $u_{k}$ the configuration $U \cup \pi (\CC)$ after the $k$\textsuperscript{th} time $U \cup \pi (\CC)$ is changed (i.e., an element is moved) by the algorithm. Additionally, denote by $M$ the number of times the configuration changes before the outer {\bf while} loop terminates. 

To establish \eqref{eq follow alg1}, it suffices to show that there is a sequence of times $t_0 = 0 \leq t_1 < t_2 < \cdots < t_M \leq N_1$ such that $u_{0} = U$, $u_{M} = \pi \big( \mathcal{A}_{\ref{alg1}} (U,a) \big)$, and
\begin{equation}\label{alg bd1}
\PP_{u_0} \big( U_{t_1} = u_1, U_{t_2} = u_2, \dots, U_{t_M} = u_{M} \big) \geq p_1.
\end{equation}

We will argue that $M \leq n(n+1)$, that we can take $t_M = n(n+1)r_1$ for $r_1 = 3dn^3 a$, and that
\begin{equation}\label{alg bd22}
\PP_{u_{k-1}} \big( U_{t_k} = u_k \big) \geq q,
\end{equation}
for each $k \in \{1,\dots,M\}$, for a constant $q = q(a,d,n) > 0$. The Markov property then implies that \eqref{alg bd1} holds with $N_1 = n(n+1)r_1$ and $p_1 = q^{n(n+1)}$.

{\em Claim 1}. We claim that $M \leq n(n+1)$. Observe that the outer and inner {\bf while} loops starting on lines {\bf 2} and {\bf 5} each repeat at most $n$ times. Indeed, $U$ loses an element every time the outer loop repeats, which can happen no more than $n$ times. Concerning the inner loop, no non-isolated element is made to be isolated, while, each time line {\bf 6} is executed, the isolated element $\ell$ is replaced by an element $x$ which is non-isolated. This can happen at most $n$ times consecutively. Accordingly, $U \cup \pi (\CC)$ changes at most $n+1$ times every time the outer loop repeats, hence $M \leq n (n+1)$.

{\em Claim 2}. We now claim that we can take $t_M = n(n+1)r_1$. It suffices to argue that, each time $U \cup \pi (\CC)$ changes, at most $r_1$ steps of HAT are required to realize the change. The configuration $U \cup \pi (\CC)$ changes due to the execution of lines {\bf 8}, {\bf 10} and {\bf 11}, {\bf 18} and {\bf 19}, or {\bf 23} and {\bf 24}. In all but one case---that of lines {\bf 18} and {\bf 19}---the transition requires only one HAT step. For lines {\bf 18} and {\bf 19}, at most $r_1$ steps are needed. Because there are at most $n(n+1)$ changes, $t_M$ can be taken to be $n(n+1)r_1$.

{\em Claim 3}. We now verify \eqref{alg bd22} by considering each way $U \cup \pi (\CC)$ can change and by bounding below the probability that it is realized by HAT. Assume $U \cup \pi (\CC)$ has changed $k-1$ times so far.
\begin{itemize}
\item Lines {\bf 8}, {\bf 10} and {\bf 11}, or {\bf 23} and {\bf 24}: Activation at $\ell$ and transport to $x$. Assume that $a$ is sufficiently large in $d$ and $n$ to exceed the constant $r$ in the statement of Lemma \ref{lex hm}. Then, since $\ell$ is the lex element of $U$ and since $\dist (U, \pi (\CC)) \geq a$, we can apply Lemma \ref{lex hm} with $A = U$ and $B = \pi (\CC)$ to find that
\begin{equation}\label{h lw bd} 
\H_{U \cup \pi (\CC)}(\ell) \geq h,
\end{equation}
for a positive number $h = h(d,n)$. By the definition of $\mu$, $x$ is the most likely destination of an element activated at $\ell$. Transport from $\ell$ occurs to at most $2dn$ sites and so, by the pigeonhole principle, the element from $\ell$ is transported to $x$ with a probability of at least $(2dn)^{-1}$. Together, these bounds imply that
\begin{equation}\label{claim3 bd1}
\PP_{u_{k-1}} \big( U_{t_k} = u_k \big) \geq h (2dn)^{-1}.
\end{equation}
\item Lines {\bf 18} and {\bf 19}: Treadmilling of $\{\ell,y\}$. In the first step, we activate at $y$ and transport to $\ell-e_1$. While $y$ is not lex in $U$, by the definition of $\expo$, it is an exposed neighbor of $\ell$. Because $\dist (U, \pi (\CC)) \geq a$, if $a$ is at least $4^d n$, we can apply Lemma \ref{lem hm comp} with $A = U$ and $B = \pi (\CC)$ to conclude that $\H_{U \cup \pi (\CC)} (y)$ is at least $c_1 h$ from \eqref{h lw bd}, for a positive number $c_1 = c_1 (d,n)$. Additionally, Lemma \ref{kesten cor} implies that an element activated at $y$ is transported to $\ell-e_1$ with a probability of at least $c_2 = c_2 (d,n)$. Consequently, denoting
\[
v_1 = \big( u_{k-1} \cup \{\ell-e_1,\ell\} \big) {\setminus} \{\ell,y\},
\]
we have
\begin{equation}\label{claim3 bd2}
\PP_{u_{k-1}} (U_1 = v_1) \geq c_1 c_2 h.
\end{equation}

Now, consider the configuration $v_m$ resulting from starting at $u_{k-1}$ and treadmilling $\{\ell,y\}$ a total of $m \geq 2$ steps in the $-e_1$ direction:
\[ v_m = \big( u_{k-1} \cup \{\ell-me_1, \ell - (m-1)e_1\} \big) {\setminus} \{\ell,y\}.\]
To obtain $v_{m+1}$, we activate at $\ell-(m-1)e_1$ and transport to $\ell-(m+1)e_1$. By the same reasoning as before,
\begin{equation}\label{claim3 bd3}
\PP_{v_m} \big( U_1 = v_{m+1} \big) \geq c_1c_2 h.
\end{equation}
By \eqref{claim3 bd2}, \eqref{claim3 bd3}, and the Markov property,
\begin{equation}\label{claim3 bd4}
\PP_{u_{k-1}} \big( U_{t_k} = u_k \big) \geq \big(c_1 c_2 h \big)^{r_1}.
\end{equation}
\end{itemize}

The bounds \eqref{claim3 bd1} and \eqref{claim3 bd4} show that, whenever $U \cup \pi (\CC)$ changes, the change can be realized by HAT (in one or more steps) with a probability of at least
\[ q = \min \big\{ h (2dn)^{-1}, (c_1c_2h)^{r_1} \big\}. \]
This proves \eqref{alg bd22}. We complete the proof by combining claims 1--3. Note that, to apply Lemma~\ref{lem hm comp} and Lemma~\ref{lex hm}, we assumed that $a$ was at least $\alpha_1 \in \Z_{\geq 1}$, where $\alpha_1$ is an integer which is at least $4^d n$ and the constant $r = r(d,n)$ from Lemma~\ref{lex hm}.
\end{proof}

\subsection{Algorithm 2}\label{sec rearrange}

Algorithm $\mathcal{A}_{\ref{alg2}}$ takes as input a tuple of configurations $\CC \in \Conf_{d,n}^\times$ that can be lined-up with separation $dn^2 a$ for an integer $a \geq 1$. In other words, the entries of $\CC$ are connected sets with at least two elements and a separation of at least $2dn^2 a$. The algorithm dictates a sequence of HAT steps that form the tuple
\begin{equation}\label{lined up tuple}
\mathcal{L} (C) = \big( \lex (\CC^i) + L_{|\CC^i|} \big)_{i=1}^{\# \CC} \in \Conf_{d,n}^\times,
\end{equation}
where $L_k$ denotes the line segment of length $k$ from the origin to $-(k-1)e_1$ \eqref{line segment}. 

For each $\CC^i$, the algorithm activates an exposed element of $\CC^i$ and transports it to $\lex (\CC^i) - e_1$. Then, among the exposed elements that have not yet been activated, the algorithm selects one (we arbitrarily choose the lex one) and transports it to $\lex (\CC_i) - 2e_1$, and so on, until every element of $\CC^i$ except $\lex (\CC^i)$ has been activated. This completes Stage 2 of the strategy of Section \ref{sec strat 11}.\\

{\centering
\begin{minipage}{\linewidth}
\begin{algorithm}[H]
\small
\NoCaptionOfAlgo
\DontPrintSemicolon
\SetAlgoNoLine
\LinesNumbered
\SetKwInOut{Input}{Input}
\SetKwInOut{Output}{Output}
    \Input{$\CC \in \Conf_{d,n}^\times$ that can be lined-up with separation $dn^2 a$, for some $a \in \Z_{\geq 1}$.}
    \Output{$\mathcal{L}(C)$ \eqref{lined up tuple}.}

 \For{$i \in \{1,\dots,\# \CC\}$}{
 $\ell_i \leftarrow \lex (\CC^i)$ \tcp*{The segment will grow from $\ell_i$.}
    \For{$j \in \{1,\dots,|\CC^i|-1\}$}{
        $x_j \leftarrow \lex \big( \big\{ z \in \CC^i {\setminus} \{\ell_i + L_j\}: \H_{\pi (\CC)} (z) > 0 \big\} \big)$ \tcp*{$x_j$ is lex among exposed elements of $\CC^i$ which have not yet been added to the growing segment.}
        $y_j \leftarrow \ell_i - j e_1$ \tcp*{$y_j$ is the next addition to the segment.}
        $\CC \leftarrow (\CC \sd{i} \{x_j\}) \cup^i \{y_j\}$ \tcp*{Update the $i$\textsuperscript{th} cluster.}
    }
  }

 \Return $\CC$
  \caption{Algorithm $\mathcal{A}_{\ref{alg2}}$}
\label{alg2}
\end{algorithm}
\end{minipage}
}

\begin{proposition}\label{alg2 works}
Let $a \in \Z_{\geq 1}$. If $\CC \in \Conf_{d,n}^\times$ can be lined-up with separation $dn^2 a$, then $\mathcal{A}_{\ref{alg2}} (\CC) = \mathcal{L} (\CC)$.
\end{proposition}

\begin{proof}
The only way that the algorithm could fail to produce $\mathcal{L} (\CC)$ is if, for some outer {\bf for} loop $i$ and inner {\bf for} loop $j$, the assignment in line {\bf 5} is impossible. This would mean that no element of $\mathcal{D} = \CC^i {\setminus} \{\ell_i + L_j\}$ was exposed in $\pi (\CC)$. While there must be an element of $\mathcal{D}$ which is exposed in $\CC^i$, the elements of $\CC^{\neq i}$ could, in principle, separate $\mathcal{D}$ from $\infty$. In fact, as we argue now, this cannot occur because the clusters remain far enough apart while the algorithm runs.

Each $\CC^i$ remains connected while the algorithm runs, so there is a ball $B_i$ of radius $n$ which contains $\CC^i$. The $B_i$ are finite and $\ast$-connected, so each $\ast$-visible boundary $\bdyvis B_i$ is connected by Lemma~\ref{kesten result}. Moreover, each $\bdyvis B_i$ is disjoint from $\cup_j B_j$ because $\dist (B_i, B_j)$ exceeds $\sqrt{d}$. This lower bound holds because the clusters are initially $2dn^2 a$ separated and the separation decreases by at most one with each of the $n$ loops of the algorithm, hence
\[
\dist (B^i, B^j) \geq \dist (\CC^i, \CC^j) - \diam (B_i) - \diam (B_j) - n \geq 2dn^2 a - 5n > \sqrt{d}.
\]

The rest of the argument, which constructs an infinite path from $\mathcal{D}$ which otherwise avoids $\CC$, is identical to the corresponding step in the proof of Proposition~\ref{help for at pairs}. We conclude that some element of $\mathcal{D}$ is exposed in $\pi (\CC)$, which completes the proof.
\end{proof}

\begin{proposition}\label{rearrange small}
There exists $\alpha_2 = \alpha_2 (d,n) \in \Z_{\geq 1}$ such that, if $a \in \Z_{\geq \alpha_2}$, then there is $p_2 = p_2 (d,n) \in (0,1]$ such that, if $\CC \in \Conf_{d,n}^\times$ can be lined-up with a separation of $dn^2 a$, then  
\begin{equation}\label{eq rearrange small}
\PP_{\pi (\CC)} \Big( U_n = \pi \big( \mathcal{A}_{\ref{alg2}} (\CC) \big) \Big) \geq p_2.
\end{equation}
\end{proposition}

\begin{proof}
Given a tuple $\CC \in \Conf_{d,n}^\times$ that satisfies the hypotheses, algorithm $\mathcal{A}_{\ref{alg2}}$ specifies for each cluster $i$ a sequence of $|\CC^i|-1$ pairs $(x_j,y_j)$, where $x_j$ is the site of activation and $y_j$ is the site to which transport occurs, to rearrange $\CC^i$ into $\lex (\CC^i) + L_{|\CC^i|}$. We note that no pair will result in a decrease in cluster separation of more than one, or an increase in cluster diameter of more than one. Because the clusters are initially $2 dn^2 a$ separated, the clusters will remain $2dn^2 a - n \geq a$ separated throughout.

Let $\alpha_2$ be an integer at least as large as $4^dn$ and the constant $r$ from Lemma~\ref{lex hm}. Accordingly, if $a \geq \Z_{\geq \alpha_2}$, then the combination of Lemma \ref{lem hm comp} and Lemma \ref{lex hm} implies that there is a constant $h = h (d,n)$ such that each $x_j$ can be activated with a probability of at least $h$. Moreover, Lemma~\ref{kesten cor} implies that there is a positive number $c = c (d,n)$ such that an element from $x_j$ can be transported to $y_j$ with a probability of at least $c$. Consequently, denoting $\CC' = (\CC \sd{i} \{x_j\}) \cup^i \{y_j\}$, the transition in line {\bf 7} occurs with a probability of at least
\begin{equation}\label{a2 bd1}
\PP_{\pi (\CC)} \big( U_1 = \pi (\CC') \big) \geq c h.
\end{equation}
By \eqref{a2 bd1} and the Markov property, and the fact that there are at most $n$ pairs, we have
\[
\PP_{\pi (\CC)} \Big( U_n = \pi \big( \mathcal{A}_{\ref{alg2}} (\CC) \big) \Big) \geq (ch)^n.
\] Taking $p_2 = (ch)^n$ gives \eqref{eq rearrange small}.
\end{proof}

\subsection{Algorithm 3}\label{sec alg3}

At the beginning of Stage 3, the elements are neatly arranged into well separated line segments pointing in the $-e_1$ direction. In Stage 3, we iteratively treadmill pairs of elements in the $-e_1$ direction from each of the line segments, until only a dimer or trimer remains of the initial segment (Figure~\ref{treadmill2}). We will label each treadmilled pair as a new cluster. To reflect this in our notation, when $A \in \Conf_d$, we will write $\CC \cup^{\# \CC +1} A$ to mean $(\CC^1, \dots, \CC^{\# \CC}, A)$.

{\centering
\begin{minipage}{\linewidth}
\begin{algorithm}[H]
\small
\NoCaptionOfAlgo
\DontPrintSemicolon
\SetAlgoNoLine
\LinesNumbered
\SetKwInOut{Input}{Input}
\SetKwInOut{Output}{Output}
    \Input{$\DD = \mathcal{L}(\CC)$ \eqref{lined up tuple}, for $\CC \in \Conf_{d,n}^\times$ that can be lined-up with separation $n^2 a$, for some $a \in \Z_{\geq 1}$.}
    \Output{An $(a,2 (\log a)^{-1})$ DOT clustering $\DD$ in $\Refcon^\times$.}

$k \leftarrow 1$ \tcp*{$k$ counts the number of new clusters.}
 \For{$i \in \{1,\dots,\# \DD\}$}{
    \For{$j \in \{1,\dots,[|\DD^i|/2]-1\}$}{
        $\ell \leftarrow \lex (\DD^i),\quad r \leftarrow 2(n-j)a$\; 
        $\DD \leftarrow \big( \DD \sd{i} \{\ell,\ell+e_1\} \big) \cup^{\# \DD +k} \{\ell - re_1, \ell - (r-1)e_1\}$ 
        \tcp*{Treadmill the pair $r$ steps, labeling it as cluster $\# \DD+k$.}
        $k \leftarrow k+1$ \tcp*{Account for the creation of a new cluster.}
    }
  }
 \Return $\DD$
  \caption{Algorithm $\mathcal{A}_{\ref{alg3}}$}
\label{alg3}
\end{algorithm}
\end{minipage}
}

\begin{proposition}\label{alg3 works}
Let $a \in \Z_{\geq 1}$. If $\CC \in \Conf_{d,n}^\times$ can be lined-up with separation $n^2 a$, then $\mathcal{A}_{\ref{alg3}} (\mathcal{L}(\CC),a)$ is an $(a,2 (\log a)^{-1})$ DOT clustering in $\Refcon^\times$.
\end{proposition}

\begin{proof}
Denote by $\DD_k$ the clustering $\DD_0 = \mathcal{L}(\CC)$ once it has been changed by the algorithm for the $k$\textsuperscript{th} time (i.e., the $k$\textsuperscript{th} time line {\bf 5} is executed). Denote by $M$ the number of times algorithm $\mathcal{A}_{\ref{alg3}}$ changes $\DD_0$.

To prove that $\mathcal{A}_{\ref{alg3}} (\DD_0,a)$ is an $(a,2(\log a)^{-1})$ DOT clustering in $\Refcon^\times$, we will verify that $\DD_M$ satisfies the DOT condition \eqref{two or three elems}, that $\sep (\DD_M) \geq a$, and that each cluster of $\DD_M$ is a connected line segment parallel to $e_1$. These conditions imply that the $(a,b)$ separation conditions \eqref{abs rel sep} hold with this $a$ and $b = 2(\log a)^{-1}$.

Concerning \eqref{two or three elems} and the claim that each cluster of $\DD_M^i$ is a connected line segment parallel to $e_1$, we note that line {\bf 5} creates connected clusters of size two and, because it is executed $[ |\DD^i|/2 ] - 1$ times for cluster $i$, when the inner {\bf for} loop ends on line {\bf 7}, only two or three (connected) elements of the original cluster $\DD^i$ remain. Accordingly, every cluster of $\DD_M$ has two or three elements and is connected. It is clear from line {\bf 5} that, since the clusters of $\DD_0$ are line segments parallel to $e_1$, this is also true of $\DD_M$.

Concerning the separation of $\DD_M$, we observe that, for each $i \in \{1,\dots,\# \DD\}$, the separation of cluster $\DD_M^i$ is at least
\[\dist (\DD_M^i, \DD_M^{\neq i}) \geq \dist (\DD_0^i,\DD_0^{\neq i}) - 2(n-1)a \geq (n^2 - 2n + 2)a \geq a,\]
because the separation of $\DD_0$ is at least $n^2 a$ ($\CC$ can be lined-up with separation $n^2 a$) and because no element is moved a distance exceeding $2(n-1) a$ by the algorithm. The same is true of $\dist (\DD_M^i, \DD_M^j)$ for each $i$ and every $j$, and for clusters $i$ and $j$ resulting from the treadmilling of different clusters of $\DD_0$. Concerning the pairwise separation of clusters $i \neq j$ formed by treadmilling pairs from the same cluster of $\DD_0$, by line {\bf 5}, we have
\[ \dist (\DD_M^i, \DD_M^j) \geq 2(n-1)a - 2(n-2)a - 1 \geq a.\] We conclude that every cluster $i$ satisfies $\dist (\DD_M^i, \DD_M^{\neq i}) \geq a$, so $\sep (\DD_M) \geq a$.

The clusters of $\DD_M$ satisfy $\diam (\DD_M^i)$ because they are connected line segments of two or three elements. Since $\sep (\DD_M) \geq a$, $\DD_M$ satisfies $\diam (\DD_M^i) \leq b \log \dist(\DD_M^i, \DD_M^{\neq i})$ with $b = 2 (\log a)^{-1}$ \eqref{abs rel sep}. We conclude that $\DD_M$ is an $(a,b)$ DOT clustering in $\Refcon^\times$.
\end{proof}

\begin{proposition}\label{break lines}
There exists $\alpha_3 = \alpha_3 (d,n) \in \Z_{\geq 1}$ such that, if $a \in \Z_{\geq \alpha_3}$, then there are $N_3 = N_3 (a,d,n) \in \N$ and $p_3 = p_3 (a,d,n) \in (0,1]$ such that, if $\CC \in \Conf_{d,n}^\times$ can be lined-up with separation $n^2 a$, then 
\begin{equation}\label{eq rearrange small2}
\PP_{U} \Big( U_{N_3} = \pi \big( \mathcal{A}_{\ref{alg3}} (\mathcal{L}(\CC),a) \big) \Big) \geq p_3.
\end{equation}
\end{proposition}

\begin{proof}
As in the proof of Proposition \ref{alg3 works}, denote by $\DD_k$ the clustering $\DD_0 = \mathcal{L} (\CC)$ once it has been changed by the algorithm for the $k$\textsuperscript{th} time (i.e., the $k$\textsuperscript{th} time line {\bf 5} is executed). Denote by $M$ the number of times algorithm $\mathcal{A}_{\ref{alg3}}$ changes $\DD_0$. Call $u_k = \pi (\DD_k)$.

To establish \eqref{eq rearrange small2}, it suffices to show that there is a sequence of times $t_0 = 0 \leq t_1 < t_2 < \cdots < t_M \leq N_3$ such that $u_{0} = U$, $u_{M} = \pi \big( \mathcal{A}_{\ref{alg3}} (\DD_0,a) \big)$, and
\begin{equation}\label{alg3 bd1}
\PP_{u_0} \big( U_{t_1} = u_1, U_{t_2} = u_2, \dots, U_{t_M} = u_{M} \big) \geq p_3.
\end{equation}

Consider outer {\bf for} loop $i$, inner {\bf for} loop $j$, and suppose that $\DD_0$ has been changed a total of $k-1$ times thus far. Let $r = 2(n-j)a$. We will first bound below the probability that HAT realizes $u_k$ as $U_{r}$ from $u_{k-1}$ (i.e., the transition reflected in line {\bf 5}). HAT can realize this transition by treadmilling the elements at $\ell = \lex (\DD_{k-1}^i)$ and $\ell+e_1$ to $\{\ell - re_1,\ell-(r-1)e_1\}$.

For example, in the first step, we activate at $\ell+e_1$ and transport to $\ell-e_1$. As observed in the proof of Proposition \ref{alg3 works}, $\DD_k$ is $a$ separated for every $0 \leq k \leq M$. Denote by $\alpha_3$ an integer at least as large as $4^d n$ and the constant $r$ from Lemma~\ref{lex hm}. If $a \in \Z_{\geq \alpha_3}$, then the hypotheses of Lemma \ref{lem hm comp} and Lemma \ref{lex hm} are satisfied with $A = \pi (\DD_k^i)$ and $B = \pi (\DD_k^{\neq i})$, and they together imply the existence of a positive lower bound $h = h(d,n)$ on $\H_{u_{k-1}} (\ell+e_1)$. It is clear that the element at $\ell+e_1$ can be transported to $\ell-e_1$ with a probability of at least $c = c(d)$, and so, denoting
\[
v_s = \big( u_{k-1} \cup \{\ell-s e_1,\ell - (s-1) e_1 \} \big) {\setminus} \{ \ell,\ell+e_1\},
\]
we have
\begin{equation*}
\PP_{u_{k-1}} ( U_1 = v_1 ) \geq ch.
\end{equation*}

We can simply repeat this argument with $\ell$ and $\ell-e_1$ in the place of $\ell+e_1$ and $\ell$, then $\ell-e_1$ and $\ell-2e_1$, and so on. With the choice $u_k = v_r$, the Markov property implies
\begin{equation}\label{alg3 bd3}
\PP_{u_{k-1}} (U_{r} = u_k) \geq (ch)^{r}.
\end{equation}
The same bound \eqref{alg3 bd3} holds for any $k \in \{1, \dots, M\}$, so, by another use of the Markov property and the fact that $M \leq n$, we find
\begin{equation}
\PP_{u_0} \big( U_{r} = u_1, U_{2r} = u_2, \dots, U_{Mr} = u_M \big) \geq (ch)^{rn}.
\end{equation}
This proves \eqref{alg3 bd1} with $N_3 = nr$ and $p_3 = (ch)^{rn}$.
\end{proof}

\subsection{Conclusion}\label{sec 14 conc}

We now use the results from the preceding subsections to prove the main result of this section.

\begin{proof}[Proof of Theorem \ref{thm form dot}]
Let $U \in \Conf_{d,n}$. It suffices to prove the result when $a$ is sufficiently large, because if $\CC \in \Clust_{a,1}^\bullet$ and $0 < a' \leq a$, then $\CC \in \Clust_{a',1}^\bullet$. With this in mind, denote by $\alpha = \alpha (d,n)$ the largest of the integers $\alpha_1$ through $\alpha_3$ in Propositions \ref{follow alg1}, \ref{rearrange small}, and \ref{break lines}, and assume that $a \in \Z_{\geq \alpha}$. By Proposition \ref{follow alg1}, there are positive numbers $N_1 = N_1 (a,d,n)$ and $p_1 = p_1 (a,d,n)$ such that
\begin{equation}\label{thm14 conc1}
\PP_U \Big( U_{N_1} = \pi \big(\mathcal{A}_{\ref{alg1}} (U, a) \big) \Big) \geq p_1.
\end{equation}
By Proposition \ref{alg1 works}, $\CC_1 = \mathcal{A}_{\ref{alg1}} (U, a)$ can be lined-up with separation $d n^2 a$. Consequently, by Proposition \ref{rearrange small}, there is a positive number $p_2 = p_2 (d,n)$ such that
\begin{equation}\label{thm14 conc2}
\PP_{\pi (\CC_1)} \Big( U_{n} = \pi \big( \mathcal{A}_{\ref{alg2}} (\CC_1) \big) \Big) \geq p_2.
\end{equation}
By Proposition \ref{alg2 works}, $\CC_2$ equals $\mathcal{A}_{\ref{alg2}} (\CC_1) = \mathcal{L} (\CC_1)$. Since $\CC_2$ is $n^2 a$ separated, by Proposition \ref{break lines}, there are positive  numbers $N_3 = N_3 (a,d,n)$ and $p_3 = p_3 (a,d,n)$ such that
\begin{equation}\label{thm14 conc3}
\PP_{\pi (\CC_2)} \Big( U_{N_3} = \pi \big( \mathcal{A}_{\ref{alg3}} (\CC_2,a) \big) \Big) \geq p_3.
\end{equation}
Denote $\CC_3 = \mathcal{A}_{\ref{alg3}} (\CC_2,a)$. By the Markov property and \eqref{thm14 conc1} through \eqref{thm14 conc3},
\begin{equation}\label{thm14 conc4}
\PP_U \big( U_{N_1+n+N_3} = \pi (\CC_3) \big) \geq p_1 p_2 p_3.
\end{equation}
By Proposition \ref{alg3 works}, $\CC_3 \in \Clust_{a,b}^\bullet \cap \Refcon^\times$ with $b = 2 (\log a)^{-1}$. Note that $a \geq 4^d n$ (the constant in Lemma~\ref{lem hm comp}), so $b < 1$. Because $\CC_3$ is $(a,b)$ separated for $b < 1$, it is also $(a,1)$ separated. Setting $N = N_1+n+N_3$ and $p = p_1 p_2 p_3$ concludes the proof.
\end{proof}

\section{Proof of Theorem \ref{irred}}\label{sec irred}

Fix $d \geq 5$ and $n \geq 2$. To prove the irreducibility of HAT on $\wtnoniso_{d,n}$, we show that HAT can form a line segment from any configuration and HAT can form any configuration from a line segment. This is the content of the next two propositions. We state them in terms of $L_k$, which denotes the line segment of $k$ elements $\{-(j-1) e_1: 1 \leq j \leq k\}$ \eqref{line segment}, and $\rad (A) = \sup \{ \|x\|: x \in A\}$, the radius of a finite set $A \subset \Z^d$.

\begin{proposition}[Set to line]\label{cor line}
Let $V \in \Conf_{d,n}$. There are $N_4 = N_4 (d,n) \in \N$ and $p_4 = p_4 (d,n,\diam(V)) \in (0,1]$ such that
\begin{equation}\label{eq cor line}
\PP_{V} \big( \wt{U}_{N_4} = \wt{L}_n \big) \geq p_4.
\end{equation}
\end{proposition}

\begin{proposition}[Line to set]\label{prop line to set}
Let $V \in \noniso_{d,n}$. There are $N_5 = N_5 (d,n,\rad (V))$ and $p_5 = p_5 (d,n,\rad (V)) \in (0,1]$ such that
\begin{equation}\label{line to set}
\PP_{L_n} \big( \wt U_{N_5} = \wt V \big) \geq p_5.
\end{equation}
\end{proposition}

Theorem~\ref{irred} is a simple consequence of the preceding propositions. 

\begin{proof}[Proof of Theorem \ref{irred}]
Let $\wt V$ and $\wt W$ belong to $\wtnoniso_{d,n}$. By Propositions \ref{cor line} and \ref{prop line to set}, there are finite numbers of steps $N$ and $N'$, and positive probabilities $p$ and $p'$ such that
\[ 
\PP_{V} \big( \wt{U}_{N} = \wt{L}_n \big) \geq p \quad \text{and} \quad \PP_{L_n} \big( \wt{U}_{N'} = \wt{W} \big) \geq p'.
\]
By the Markov property at time $N$, the preceding bounds imply 
\[ 
\PP_V \big( \wt{U}_{N+N'} = \wt{W} \big) \geq p p' > 0,
\] 
which implies that HAT is irreducible on $\wtnoniso_{d,n}$.
\end{proof}

Next, we prove Proposition~\ref{cor line}.

\begin{proof}[Proof of Proposition~\ref{cor line}]
Let $a$ be the smallest integer that is larger than the constants denoted $r = r(d,n)$ in Lemma \ref{lex hm} and Lemma \ref{lex hm2}. By Theorem \ref{thm form dot}, there is a positive integer $M = M(d,n)$ and a positive number $q = q (d,n)$, such that there exists $\CC \in \Clust_{a+n,1}^\bullet (U_M)$.

Let $\ell$ be the lex element of $U_M$, which we assume w.l.o.g.\@ belongs to $\CC^1$. Because $\CC$ is $a+n$ separated, we can activate any lex element of any cluster with a probability of at least $q_2 = q_2 (d,n) > 0$ by Lemma \ref{lex hm}. Then, by Lemma \ref{lex hm2}, we can transport to $\ell-e_1$ with a probability of at least $q_3 = q_3 (d,n, \diam (V))$. Reassigning the element at $\ell-e_1$ to cluster $\CC^1$, the resulting clusters are at least $a + n-1$ separated.

Because the resulting clusters are still $a$ separated, we can simply repeat this process, transporting an element to $\ell-2e_1$, and so on. Continuing in this fashion for a total of $n$ steps results in $U_{M+n} = \ell + L_n$. The preceding discussion and the Markov property imply
\[ \PP_V \big( \wt U_{M+n} = \wt L_n \big) \geq q_1 (q_2 q_3)^n.\]
Setting $N_4 = M+n$ and $p_4 = q_1 (q_2 q_3)^n$ gives \eqref{eq cor line}.
\end{proof}

We will prove Proposition~\ref{prop line to set} with an argument by induction. To facilitate the induction step, it is convenient to prove the following, more detailed claim.

\begin{proposition}\label{induc arg}
Let $n \geq 2$ and let $V \in \noniso_{d,n}$. In terms of $r = [\, \rad (V)\,]$, there are positive integers $M = 4^d n^2 r$ and $\ell = M/n$, and a sequence $( (x_i,y_i) )_{i=1}^M$ in $\Z^d \times \Z^d$ 
such that, setting
\[
W_0 = L_n \quad \text{and} \quad W_j = (W_{j-1} {\setminus} \{x_j\}) \cup \{y_j\}, \quad 1 \leq j \leq M,
\]
the following conclusions hold:
\begin{enumerate}
\item[(i)] $W_M = V$.
\item[(ii)] For each $1 \leq j \leq M$, $x_j$ is exposed in $W_{j-1}$.
\item[(iii)] For each $1 \leq j \leq M$, there is a path $\Gamma_j$ from $x_j$ to $y_j$, which lies outside of $W_{j-1} {\setminus} \{x_j\}$ but inside of $B (r+dn)$, and which has a length of at most $\ell$.
\end{enumerate}
\end{proposition}

Before proving the proposition, let us explain how Proposition~\ref{prop line to set} follows from it. 

\begin{proof}[Proof of Proposition~\ref{prop line to set}]
By conclusion (iii) of Proposition~\ref{induc arg}, for each $1 \leq j \leq M$, we have $W_{j-1} \subseteq B(r + d n)$. By this observation and conclusion (ii), the activation component $\H_{W_{j-1}} (x_j)$ of each transition is at least a positive number $q_1 = q_1 (d,n,r)$. Again, by (iii), there is a path $\Gamma_j$, with a length of at most $\ell$, which can realize the transport step from $x_j$ to $y_j$. Consequently, in terms of $\tau = \tau_{W_{j-1} {\setminus} \{x_j\}}$, the transport component $\P_{x_j} ( S_{\tau - 1} = y_j \mid \tau < \infty)$ of each transition is at least $q_2 = (2d)^{-\ell-1}$. By the Markov property and conclusion (i), the probability in \eqref{line to set} is at least the product of these components, over $M$ steps:
\[
\PP_{L_n} \big( \wt{U}_M = \wt{V} \big) \geq (q_1 q_2)^M.
\]
\end{proof}

Lastly, we prove Proposition~\ref{induc arg}.

\begin{proof}[Proof of Proposition \ref{induc arg}]
The proof is by induction on $n$. The base case of $n=2$ is trivial because $\wtnoniso_{d,2}$ has the same elements as the equivalence class $\wt{L}_2$. Now suppose the claim holds up to $n-1$ for $n \geq 3$.

There are two cases, which we phrase in terms of the ``exposed'' boundary of $V$:
\[ \bdye V = \{x \in V: \H_V (x) > 0\}.\]
\begin{enumerate}
    \item There is a non-isolated $x \in \bdye V$ such that $V{\setminus}\{x\} \in \noniso_{d,n-1}$.
    \item For every non-isolated $x \in \bdye V$, $V {\setminus} \{x\} \in \iso_{d,n-1}$.
\end{enumerate}

 {\em Case 1}. Peform the following steps. In what follows, denote $r = [\rad (V)]+1$.
 \begin{itemize}
 \item[Step 1:] {\em ``Treadmill'' a pair of elements in the $-e_1$ direction for $m = r + dn - 2$ steps.} Specifically, activate the element $e_1$ and transport it to $-e_1$, then activate the element at the origin and transport it to $-2e_1$, followed by activation at $-e_1$ and transport to $-3e_1$, and so on.  

 \item[Step 2:] {\em Isolate an element outside of $B(r+ dn-2)$.} At the end of Step 1, an element lies at $-m e_1$ and another at $-(m-1)e_1$. Activate the latter and transport it to the $e_1$, then activate the element at $(n-1)e_1$ and transport it to the origin. 

 \item[Step 3:] {\em Use the induction hypothesis to form $(V{\setminus}\{x\}) \cup \{-m e_1\}$, for a particular $x$.} By the end of Step 2, the configuration is $L_{n-1} \cup \{-m e_1\}$. 
We use the induction hypothesis to form $V {\setminus} \{x\}$ from the $L_{n-1}$ subset, where $x$ is a non-isolated element of $\bdye V$ such that $V {\setminus} \{x\} \in \noniso_{d,n-1}$.

The use of the induction hypothesis guarantees that there is a sequence of $4^d(n-1)^2 r$ HAT steps from $L_{n-1}$, which: (i) result in $(V{\setminus}\{x\}) \cup \{-m e_1\}$; (ii) have positive activation components; and (iii) have transport steps that are realized by random walk paths with lengths of at most $4^d (n-1) r$, which remain inside $B(r+d(n-1))$.

 \item[Step 4:] {\em Transport the element at $-m e_1$ to $x$.} We activate the element at $-m e_1$ and transport it to $x$, which is possible because $x$ is non-isolated and exposed in $V$. Because $V{\setminus}\{x\}$ lies in $B(r)$, Lemma \ref{kesten cor} implies that there is a path from $-m e_1$ to $x$ which avoids $V{\setminus}\{x\}$, has a length of at most $\ell = 4^d n r$, and lies in $B(r+dn-1)$.
 \end{itemize}

Note that Steps 1 and 2 require $r+dn$ HAT steps, the activation components of which are positive and the transport components of which can be realized by paths of length at most $r+dn \leq \ell$. Steps 3 and 4 require $4^d(n-1)^2 r + 1$ HAT steps, again with positive activation components and transport components realized by paths of length at most $\ell$. In total, at most $M = 4^d n^2 r$ HAT steps are needed and, since all paths lie in $B(r+dn-1)$, conclusions (i) through (iii) hold.

{\em Case 2}. Because we cannot remove a non-isolated element of $V$ without obtaining an isolated set---a set to which the induction hypothesis does not apply---we must instead use the induction hypothesis to form a set related to $V$. In fact, the first two steps are the same as in Case 1, so we begin with the configuration $L_{n-1} \cup \{-m e_1\}$ and specify the third and subsequent steps.

\begin{itemize}
\item[Step 3':] {\em Use the induction hypothesis.} Let $w$ and $y$ be the least and greatest elements of $V$ in the lexicographic order, and let $x$ be any non-isolated element of $\bdye V$. We use the induction hypothesis to form
\[
V' = (V {\setminus} \{x,y\}) \cup \{w-e_1\},
\]
which is possible because $V' \in \wtnoniso_{d,n-1}$.

 The result is a sequence of $4^d(n-1)^2 r$ HAT steps from $L_{n-1}$ which form $V' \cup \{-m e_1\}$ with positive activation components and transport components which are realized by random walk paths with the same properties as in Step 3.

\item[Step 4':] {\em Activate the element at $-m e_1$ and transport it to $w-2e_1$.} 

\item[Step 5':] {\em Treadmill the pair $\{w-e_1,w-2e_1\}$.} Since $w$ is the least element of $V$ in the lexicographic order, $w-e_1$ and $w-2e_1$ are the only elements which lie in $O_1 = \{ z \in \Z^d: z \cdot e_1 \leq w \cdot e_1\}$. Similarly, due to the choice of $y$, it is the only element which lies in $O_2 = \{ z \in \Z^d: z \cdot e_d \geq y \cdot e_d\}$. 

Consequently, it is possible to treadmill the pair $\{w-e_1,w-2e_1\}$ to:
\begin{itemize}
\item $B(r+3)^\cc$ without leaving $O_1$; then
\item $O_2$ without leaving $B(r+6) {\setminus} B(r + 3)$; and
\item $\{y,y+e_d\}$ without leaving $O_2$.
\end{itemize}
This requires at most $f_2 = 10 r$ HAT steps, each of which has a positive activation component and a transport component realized by a random walk path of length five.

\item[Step 6':] {\em Activate at $\{y+e_d\}$ and transport to $x$.} The configuration at the end of Step 5' is $(V {\setminus} \{x\}) \cup \{y+e_d\}$, so activating the element at $\{y+e_d\}$ and transporting it to $x$ (which is possible because $x$ is an exposed, non-isolated element of $\bdye V$), forms $V$.  
\end{itemize}

Recall that Steps 1 and 2 require $r+dn$ HAT steps, which can be realized by paths of length at most $r+dn \leq \ell$. Steps 3' and 4' require $4^d(n-1)^2r + 1$ HAT steps, realized by paths of length at most $\ell$. Steps 5' and 6' require $10r+1$ HAT steps, with paths satisfying the same length bound. All activation components are positive. At most $M$ HAT steps are needed in total and, since all paths lie in $B(r+dn-1)$, conclusions (i) through (iii) hold.
\end{proof}

\section*{Acknowledgements}

I thank Shirshendu Ganguly and Alan Hammond for helpful discussions. In particular, I am grateful to Shirshendu Ganguly for suggesting Theorem \ref{thm form dot} to me and to Alan Hammond for sharing with me his prediction that dimers and trimers would drive the transience of HAT in higher dimensions. I thank the two anonymous referees for their careful reading of this paper and valuable feedback.


\end{document}